\documentclass[11pt,preprint]{imsart}
\usepackage{amsmath,amsthm,verbatim,amssymb,amsfonts,amscd, graphicx}
\usepackage{amssymb,bm}
\usepackage{mathrsfs}
\usepackage{hyperref}
\hypersetup{
    colorlinks,
    citecolor=black,
    filecolor=black,
    linkcolor=blue,
    urlcolor=black
}
\usepackage{url}
\usepackage{verbatim}
\RequirePackage{natbib}
\usepackage{constants}
\usepackage{enumerate}
\newcommand{\rom}[1]{\uppercase\expandafter{\romannumeral #1\relax}}
\newcommand{\beas}{\begin{eqnarray*}}
\newcommand{\enas}{\end{eqnarray*}}
\newcommand{\bea}{\begin{eqnarray}}
\newcommand{\ena}{\end{eqnarray}}
\newcommand{\bms}{\begin{multline*}}
\newcommand{\ems}{\end{multline*}}
\newcommand{\bels}{\begin{align*}}
\newcommand{\enls}{\end{align*}}
\newcommand{\bel}{\begin{align}}
\newcommand{\enl}{\end{align}}

\newcommand{\ignore}[1]{}

\newtheorem{theorem}{Theorem}[section]
\newtheorem{corollary}{Corollary}[section]

\newtheorem{lemma}{Lemma}[section]
\newtheorem{definition}{Definition}[section]

\newtheorem{assumption}{Assumption}[section]

\def\blfootnote{\xdef\@thefnmark{}\@footnotetext}

\newcommand{\expect}[1]{\mathbb{E}{\l[#1\r]}}
\newcommand{\wt}[1]{\widetilde{#1}}
\newcommand{\mc}[1]{\mathcal{#1}}
\newcommand{\mf}[1]{\mathbf{#1}}

\newcommand{\dotp}[2]{\left\langle#1,#2\right\rangle}

\newcommand{\mb}{\mathbb}
\newcommand\argmin{\mathop{\mbox{argmin}}}

\newcommand{\sign}{\mathrm{sign}}
\def\r{\right}
\def\l{\left}

\begin{document}

\title{\Large Structured Recovery with Heavy-tailed Measurements: A Thresholding Procedure and Optimal Rates}
\runtitle{Recovery with Heavy-tailed Measurements}
\begin{aug}
\author{\fnms{Xiaohan} \snm{Wei}\thanksref{t1}}
  \thankstext{t1}{Department of Electrical Engineering, University of Southern California}
\runauthor{X. Wei}
\affiliation{University of Southern California}
\end{aug}

\begin{abstract}
This paper introduces a general regularized thresholded least-square procedure estimating a structured signal $\theta_*\in\mathbb{R}^d$ from the following observations:
\[
y_i = f(\dotp{\mathbf{x}_i}{\theta_*}, \xi_i),~i\in\{1,2,\cdots,N\},
\] 
with i.i.d. heavy-tailed measurements $\{(\mathbf{x}_i,y_i)\}_{i=1}^N$. A general framework analyzing the thresholding procedure is proposed, which boils down to computing three critical radiuses of the bounding balls of the estimator. Then, we demonstrate these critical radiuses can be tightly bounded in the following two scenarios: (1) The 
link function $f(\cdot)$ is linear, i.e.
$
y = \dotp{\mathbf{x}}{\theta_*} + \xi,
$
with 
 $\theta_*$ being a sparse vector and $\{\mathbf{x}_i\}_{i=1}^N$ being general heavy-tailed random measurements with bounded 
$(20+\epsilon)$-moments. (2) The function $f(\cdot)$ is arbitrary unknown (possibly discontinuous) and $\{\mathbf{x}_i\}_{i=1}^N$ are heavy-tailed elliptical random vectors with bounded $(4+\epsilon)$-moments. In both scenarios, we show under these rather minimal bounded moment assumptions, such a procedure and corresponding analysis lead to optimal sample and error bounds with high probability in terms of the structural properties of $\theta_*$. 
\end{abstract}
\maketitle

\section{Introduction}
In mathematical statistics, it is common to assume that data satisfy an underlying model along with a set of assumptions on this model -- for example, that the sequence of vector-valued observations is i.i.d. and has multivariate normal distribution. 
Since real-world data typically do not fit the model or satisfy the assumptions exactly (e.g., due to outliers and noise), reducing the number and strictness of the assumptions helps to reduce the gap between the ``mathematical'' world and the ``real'' world. The concept of robustness occupies a central role in understanding this gap. 
One of the viable ways to model noisy data and outliers is to assume that the observations are generated by a heavy-tailed distribution, and this is precisely the approach that we follow in this work.  


The goal of this paper is to propose and analyze robust estimators of a high-dimensional vector $\theta^*\in\mathbb{R}^d$ from the following model:
\begin{equation}\label{eq:single-index}
y = f(\dotp{\mathbf{x}}{\theta_*},\xi),
\end{equation}
where the measurement vector $(\mathbf{x},y)$ is heavy-tailed with only constant number of moments.
The function $f:\mathbb{R}^2\rightarrow\mathbb{R}$ is a link function which can be unknown, and $\xi$ is the real-valued noise independent of $\mathbf{x}$. Statistical estimation in the presence of outliers and heavy-tailed data has recently attracted the
attention of the research community, and the literature covers a wide range of topics. A comprehensive review is beyond the scope of this paper. Here we focus mainly on the works related to the single-index model \eqref{eq:single-index} and in particular its special case -- the sparse recovery problem.



\subsection{Sparse recovery}
 When $f(\cdot)$ is a linear function, i.e.
$
y = \dotp{\mathbf{x}}{\theta_*} + \xi,
$
and $\theta_*\in \mb R^d$ possesses a certain structure, the problem reduces down to the the classical sparse recovery. A typical method recovering $\theta_*$ from a sequence of i.i.d. copies of $(\mathbf{x},y)$, (i.e. $\l\{(\mathbf{x}_i,y_i)\r\}_{i=1}^N$)  is to solve the following regularized least-square optimization problem (LASSO):
\begin{equation}\label{eq:opt0}
\widehat{\theta} := \argmin_{\theta\in\mathbb{R}^d} \frac1N\sum_{i=1}^N\l(\dotp{\mathbf{x}_i}{\theta} - y_i \r)^2 + \lambda\Psi(\theta)
\end{equation}
where $\Psi:\mathbb{R}^d\rightarrow\mathbb{R}$ is a structure inducing norm function and $\lambda$ is a trade-off parameter.
Over the past two decades, extensive progress has been made regarding this problem under the assumption that the sensing vectors are isotropic subgaussian and the noise is also subgaussian, e.g. \citep{tibshirani1996regression, candes2006stable, candes2008restricted, bickel2009simultaneous, hastie2015statistical}. Formally, we have the following definition regarding the aforementioned properties of the measurements:
\begin{definition}\label{def:subgaussian}
A symmetric random vector $\mathbf{x}\in\mathbb{R}^d$ is isotropic if $\expect{\mathbf{x}\mathbf{x}^T} = \mathbf{I}_{d\times d}$. It is subgaussian if for any $\mathbf v\in\mathbb{S}^{d-1}$, 
$\expect{\l|\dotp{\mathbf v}{\mathbf{x}} \r|^p}^{1/p}\leq C\sqrt{p}\cdot\expect{\l| \dotp{\mathbf v}{\mathbf{x}} \r|^2}^{1/2},~\forall p\geq1$ for some absolute constant $C>0$.
\end{definition}

In the scenario where $\theta_*$ is a $s$-sparse vector and $\Psi(\cdot) = \|\cdot\|_1$, 
given the above assumption, proving the performance bound on \eqref{eq:opt0} involves demonstrating the fact that if 
\begin{equation}\label{eq:optimal-sample}
N\gtrsim s\log(d),
\end{equation}
then, the restricted isometric property (RIP) holds for the measurement matrix 
$$\mf \Gamma = \frac{1}{\sqrt{N}}[\mathbf{x}_1, \mathbf{x}_2,\cdots, \mathbf{x}_N]^T,$$ 
over all $s$-sparse vectors $\mathbf v\in\mathbb{R}^d$, i.e. there exists a constant $\delta\in(0,1)$, such that 
\begin{equation}\label{eq:rip}
(1-\delta)\|\mathbf v\|_2\leq \|\mathbf{\Gamma} \mathbf{v}\|_2 \leq(1+\delta)\|\mathbf v\|_2.
\end{equation}
After this, one can show that 
\begin{equation}\label{eq:optimal-error}
\l\| \widehat{\theta}_0 - \theta_* \r\|_2\lesssim \sqrt{\frac{s\log d}{N}}
\end{equation}
with very high probability.

As is mentioned in a few previous works, e.g. \citep{Fan-robust-estimation-2017, Fan-huber-regression-2017}, such an isotropic subgaussian assumption, although quite convenient in analysis, is unrealistic in many applications involving heavy-tailed data (e.g. the functional magnetic resonance imaging(fMRI) \cite{eklund2016cluster}).
On the other hand, the RIP condition is not true with the optimal sample rate \eqref{eq:optimal-sample} when the tail of $\dotp{\mathbf v}{\mathbf X}$ decays slower than subgaussian.
 This leads to the question: Can we still obtain the optimal sample and error rate as those of \eqref{eq:optimal-sample} and \eqref{eq:optimal-error} 
\textit{without isotropic subgaussian assumption}? 

A crucial step answering this question is made in the seminal work \citep{mendelson2014learning}, 
An important observation underlying this work is that, in a typical subgaussian scenario, only the lower bound of the RIP condition \eqref{eq:rip} is used in the proof of \eqref{eq:optimal-error}, and, in fact, the lower bound of \eqref{eq:rip} can be satisfied under much weaker assumptions than the upper bound. Lower bounding the quadratic form $\|\mathbf{\Gamma v}\|_2^2$ also appears in the earlier work \cite{oliveira2013lower}, where the author obtains a high probability lower bound on $\|\mathbf{\Gamma v}\|_2^2$ with weak moment assumptions on the matrix $\mathbf{\Gamma}$.
Following this idea,  \cite{mendelson2014learning} introduces the following ``small-ball'' condition for the random vector $\mathbf x\in\mathbb{R}^d$:
\begin{definition}
A random vector $\mathbf{x}$ is said to satisfy the \textbf{small-ball condition} over a set $\mathcal{H}\subseteq\mathbb{R}^d$ if for any $\mathbf v\in\mathcal{H}$, there exist positive constants $\delta$ and $Q$ so that
\[
\inf_{\mathbf{v}\in\mathcal{H}} Pr\l( \l| \dotp{\mathbf v}{\mathbf{x}} \r| \geq \delta\|\mathbf{v}\|_2 \r) \geq Q.
\]
\end{definition}
The small-ball assumption was first introduced in the seminal work \citep{koltchinskii2015bounding} to get rid of the strong tail assumption lower-bounding singular values of random matrices. Its power in regression problems was demonstrated in \citep{mendelson2014learning}. 
This assumption is much weaker than the subgaussian assumption and, in particular, it allows for heavy-tailed measurement vector $\mathbf x$ (see \cite{mendelson2014learning} for detailed discussions). Thus, under this small ball assumption with $\mathcal{H}$ being the set of all $s$-sparse vectors in $\mathbb{R}^d$, \citep{lecue2017sparse} shows that by assuming the condition that $\mathbf{x}$ has subgaussian property up to only $\log d$ moments, i.e.
$
\expect{\l|\dotp{\mathbf v}{\mathbf x} \r|^p}^{1/p}\leq C\sqrt{p}\cdot\expect{\l| \dotp{\mathbf v}{\mathbf x} \r|^2}^{1/2},~\forall 2\leq p \leq c_1\log d,
$
where $c_1>0$ is an absolute constant,
one can achieve the same sample and error rate \eqref{eq:optimal-sample} and \eqref{eq:optimal-error} with high probability. 

An immediate next question is: Can we obtain the optimal sample and error rate with moment assumption weaker than $\mathcal{O}(\log d)$? Recently, the works \citep{Fan-robust-estimation-2017} and \citep{Fan-huber-regression-2017} propose a new class of thresholded estimators for sparse recovery, based on the earlier work \citep{catoni2012challenging} on adaptive shrinkage for heavy-tailed mean estimation. While their methods are quite effective when dealing with the heavy-tailed noise $\{\xi_i\}_{i=1}^N$, the sample rate is in general suboptimal when it comes to heavy-tailed design vectors 
$\{\mathbf{x}_i\}_{i=1}^N$. More specifically, they show when the measurement vector $\mathbf{x}_i$ has only bounded $(4+\epsilon)$-moments, a form of thresholded LASSO estimator guarantees the optimal error rate \eqref{eq:optimal-error} with high probability, when the number of samples satisfies $N\geq s^2\log d$ and $\|\theta_*\|_1\leq R$ for some absolute constant $R>0$.

\subsection{Structured single-index model}
When $f(\cdot)$ is a general \textit{unknown} function (can be non-convex or even discontinuous), \eqref{eq:single-index} is often referred to as the \textit{single-index} model. Since $f(\dotp{\mathbf{x}_i}{\theta_*},\varepsilon_i) = f\l(a^{-1}\dotp{\mathbf{x}_i}{a\theta_*},\varepsilon_i\r)$ for any $a>0$, one can only hope to recover $\theta_*$ up to scaling and without loss of generality, we assume $\|\theta_*\|_2=1$.
The majority of the aforementioned works assume that the link function $f(\cdot)$ is linear, and their results cannot be applied directly to the case with unknown $f(\cdot)$.

However, when the measurement vectors $\mathbf{x}_i$'s are isotropic Gaussian,
a somewhat surprising result states that one can estimate $\theta_*$ directly up to scaling, avoiding any preliminary link function estimation step. More specifically, \citep{Brillinger1983A-generalized-l00} proves that 
$\eta\theta_\ast = \argmin_{\theta\in\mathbb{R}^d}\mb E\l( y - \dotp{\theta}{\mathbf{x}}\r)^2$, where $\eta =  \mathbb{E}\dotp{y\mathbf{x}}{\theta_\ast}$. 
The proof is also surprisingly simple which uses rotational invariance property of Gaussian vectors as follows:
\begin{align*}
\argmin_{\theta\in\mb R^d}\mb E\l( y - \dotp{\theta}{\mf{x}}\r)^2 =& \argmin_{\theta\in\mb R^d}\|\theta\|_2^2 - 2\mb E y\dotp{\mf x}{\theta} \\
=& \argmin_{\theta\in\mb R^d}\|\theta\|_2^2 -2 \mb E y\dotp{\mf x}{\theta_*}\dotp{\theta}{\theta_*} - 
2\mb E y\dotp{\mf x}{\theta_*^\perp}\dotp{\theta}{\theta_*^\perp}\\
=& \argmin_{\theta\in\mb R^d}\|\theta\|_2^2 -2 \mb E y\dotp{\mf x}{\theta_*}\dotp{\theta}{\theta_*} 
=\argmin_{\theta\in\mb R^d}\|\theta-\eta\theta_*\|_2^2,
\end{align*}
where we use $\theta_*^\perp$ to denote the vector in the $(\theta_*,\theta)$ plane perpendicular to $\theta_*$, and the third equality follows from the fact that  $\dotp{\mf x}{\theta_*^\perp}$ is a mean 0 Gaussian random variable independent of $\dotp{\mf x}{\theta_*}$.
Later, \citep{li1989regression} extends this result to the more general case of elliptically symmetric distributions, which includes the Gaussian distribution as a special case. In general, it is not always possible to recover $\theta_*$: see \citep{ai2014one} for an example in the case when $f(x)=sign(x)$.

More recently, the works \citep{plan2014high,plan2016generalized, yi2015optimal}
presented the non-asymptotic study for the
case of Gaussian measurements in the context of high-dimensional structured estimation. Basically, they show that when the measurement vectors $\l\{\mathbf{x}_i\r\}_{i=1}^N$ are Gaussian, the unknown nonlinearity can be treated as additional noise and one can recover $\theta_*$ up to scaling with the optimal sample and error rate by solving the LASSO problem \eqref{eq:opt0}. The work \citep{goldstein2016non} considers general non-Gaussian measurements with i.i.d. subgaussian entries and show that the performance of the estimator is further related to the Stein's measure of discrepancy between the distribution of the entries and Gaussian distribution. 
However, the key assumption of Gaussianity precludes situations where the measurements are heavy-tailed, and hence might be overly restrictive for some practical applications, such as high-dimensional noisy image recovery and face recognition problems \citep{face-recognition}.

To treat the heavy-tailed scenario, \citep{goldstein2016structured} considers the elliptically symmetric measurements $\l\{\mathbf{x}_i\r\}_{i=1}^N$, proposes an adaptively thresholded estimator of $\eta\theta_*$ and proves a tight non-asymptotic deviation bounds under the weak $(4+\epsilon)$-moments assumption on 
$\mathbf{x}_i$ and $y_i$. More specifically, suppose $\eta\theta_*$ lies in a compact set $\Theta$ and the measurements $\l\{\mathbf{x}_i\r\}_{i=1}^N$ are isotropic, then, define the estimator $\widehat \theta_N$ as the solution to the constrained optimization problem:
\begin{align*}
\widehat\theta_N:=\argmin\limits_{\theta\in \Theta}\|\theta\|_2^2  - \frac{2}{N}\sum_{i=1}^N\dotp{\widetilde y_i\widetilde{\mf x}_i}{\theta},
\end{align*}
where $\widetilde y_i$ and $\widetilde{\mf x}_i$ are properly truncated versions of $y_i$ and $\mathbf{x}_i$. They show that the proposed estimator enjoys the following tight performance bound for any $\beta\geq 2$ and $N\geq \beta^2\l( \omega(D(\Theta,\eta\theta_*)\cap\mathbb{S}^{d-1})+1 \r)^2$:

\begin{align}\label{eq:prev-bounds}
\mb P\left(\left\|\widehat{\theta}_N-\eta\theta_*\right\|_2\geq C_1\frac{(\omega(D(\Theta,\eta\theta_*)\cap\mathbb{S}^{d-1})+1)\beta}{\sqrt{N}}\right)\leq C_2e^{-\beta/2},
\end{align}
where $C_1$ is a dimension-free positive constant depending only on the moment bounds of $y_i$ and $\mathbf{x}_i$, $C_2$ is an absolute constant and $\omega(D(\Theta,\eta\theta_*)\cap\mathbb{S}^{d-1})$ is the Gaussian mean-width on the intersection of the descent cone of $\Theta$ at $\eta\theta_*$ and a unit sphere. Note that such a quantity measures the complexity of recovery $\theta_*$. For example, 
 the work \citep{chandrasekaran2012convex} shows that when taking 
$\Theta = \{\mathbf{x}\in\mathbb{R}^d:~\|\mathbf{x}\|_{1}\leq1\}$, i.e. the unit ball of $\|\cdot\|_1$, and $\theta_*$ is $s$-sparse, we have
$\omega(\mathbb{S}^{d-1}\cap D(\Theta,\theta_*))$ is on the order of $\sqrt{s\log(d/s)}$.

The problem with the above estimator is that it requires the full  knowledge of the covariance structure of $\mathbf{x}_i$, i.e. it is isotropic. It is not known how to obtain the optimal sample and error rate estimating $\eta\theta_*$ with only bounded moment assumption and \textit{without} the knowledge of the covariance structure.
It is also worth noting that \citep{yang2017stein} proposes a high-dimensional thresholded score function estimator, which allows one to take general  measurement vectors with i.i.d. entries and bounded $(4+\epsilon)$-moments, albeit at the cost of knowing the distribution function of $\mathbf{x}_i$.

\subsection{Our contributions}
This paper introduces a simple regularized thresholded procedure recovering a structured signal $\theta_*\in\mathbb{R}^d$, by feeding \eqref{eq:opt0} with an adaptively truncated version of $\l\{\l(\mathbf{x}_i,y_i\r)\r\}_{i=1}^N$. We propose a general analysis framework which boils down to computing three critical radiuses of bounding balls regarding the estimator.
Based on this framework, we show the following:
\begin{enumerate}
\item When the link function $f(\cdot)$ is linear, $\Psi(\cdot) = \|\cdot\|_1$, and $\theta_*$ is an $s$-sparse vector with $\|\theta_*\|_2\leq1$, one only requires finite $(20+\epsilon)$-moments on $\mathbf{x}_i$, $y_i$ and  finite $(5+\epsilon)$-moments on the noise $\xi$ in order to guarantee the optimal sample and error rate regarding the estimator. This improves upon the previous suboptimal sample rate of $N\gtrsim s^2\log(d)$ for bounded moment measurements obtained in \citep{Fan-robust-estimation-2017, Fan-huber-regression-2017}, removing the assumption that $\|\theta_*\|_1\leq R$ in aforementioned works, and at the same time relaxing the $c_1\log d$ moment requirement in \citep{lecue2016regularization-2, lecue2017sparse}
for sparse recovery with optimal rates.
\item When the link function $f(\cdot)$ is arbitrary unknown, $\mathbf{x}_i$ is elliptical symmetric, and the set of sub-differentials of $\Psi(\cdot)$ norm near $\theta_*$ is large, one can recover $\theta_*$ up to constant scaling, requiring only $(4+\epsilon)$ moments on $\mathbf{x}_i$ and $y_i$. The sample and error rates depend on the structural property of $\theta_*$ and is tight. In particular, we show our bounded delivers the
optimal sample and error rates in the sparse and low-rank recovery scenarios.
\end{enumerate}

It is also worth noting that our estimators require neither the knowledge of covariance matrix of $\mathbf{x}_i$ nor explicit form of distribution functions, thereby significantly relaxing the assumptions on prior information in previous robust recovery works (e.g. \citep{goldstein2016structured, yang2017stein}). 

The rest of the paper is organized as follows: In Section \ref{sec:general-framework}, we introduce the general thresholding procedure for heavy-tailed measurements and main performance bounds. In Section \ref{sec:prelim-analysis}, we introduce a unified framework analyzing the thresholded estimators and sketch the proofs of main results. The conclusion is given in Section \ref{sec:conclude} and we detail the proofs in appendices.

\section{Main Results on Thresholded Estimators}\label{sec:general-framework}
Our goal is to robustify the penalized least-square \eqref{eq:opt0} in the scenario of heavy-tailed measurements $\{(\mathbf{x}_i,y_i)\}_{i=1}^N$. Throughout the paper, we adopt the following assumption on the measurements:
\begin{assumption}\label{assumption:moment}
The samples $\{(\mathbf{x}_i,y_i)\}_{i=1}^N$ are i.i.d. copies of $(\mathbf{x},y)$ with $\expect{\mathbf{x}} = 0$, generated from the model \eqref{eq:single-index} such that for some absolute constant $q>0$, there exists absolute constants $\nu,\nu_q,\kappa>0$,
\begin{itemize}
\item Bounded kutosis: $\sup_{\mathbf{v}\in\mathbb{S}^{d-1}}\expect{|\dotp{\mathbf{x}}{\mathbf{v}}|^4}\leq \nu$.
\item Bounded moments: $\|y\|_{L_q}:=\expect{|y|^q}^{1/q}\leq \nu_q$ and  $\|x_i\|_{L_q}:=\expect{|x_i|^q}^{1/q}\leq \nu_q,~\forall i\in\{1,2,\cdots,d\}$.
\item Non-degeneracy: $\inf_{\mathbf{v}\in\mathbb{S}^{d-1}}\expect{|\dotp{\mathbf{x}}{\mathbf{v}}|^2}\geq\kappa$.
\end{itemize}
\end{assumption}

The values of $q$ in the above assumption are problem-specific. In the sparse recovery scenario with general measurements, we require $q>20$. For the single-index model with elliptical symmetric measurements, we only require $q>4$.

Next, we have the following basic definitions:

\begin{definition}[Gaussian mean width]
\label{gmw}
The Gaussian mean width of a set $T \subseteq \mathbb{R}^d$ is defined as
\[
\omega(T):=\expect{\sup_{t\in T}~\langle\mathbf{g},t\rangle},
\]
where $\mathbf{g}\sim\mathcal{N}(0,\mathbf{I}_{d\times d})$. 
\end{definition}

\begin{definition}[$\psi_q$-norm]
For $q \ge 1$, the $\psi_q$-norm of a random variable $X\in \mb R$ is given by
\[\|X\|_{\psi_q}=\sup_{p\geq1}p^{-\frac{1}{q}}(\expect{|X|^p})^{\frac1p}.\]
Specifically, the cases $q=1$ and $q=2$ are known as the sub-exponential and sub-Gaussian norms respectively.  
We will say that $X$ is sub-exponential if $\|X\|_{\psi_1}<\infty$, and $X$ is subgaussian if $\|X\|_{\psi_2}<\infty$.
\end{definition}

Let $\mathbf{\Sigma} := \expect{\mathbf{x}\mathbf{x}^T}$ be the covariance matrix of the measurement vector. 
A centered random vector $\mathbf{x}\in\mathbb{R}^d$ has elliptically symmetric (alternatively, elliptically contoured or just elliptical) distribution with parameters $\mathbf{\mathbf{\Sigma}}$ and $F_{\mu}$, denoted $\mathbf{x}\sim\mathcal{E}(0,~\mathbf{\mathbf{\Sigma}},~F_{\mu})$, 
if 
\begin{equation}
\label{elliptical-definition}
\mathbf{x}\stackrel{d}{=}\mu\mathbf{B}U,
\end{equation}
where $\stackrel{d}{=}$ denotes equality in distribution, $\mu$ is a scalar random variable with cumulative distribution function $F_{\mu}$, $\mathbf{B}$ is a fixed $d\times d$ matrix such that the covariance matrix $\mathbf{\mathbf{\Sigma}}=\mathbf{B}\mathbf{B}^T$, and $U$ is uniformly distributed over the unit sphere $\mathbb{S}^{d-1}$ and independent of $\mu$. 
Note that distribution $\mathcal{E}(0,~\mathbf{\mathbf{\Sigma}},~F_{\mu})$ is well defined, as if $\mathbf{B}_1\mathbf{B}_1^T=\mathbf{B}_2\mathbf{B}_2^T$, then there exists a unitary matrix $\mathbf{Q}$ such that $\mathbf{B}_1=\mathbf{B}_2\mathbf{Q}$, and $\mathbf{Q}U\stackrel{d}{=}U$. Along these same lines, we note that representation \eqref{elliptical-definition} is not unique, as one may replace the pair $(\mu,~\mathbf{B})$ with $\left(c\mu,~\frac{1}{c}\mathbf{B}\mathbf{Q}\right)$ for any constant $c>0$ and any orthogonal matrix $\mathbf{Q}$. 
To avoid such ambiguity, in the following we allow $\mathbf{B}$ to be any matrix satisfying $\mathbf{B}\mathbf{B}^T=\mathbf{\mathbf{\Sigma}}$, and noting that the covariance matrix of $U$ is a multiple of the identity.

An important special case of the family $\mathcal{E}(0,~\mathbf{\mathbf{\Sigma}},~F_{\mu})$
of elliptical distributions is the Gaussian distribution $\mathcal{N}(0,\mathbf{\mathbf{\Sigma}})$, where $\mu=\sqrt{z}$ with $z \stackrel{d}{=} \chi_d^2$, and the characteristic generator is $\psi(x)=e^{-x/2}$. Note that $\expect{\mu^2}$ is usually of order $d$.

Define a scaling constant 
\begin{equation}\label{eq:scaling-const}
\eta := \expect{y\dotp{\mathbf{x}}{\theta_*}}\left/\|\mathbf{\Sigma}^{1/2}\theta_*\|_2^2\right.
\end{equation}
We assume $\eta \neq 0$. 
Note that $\eta = 1$ when 
$f(\cdot)$ is a linear function, and the noise $\xi$ is independent of $\mathbf{x}$. In more general scenarios where $f(\cdot)$ is arbitrary, this assumption implies 
$ \expect{f(\dotp{\mathbf{x}}{\theta_*},\varepsilon)\dotp{\mathbf{x}}{\theta_*}}\neq0$. In particular, it precludes the case where $f:\mathbb{R}^2\rightarrow\mathbb{R}$ is symmetric on the first variable.

Based on these assumptions, our robust estimator involves generating the truncated measurements $\{(\widetilde{\mathbf{x}}_i,\widetilde{y}_i)\}_{i=1}^N$ from the samples $\{(\mathbf{x}_i,y_i)\}_{i=1}^N$ and solving the following regularized thresholded least-square:
\begin{equation}\label{eq:rls}
\widehat{\theta}_N := \argmin_{\theta\in\mathbb{R}^d} \frac1N\sum_{i=1}^N\l(\dotp{\widetilde{\mathbf{x}}_i}{\theta} - \widetilde{y}_i \r)^2 + \lambda\Psi(\theta),
\end{equation}
where the precise form of $\{(\widetilde{\mathbf{x}}_i,\widetilde{y}_i)\}_{i=1}^N$ will be problem-specific:
\begin{itemize}
\item
In the case of sparse recovery, we take
$\widetilde{\mathbf{x}}_i$ such that 
\begin{equation}\label{eq:trunc1}
\widetilde{x}_{ij} = \sign\l(x_{ij}\r)\l(|x_{ij}|\wedge\tau\r),~~\forall j\in\{1,2,\cdots,d\},
\end{equation}
and $\widetilde{y}_i =\sign(y_i)\l( |y_i|\wedge\tau\r)$, where $\tau = \l(N/\log \l(ed\r)\r)^{1/4}$. 
\item In the case of single-index model with elliptical symmetric measurements, we take
\begin{equation}\label{eq:trunc2}
\widetilde{\mathbf{x}}_i = \frac{\sqrt{d}\mathbf{x}_i}{\|\mathbf{x}_i\|_2}\cdot\l( \frac{\|\mathbf{x}_i\|_2}{\sqrt{d}}\wedge\tau \r)
\end{equation}
and $\widetilde{y}_i =\sign(y_i)\l( |y_i|\wedge\tau\r)$, where $\tau = N^{2/(q+4)}$.
\end{itemize}

For the rest of the paper, the notations $B_\Psi(\mf x, r)$, $B_2(\mf x, r)$ denote the ball of radius $r$ centered at $\mf x$ for $\Psi$-norm, 2-norm respectively, and $S_\Psi(\mf x, r)$, $S_2(\mf x, r)$ denote the sphere of radius $r$ centered at $\mf x$ for $\Psi$-norm, 2-norm respectively.

\subsection{New result on sparse recovery}
Recall that in the sparse recovery problem we have the measurements $\l\{(\mf x_i, y_i)\r\}_{i=1}^N$ are heavy-tailed satisfying
\[
y_i = \dotp{\mf x_i}{\theta_*}+ \xi_i,~\forall i\in\{1,2,\cdots,N\}.
\]
We assume that $\theta_*$ is an $s$-sparse vector such that 
$\|\theta_*\|_2\leq1$, and also the following holds.

\begin{assumption}\label{assume:sparse-recovery}
There exists some $q>20$ such that Assumption \ref{assumption:moment} holds and the noise $\xi_i$ satisfies $\|\xi_i\|_{L_{q'}}<\infty$ for some $q'>5$, where 
$\|\xi_i\|_{L_{q'}}= \expect{|\xi_i|^q}^{1/q}$.
\end{assumption}

Recall that the scaling constant in this scenario is $\eta=1$, the $\Psi$-norm is taken to be $\|\cdot\|_1$-norm and the estimator is
\[
\widehat{\theta}_N := \argmin_{\theta\in\mathbb{R}^d} \frac1N\sum_{i=1}^N\l(\dotp{\widetilde{\mathbf{x}}_i}{\theta} - \widetilde{y}_i \r)^2 + \lambda\|\theta\|_1,
\]
where $\l\{(\wt{\mf x}_i, \wt{y}_i)\r\}_{i=1}^N$ are given by \eqref{eq:trunc1}.
we have the following theorem on the performance of our proposed thresholded LASSO estimator:

\begin{theorem}\label{thm:sparse-recovery}
Let $\delta = \frac12\sqrt{\frac{\kappa}{2}}$ and $Q=\frac{\kappa^2}{8\nu}$.
Suppose  Assumption \ref{assume:sparse-recovery} holds, 
$N\geq C\l( \frac{s_0}{Q^2} + \frac{\nu+1}{\nu} \r)\beta^2\log(ed) + cs\log(ed)$ for some absolute constants $C,~c>1$, $\lambda = \overline C_0(\nu_q,\nu,\xi)  \frac{wu^2v+w\beta^{3/4}}{\delta^2Q}\sqrt{\frac{\log(ed)}{N}}$ , and $s_0 = \frac{c_0\sqrt{\nu}}{\delta^2Q}s\leq d$ for some absolute constant $c_0>0$. Then, with probability at least 
\begin{multline*}
1-c' \l(e^{-\beta}+e^{-v^2}
+ (u^{-q/4}+u^{-q'})(ed)^{-(c-1)}\r.\\
\l.+(eN)^{-\frac{q}{12}+1}(\log(eN))^{q/6}w^{-q/6} +  (eN)^{-\frac{q'}{4}+1}(\log(eN))^{q'/2}w^{-q'}\r),
\end{multline*}
for some absolute constant $c'>0$,
we have 
\begin{align*}
\|\widehat\theta_N - \theta_*\|_2\leq& \overline C_1(\nu_q,\nu,\xi) \frac{wu^2v+w\beta^{3/4}}{\delta^2Q}\sqrt{\frac{s\log(ed)}{N}}\\
\|\widehat\theta_N - \theta_*\|_1\leq& \overline C_2(\nu_q,\nu,\xi) \frac{wu^2v+w\beta^{3/4}}{\delta^2Q}s\sqrt{\frac{\log(ed)}{N}},
\end{align*}
for any $\beta,u,v,w>6$, where 
$\overline C_i(\nu_q,\nu,\xi):= C_i\l(\nu_q^3+\nu_q^{5/2}+\nu_q^{3/2} + \|\xi\|_{L_{q'}}(\nu_q+1) + \nu_q^2+\nu_q^4\r)^2,~i=0,1,2$, and $C_i$ are absolute constants.
\end{theorem}

\subsection{New result on single-index model}
Consider recovering $\theta_*\in\mb{R}^d$ from the non-linear observation $y_i = f(\dotp{\theta_*}{\mf x_i},\xi_i),~i=1,2,\cdots,N$, where $\mf{x}_i,~i=1,2,\cdots,N$ are i.i.d. elliptical symmetric random vectors,
 $\xi_i,~i=1,2,\cdots, N$ are i.i.d. noise independent of $\mf x_i$, and $f:\mb R^2\rightarrow \mb R$ is an arbitrary fixed unknown function such that 
$\expect{y_i \dotp{\theta_*}{\mf x_i}} \neq 0$. Without loss of generality, we assume that $\l\|\mathbf{\Sigma}^{1/2} \theta_* \r\|_2^2= \dotp{\mathbf{\Sigma}^{1/2} \theta_*}{\mathbf{\Sigma}^{1/2} \theta_*} = 1$, then, the scaling constant defined in \eqref{eq:scaling-const} is 
\begin{equation}\label{eq:scaling-const-2}
\eta = \expect{y_i \dotp{\theta_*}{\mf x_i}}.
\end{equation}

\begin{assumption}\label{assume:single-index}
There exists some $q=4(1+\epsilon)$ for some $\epsilon>0$ such that Assumption \ref{assumption:moment} holds.
\end{assumption}

Recall that our estimator in this scenario is \eqref{eq:rls}
with 
$\{(\widetilde{\mathbf{x}}_i,\widetilde{y}_i)\}_{i=1}^N$ is defined according to \eqref{eq:trunc2}.
When $\eta\theta_*$ is close to an $s$-sparse vector $\theta_0$, and $\Psi(\cdot) = \|\cdot\|_1$, we have the following theorem:
\begin{theorem}\label{thm:main-sparse-2}
Let $\delta = \frac12\sqrt{\frac{\kappa}{2}}$ and $Q=\frac{\kappa^2}{8\nu}$.
Suppose Assumption \ref{assume:single-index} holds and the vector $\eta\theta_*$ satisfies 
$$\|\eta\theta_*-\theta_0\|_1\leq c(\nu,\kappa,\nu_q)\frac{s\sqrt{\log(ed/s)}}{\sqrt{N}},$$
for some $\theta_0$ such that $\|\theta_0\|_0 = s$, then, under the condition that 
\[
N \geq c_0(\nu,\kappa,\nu_q)\l(\frac{\delta t + \beta\sqrt{s\log(ed/s)}}{\delta Q}\r)^2 + \frac{4}{Q^2}\frac{\expect{\mu^2}^2\lambda_{\max}(\mf\Sigma)}{d^2}.
\]
and $\lambda = c_1(\nu,\kappa,\nu_q)\beta\sqrt{\log(ed/s)/N}$, then,
\[
\l\| \widehat{\theta}_N - \eta\theta_* \r\|_2\leq C_0(\nu,\kappa,\nu_q)\frac{\beta\sqrt{s\log(ed/s)}}{\sqrt{N}},
~~
\l\| \widehat{\theta}_N - \eta\theta_* \r\|_1\leq C_1(\nu,\kappa,\nu_q)\frac{\beta s\sqrt{\log(ed/s)}}{\sqrt{N}}, \]
with probability at least $1-e^{-\beta}-e^{-t^2}$ for any $\beta,t>2$, where 
$C_0(\nu,\kappa,\nu_q),~C_1(\nu,\kappa,\nu_q), ~c(\nu,\kappa,\nu_q),~c_0(\nu,\kappa,\nu_q)$ and 
$c_1(\nu,\kappa,\nu_q)$ are all constants depending only on $\nu,\kappa,\nu_q$ in Assumption \ref{assumption:moment}.
\end{theorem}

When $\eta\theta_*$ is close to a rank $s$ matrix $\theta_0$, and $\Psi(\cdot) = \|\cdot\|_*$, the nuclear norm of a matrix, we have the following theorem:
\begin{theorem}\label{thm:low-rank}
Let $\delta = \frac12\sqrt{\frac{\kappa}{2}}$ and $Q=\frac{\kappa^2}{8\nu}$. 
Suppose Assumption \ref{assume:single-index} holds and the matrix $\eta\theta_*\in\mathbb{R}^{m\times n}$ satisfies 
$$\|\eta\theta_*-\theta_0\|_*\leq c(\nu,\kappa,\nu_q)\frac{s\sqrt{m+n}}{\sqrt{N}},$$
for some matrix $\theta_0$ such that
where $\text{rank}(\theta_0) = s$, then, under the condition that 
\[
N \geq c_0(\nu,\kappa,\nu_q)\l(\frac{\delta t + \beta\sqrt{s(m+n)}}{\delta Q}\r)^2+ \frac{4}{Q^2}\frac{\expect{\mu^2}^2\lambda_{\max}(\mf\Sigma)}{d^2},
\]
and $\lambda = c_1(\nu,\kappa,\nu_q)\frac{\beta\sqrt{m+n}}{\delta Q}$, where $C_1,C_2$ are some constants, 
\[
\l\| \widehat{\theta}_N - \eta\theta_* \r\|_2\leq C_0(\nu,\kappa,\nu_q)\frac{\beta\sqrt{s(m+n)}}{\delta Q},
~~
\l\| \widehat{\theta}_N - \eta\theta_* \r\|_*\leq C_1(\nu,\kappa,\nu_q)\frac{\beta s\sqrt{m+n}}{\delta Q},
\]
with probability at least $1-e^{-\beta}-e^{-t^2}$ for any $\beta,t>1$ where $C_0(\nu,\kappa,\nu_q),~C_1(\nu,\kappa,\nu_q), ~c(\nu,\kappa,\nu_q),~c_0(\nu,\kappa,\nu_q)$ and 
$c_1(\nu,\kappa,\nu_q)$ are all constants depending only on $\nu,\kappa,\nu_q$ in Assumption \ref{assumption:moment}.

\end{theorem}

\section{A Unified Preliminary Analysis}\label{sec:prelim-analysis}
We start with the usual optimality analysis of \eqref{eq:rls}. Since $\widehat{\theta}_N$ minimizes the right hand side of \eqref{eq:rls}, we have
\[
\frac1N\sum_{i=1}^N\l(\dotp{\widetilde{\mathbf{x}}_i}{\widehat{\theta}_N} - \widetilde{y}_i \r)^2 + \lambda\Psi\l(\widehat{\theta}_N\r)
\leq \frac1N\sum_{i=1}^N\l(\dotp{\widetilde{\mathbf{x}}_i}{\eta\theta_*} - \widetilde{y}_i \r)^2 + \lambda\Psi(\eta\theta_*)
\]
Simple algebraic manipulations give
\begin{equation}\label{eq:opt1}
\frac{1}{N}\sum_{i=1}^N\dotp{\widetilde{\mathbf{x}}_i}{\widehat{\theta}_N-\eta\theta_*}^2 - 
\frac{2}{N}\sum_{i=1}^N\dotp{\widetilde{\mathbf{x}}_i}{\widehat{\theta}_N-\eta\theta_*}
\l(\widetilde{y}_i - \dotp{\widetilde{x}_i}{\eta\theta_*}\r) + \lambda\l( \Psi\l(\widehat{\theta}_N\r) - \Psi\l(\eta\theta_*\r)\r)\leq0.
\end{equation}
To simplify the notations, for any $\mathbf{v} \in\mathbb{R}^d$, define 
\begin{align*}
\mathcal{Q}_{\mathbf{v}}(\mf x)&:= \dotp{\wt{\mathbf{x}}}{\mathbf{v}}^2 \\
\mc{M}_{\mf v}(\mf x)&:=\l(\wt{y}-\dotp{\wt{\mathbf{x}}}{\eta\theta_*}\r)\dotp{\wt{\mf x}}{\mf{v}} - \expect{\l(\wt{y}-\dotp{\wt{\mathbf{x}}}{\eta\theta_*}\r)\dotp{\wt{\mf x}}{\mf{v}}}  \\
\mc{V}_{\mf v}& := \expect{\l(\wt{y} - \dotp{\wt {\mf x}}{\eta\theta_*}\r)\dotp{\wt{\mf{x}}}{\mf{v}}} 
\end{align*}
In addition, for any Borel measurable function $G:\mathbb{R}^d\rightarrow\mb{R}$, $\mc{P}_NG:=\frac{1}{N}\sum_{i=1}^NG(\mf{x}_i)$. Let \begin{equation}\label{eq:L}
\mathcal{L}^{\lambda}_{\mathbf{v}}(\mf x) := \mathcal{Q}_{\mathbf{v}}(\mf x) - 2\mc{M}_{\mf v}(\mf x) - 2 \mc{V}_{\mf v} + \lambda\l( \Psi\l(\eta\theta_*+\mf{v}\r) - \Psi\l(\eta\theta_*\r) \r)
\end{equation}
Having defined these notations, the criterion \eqref{eq:opt1} simply implies $\mc{P}_N\mathcal{L}^{\lambda}_{\widehat{\theta}_N-\eta\theta_*}\leq0$. Our goal is then to show that for any $\theta\in\mb{R}^d$ such that $\|\theta-\eta\theta_*\|_2>r$, where $r>0$ is a certain bounding radius, then, 
$$\mc{P}_N\mathcal{L}^{\lambda}_{\theta-\eta\theta_*} =  
\mc{P}_N\mathcal{Q}_{\theta-\eta\theta_*} - 2\mc{P}_N\mc{M}_{\theta-\eta\theta_*} - 2 \mc{V}_{\theta-\eta\theta_*} + \lambda\l( \Psi\l(\theta\r) - \Psi\l(\eta\theta_*\r) \r)> 0. $$
The intuition why one would expect this to happen is as follows. Suppose $\Psi(\cdot)$ is not a smooth function near $\eta\theta_*$ and the set of sub-differentials of the norm function $\Psi(\cdot)$ near $\eta\theta_*$ (which we denote as $\partial\Psi(\eta\theta_*)$) is ``large'', then, the set of descent directions i.e. 
$D_{\Psi}(\eta\theta_*):= \l\{ \theta\in\mathbb{R}^d:~\Psi(\theta)\leq \Psi(\eta\theta_*) \r\}$ would be relatively small.\footnote{The descent cone and the cone of sub-differentials are dual to each other.} This implies 
\begin{itemize}
\item For $\theta\in\mathbb{R}^d$ not in the descent directions, $\Psi(\theta)> \Psi(\eta\theta_*)$, and for an appropriate choice of $\lambda$, the possibly negative linear terms $- 2\mc{P}_N\mc{M}_{\theta-\eta\theta_*} - 2 \mc{V}_{\theta-\eta\theta_*}$ would be dominated by $\Psi(\theta)- \Psi(\eta\theta_*)$.
\item For the set of $\theta\in\mathbb{R}^d$ in the descent directions, we would expect the quadratic term $\mc{P}_N\mathcal{Q}_{\theta-\eta\theta_*}$ to dominate the linear terms $- 2\mc{P}_N\mc{M}_{\theta-\eta\theta_*} - 2 \mc{V}_{\theta-\eta\theta_*}$. For sufficiently small set of descent directions and proper choices of random measurement vectors, $\mc{P}_N\mathcal{Q}_{\theta-\eta\theta_*}$ would be a non-degenerated quadratic form over the set of descent directions (i.e. 
$\mc{P}_N\mathcal{Q}_{\theta-\eta\theta_*}\geq c\|\theta-\eta\theta_*\|_2^2$ for some constant $c>0$), which dominates the linear terms $2\mc{P}_N\mc{M}_{\theta-\eta\theta_*}$ and $2 \mc{V}_{\theta-\eta\theta_*}$ for all $\theta$ sufficiently away from $\eta\theta_*$.
\end{itemize}

Following the idea of \citep{lecue2016regularization-2, lecue2016regularization}, which offers a promising alternative to the usual RIP analysis, we concretize these two intuitions by considering the intersection of an $L_2$-ball $B_{2}(\eta\theta_*,r)$ and a $\Psi$-ball $B_\Psi(\eta\theta_*,\rho)$, with a properly chosen $\rho>0$, and we aim to show that if $\theta$ is outside of $B_{2}(\eta\theta_*,r)\cap B_\Psi(\eta\theta_*,\rho)$ with appropriate choices of $r$ and $\rho$, then,  
$\mc{P}_N\mathcal{L}^{\lambda}_{\theta-\eta\theta_*}>0$. As is shown in Fig. \ref{fig:geometry},  having this intersection essentially divides the space outside of $B_{2}(\eta\theta_*,r)\cap B_\Psi(\eta\theta_*,\rho)$ into two types of regions: 1. The region containing the set of descent directions $D_{\Psi}(\eta\theta_*)$, where the quadratic term $\mc{P}_N\mathcal{Q}_{\theta-\eta\theta_*}$ is expected to take effect. 2. The region where $\Psi(\theta)>\Psi(\eta\theta_*)$, and the term $\lambda(\Psi(\theta) - \Psi(\eta\theta_*))$ is expected to take effect.


Let $\Lambda_Q,~\Lambda_M$ and $\Lambda_{\mc{V}}$ be three positive constants.
For chosen $\rho>0$ and $p_{\mathcal{Q}},p_{\mathcal{M}}\in(0,1)$, 
we define three critical radiuses:
\begin{align*}
r_{\mathcal{Q}} :=& \inf\l\{ r>0: Pr\l( \inf_{\theta\in S_{2}(\eta\theta_*,r)\cap B_\Psi(\eta\theta_*,\rho)}\mc{P}_N\mathcal{Q}_{\theta-\eta\theta_*}
\geq \Lambda_Qr^2 \r)\geq1-p_{\mathcal{Q}} \r\},\\
r_{\mc{V}}:=& \inf\l\{ r>0:\sup_{\theta\in B_{2}(\eta\theta_*,r)\cap B_\Psi(\eta\theta_*,\rho)}\l| \mc{V}_{\theta-\eta\theta_*} \r| \leq \Lambda_{\mc{V}}r^2 \r\},\\
r_{M}:= &\inf\l\{ r>0:~Pr\l( \sup_{\theta\in B_{2}(\eta\theta_*,r)\cap B_\Psi(\eta\theta_*,\rho)}\l|\mc{P}_N\mc{M}_{\theta-\eta\theta_*}  \r|  
\leq \Lambda_M r^2\r)\geq 1-p_{\mathcal{M}} \r\},
\end{align*}

We then set 
$$r(\rho) := \max\l\{ r_{\mathcal{Q}}, r_{\mathcal{M}}, r_{\mc{V}} \r\}.$$
Define the set of sub-differentials of the norm function $\Psi(\cdot)$ near $\eta\theta_*$ (i.e. within $\Psi$-radius of $\rho/16$) as
\begin{equation}\label{eq:sub-diff}
\Gamma_\Psi(\eta\theta_*,\rho) := \l\{ \mf{z}\in\mathbb{R}^d: \Psi(\mf u+\Delta \mf u) - \Psi(\mf u)\geq \dotp{\mf{z}}{\Delta\mf{u}} ,~~ \exists \mf u\in B_{\Psi}\l(\eta\theta_*,\frac{\rho}{16}\r),~\forall \Delta \mf u \in\mathbb{R}^d \r\}.
\end{equation}
Then, the set $\Gamma_\Psi(\eta\theta_*,\rho)$ being ``large'' is characterized by the following quantity:
\[
\Delta(\eta\theta_*,\rho) := \inf_{\theta\in B_{2}(\eta\theta_*,r)\cap S_\Psi(\eta\theta_*,\rho)}
~\sup_{\mf{z}\in\Gamma_\Psi(\eta\theta_*,\rho)}\dotp{\mf z}{\theta-\eta\theta_*}
\]
This key concept is first introduced in the works \citep{lecue2016regularization-2, lecue2016regularization}.
It characterizes the minimum amount of increase of the norm function $\Psi(\cdot)$ from $\Psi(\eta\theta_*)$ on the boundary of region II in Fig. \ref{fig:geometry}, and the set of sub-differentials $\Gamma_\Psi(\eta\theta_*,\rho)$ being ``large'' means for any $\theta\in B_{2}(\eta\theta_*,r)\cap S_\Psi(\eta\theta_*,\rho)$, there exists a vector in $\Gamma_\Psi(\eta\theta_*,\rho)$ which is close to the sub-differential of $\theta-\eta\theta_*$.

Our goal is to show that when $\theta\not\in B_{2}(\eta\theta_*,r(\rho))\cap B_\Psi(\eta\theta_*,\rho)$ and $\Delta(\eta\theta_*,\rho)$ is comparable to $\rho$, then, one has $\mc{P}_N\mathcal{L}^{\lambda}_{\theta-\eta\theta_*}>0$, as is shown in the following theorem.

\begin{theorem}\label{thm:master}
Suppose  there exists $\rho>0$ and $c_2\frac{r(\rho)^2}{\rho}\leq\lambda\leq c_1\frac{r(\rho)^2}{\rho}$ for some constant $c_1,c_2$, such that 
\begin{enumerate}
\item
$\Lambda_Q> 2(\Lambda_M+\Lambda_{\mc{V}})+c_1$ and $c_2\geq4(\Lambda_M+\Lambda_{\mc{V}})$. 
\item $\Delta(\eta\theta_*,\rho)\geq\frac34\rho$.
\end{enumerate}
Then, for any $\theta\not\in B_{2}(\eta\theta_*,r(\rho))\cap B_\Psi(\eta\theta_*,\rho)$, $\mc{P}_N\mathcal{L}^{\lambda}_{\theta-\eta\theta_*}>0$ with probability at least $1-p_{\mathcal{Q}}-p_{\mathcal{M}}$.
\end{theorem}

\subsection{Sparse recovery with heavy-tailed measurements}\label{sec:sparse-recovery}
We have the following bounds on the critical radiuses for the case of sparse recovery.
\begin{lemma}\label{lem:bound-on-radiuses}
Suppose $N\geq C_0\l( \frac{s_0}{Q^2} + \frac{\nu+1}{\nu} \r)\beta^2\log(ed)+ \frac{\nu}{Q}s_0\log(ed)+cs\log(ed)$ for some absolute constants $C,c>1$, $s_0 = \frac{c_0\sqrt{\nu}}{\delta^2Q}s\leq d$ for some absolute constant $c_0>0$ and Assumption \ref{assume:sparse-recovery} holds, then,
\begin{align*}
r_{\mathcal{Q}}\leq&\sqrt{\frac{2}{c_0s}}\rho,~~~~r_{\mathcal{V}}\leq 8\l(\nu_q^2+\nu_q^4\r)  \l( \frac{\rho}{\delta^2Q} \r)^{1/2}\l( \frac{\log(ed)}{N} \r)^{1/4},\\
r_{\mathcal{M}}\leq& C(\nu_q,\xi)\l( \frac{wu^2v+w\beta^{3/4}}{\delta^2Q}\sqrt{\frac{s\log(ed)}{N}} + \sqrt{\rho\frac{wu^2v+w\beta^{3/4}}{\delta^2Q}}\l(\frac{s\log(ed)}{N}\r)^{1/4} \r),
\end{align*}
when taking $p_{\mathcal{Q}} = c'e^{-\beta}$ for some absolute constant $c'>0$ and
\begin{multline*}
p_{\mathcal{M}} = 2e^{-\beta}+2e^{-v^2}
+c'\l((u^{-q/4}+u^{-q'})(ed)^{-(c-1)}\r.\\
\l.+(eN)^{-\frac{q}{12}+1}(\log(eN))^{q/6}w^{-q/6} +  (eN)^{-\frac{q'}{4}+1}(\log(eN))^{q'/2}w^{-q'}\r),
\end{multline*}
in the definitions of $r_{\mathcal{Q}}$ and $r_{\mathcal{M}}$ for $\beta,u,v,w>6$, where $C(\nu_q,\xi) := C'\l(\nu_q^3+\nu_q^{5/2}+\nu_q^{3/2} + \|\xi\|_{L_{q'}}(\nu_q+1)\r)$, for some absolute constant $C>0$.
\end{lemma}

The above lemma is the combination of Lemma \ref{lem:bound-Q},~\ref{lem:bound-M},~\ref{lem:bound-V} proved in the appendix.  The bound on  $r_{\mathcal{Q}}$ relies on a new truncated small-ball argument in conjunction with a bookkeeping VC argument. The bound on $r_{\mathcal{M}}$ relies on a new analysis on the truncated multiplier process (Lemma \ref{lem:bound-PM}) leveraging the fact that the bias is small if we truncate 
$(\mathbf{x}_i,y_i)$ at a high enough level (namely, at level $\tau = (N/\log(ed))^{1/4}$) with enough moments ($q>20$) assumed. 

Recall that the final radius bound  $r(\rho) := \max\l\{ r_{\mathcal{Q}}, r_{\mathcal{M}}, r_{\mc{V}} \r\}$. Thus, $r(\rho)$ is bounded above by the maximum of the bounds in Lemma \ref{lem:bound-on-radiuses}.
In view of Theorem \ref{thm:master}, we need to check if $\Delta(\eta\theta_*,\rho)\geq\frac34\rho$ holds. This is done via the following 
characterization of the set of sub-differentials whose proof is fairly standard and delayed to the appendix.
\begin{lemma}\label{lem:sparse-eq}
Suppose $\|\eta\theta_* - \theta_0\|_1\leq\rho/16$, where $\theta_0$ an $s$-sparse vector and $\rho\geq 8r(\rho)\sqrt{s}$, then, $\Delta(\eta\theta_*,\rho)\geq 3\rho/4$.
\end{lemma}
\begin{proof}[Proof of Theorem \ref{thm:sparse-recovery}]
By Lemma \ref{lem:bound-on-radiuses} we have, with probability at least $1-p_{\mathcal{Q}}-p_{\mathcal{M}}$,
\begin{multline*}
r(\rho)\leq  \sqrt{\frac{2}{c_0s}}\rho + C(\nu_q,\nu,\xi)\l( \frac{wu^2v+w\beta^{3/4}}{\delta^2Q}\sqrt{\frac{s\log(ed)}{N}} + \sqrt{\rho\frac{wu^2v+w\beta^{3/4}}{\delta^2Q}}\l(\frac{s\log(ed)}{N}\r)^{1/4} \r) \\
+ 8\l(\nu_q^2+\nu_q^4\r)\l( \frac{\rho}{\delta^2Q} \r)^{1/2}\l( \frac{\log(ed)}{N} \r)^{1/4},
\end{multline*}
Thus, by Lemma \ref{lem:sparse-eq}, the sparsity condition $\Delta(\eta\theta_*,\rho)\geq 3\rho/4$ is satisfied for any $\rho\geq 8r(\rho)\sqrt{s}$, 
which implies $\Delta(\eta\theta_*,\rho)\geq 3\rho/4$ for any 
\[
\rho\geq  \overline C_2(\nu_q,\nu,\xi) \frac{wu^2v+w\beta^{3/4}}{\delta^2Q}s\sqrt{\frac{\log(ed)}{N}},
\]
where $\overline C_2(\nu_q,\nu,\xi)= C_2\l(\nu_q^3+\nu_q^{5/2}+\nu_q^{3/2} + \|\xi\|_{L_{q'}}(\nu_q+1) + \nu_q^2+\nu_q^4\r)^2$ for some absolute constant 
$C_2>0$. One can take the equality in the above bound and it follows,
\[
r(\rho)\leq  \overline C_1(\nu_q,\nu,\xi) \frac{wu^2v+w\beta^{3/4}}{\delta^2Q}\sqrt{\frac{s\log(ed)}{N}}.
\]
Taking the equality in the above bound and the claim follows from setting $\Lambda_Q := \delta^2Q/4$, $\Lambda_M:= \delta^2Q/64$, $\Lambda_{\mc V}:=  \delta^2Q/64$,
and 
\[
\lambda = C\frac{r(\rho)^2}{\rho} = \overline C_0(\nu_q,\nu,\xi)\frac{wu^2v+w\beta^{3/4}}{\delta^2Q}\sqrt{\frac{\log(ed)}{N}},
\]
in Theorem \ref{thm:master}.
\end{proof}

\subsection{Single-index model with heavy-tailed elliptical measurements}\label{sec:single-index}
We sketch the proof of Theorem \ref{thm:main-sparse-2}. The proof of Theorem \ref{thm:low-rank} is similar and given in the appendix.
\begin{lemma}\label{lem:bound-radiuses-2}
Define $\Omega_{\mathcal{Q}}:=\l\{r>0:  N\geq\frac{4}{Q^2}\frac{\expect{\mu^2}^2\lambda_{\max}(\mf{\Sigma})}{d^2} + \omega(S_{2}(0,r)\cap B_\Psi(0,\rho) + r)^2\r\}$ and $\Omega_{\mathcal{M}}:=\l\{r>0: N\geq (\omega(S_{2}(0,r)\cap B_\Psi(0,\rho)) + r)^2\r\} $, then, by taking $p_{\mathcal{Q}}=ce^{-t^2}$ for some absolute constant $c>0$ and $p_{\mathcal{M}} = e^{-\beta}$, we have
\begin{align*}
r_{\mathcal{Q}}\leq& \inf\l\{ r\in\Omega_{\mathcal{Q}}: \l(\frac{\delta Q}{2} - \frac{\delta t + C(\nu_q,\kappa)}{\sqrt{N}}\r)r\geq C(\nu_q,\kappa)\frac{\omega(S_{2}(0,r)\cap B_\Psi(0,\rho))}{\sqrt{N}} \r\},\\
r_{\mathcal{M}}\leq& \inf\l\{ r\in\Omega_{\mathcal{M}}:  C'(\nu,\kappa,\nu_q)\beta\frac{\omega(S_{2}(0,r)\cap B_\Psi(0,\rho)) + r}{\sqrt{N}}\leq \frac{\delta^2Q^2}{64}r^2 \r\},\\
r_{\mc V}\leq& \frac{64C''(\nu,\kappa,\nu_q)}{\delta^2Q^2\sqrt N},
\end{align*}
where $C(\nu_q,\kappa)$, $C'(\nu,\kappa,\nu_q)$, $C''(\nu,\kappa,\nu_q)$ are constants depending only on $\nu,\kappa,\nu_q$ in Assumption \ref{assumption:moment}.
\end{lemma}
The above lemma is a combination of Corollary \ref{lem:bound-Q-2}, Lemma \ref{lem:bound-M-2} and \ref{lem:bound-V-2} in the appendix. The bound on $r_{\mathcal{Q}}$ is established through another truncated small-ball argument, in conjunction with the key lower bound on quadratic forms in \citep{mendelson2014learning} as well as a recent bound on truncated multiplier process in \citep{goldstein2016structured}. The bounds on $r_{\mathcal{M}}$ and $r_{\mc V}$ relies on the bound in \citep{goldstein2016structured} again and the rotational symmetric property of the elliptical symmetric distribution.

We also need the following lemma bounding the Gaussian mean-width $ \omega( B_\Psi(0,\rho)\cap B_2(0,r))$.
\begin{lemma}[Lemma 5.3 of \citep{lecue2016regularization}]\label{lem:mendelson-sparse}
Suppose $\Psi(\cdot) = \|\cdot\|_1$, then,
there exists an absolute constant $C_0$ for which the following holds,
\[
\omega( B_\Psi(0,\rho)\cap B_2(0,r))
\leq C_0\min_k\l\{ r\sqrt{(k-1)\log(ed/(k-1))} + \rho\sqrt{\log(ed/k)} \r\}
\]
\end{lemma}

\begin{proof}[Proof of Theorem \ref{thm:main-sparse-2}]
First, taking $k=s$ in Lemma \ref{lem:mendelson-sparse} gives 
\[
\omega( B_\Psi(0,\rho)\cap B_2(0,r))\leq C_0 \l(r\sqrt{s\log(ed/s)} + \rho\sqrt{\log(ed/s)}\r),
\]
for some absolute constant $C_0>0$.
By Lemma \ref{lem:bound-radiuses-2}, we have, with probability at least $1-p_{\mathcal{Q}}$,
\[
r_{\mathcal{Q}}\leq \inf\l\{ r\in\Omega_{\mathcal{Q}}: \l(\frac{\delta Q}{2} - \frac{\delta t + C(\nu_q,\kappa)}{\sqrt{N}}\r)r\geq \frac{C_2(\nu_q,\kappa) \l(r\sqrt{s\log(ed/s)} + \rho\sqrt{\log(ed/s)}\r)}{\sqrt{N}} \r\},
\]
and when 
\begin{multline}\label{N-bound}
N\geq \l(\delta t + C(\nu_q,\kappa) + C_2(\nu_q,\kappa)\sqrt{s\log(ed/s)}\r)^2\cdot \frac{16}{\delta^2Q^2}\\
+\frac{4}{Q^2}\frac{\expect{\mu^2}^2\lambda_{\max}(\mf{\Sigma})}{d^2} + C_0\l(\l(r\sqrt{s\log(ed/s)} + \rho\sqrt{\log(ed/s)}\r) + r\r)^2,
\end{multline}
it follows
\[
r_{\mathcal{Q}}\leq \frac{4\rho\sqrt{\log(ed/s)}}{\sqrt{N}}.
\]
Also, by Lemma \ref{lem:bound-radiuses-2}, when $N\geq C_0^2 \l(r\sqrt{s\log(ed/s)} + \rho\sqrt{\log(ed/s)}\r)^2$, with probability at least $1-p_{\mathcal{M}}$,
\[
r_{\mathcal{M}}\leq C(\nu,\kappa,\nu_q)\l( \frac{\beta\sqrt{s\log(ed/s)}}{\delta^2Q^2\sqrt{N}} + \frac{\rho\sqrt{\log(ed/s)}}{\delta Q\sqrt{N}} \r)
\]
and
$
r_{\mc V}\leq \frac{64C(\nu,\kappa,\nu_q)}{\delta^2Q^2\sqrt N}.
$
Overall, since the final radius bound $r(\rho)= \max\{ r_{\mathcal{Q}},r_{\mathcal{M}}, r_{\mc V} \}$, 
we have when $N$ satisfies the bound \eqref{N-bound}, 
\[
r(\rho)\leq  \frac{4\rho\sqrt{\log(ed/s)}}{\sqrt{N}} + C(\nu,\kappa,\nu_q)\l( \frac{\beta\sqrt{s\log(ed/s)}}{\delta^2Q^2\sqrt{N}} + \frac{\rho\sqrt{\log(ed/s)}}{\delta Q\sqrt{N}} \r) + \frac{64C(\nu,\kappa,\nu_q)}{\delta^2Q^2\sqrt N}.
\]
By Lemma \ref{lem:sparse-eq}, $\rho\geq 8r(\rho)\sqrt{s}$ implies the sparsity condition $\Delta(\eta\theta_*,\rho)\geq 3\rho/4$. Thus, the sparsity condition holds for any 
$
\rho\geq C_1(\nu,\kappa,\nu_q)\beta s\sqrt{\log(ed/s)/N}.
$
In particular, take the equality in this bound and this implies $r(\rho)\leq C_0(\nu,\kappa,\nu_q)\beta\sqrt{s\log(ed/s)/N}$ for any 
$N\geq c_0(\nu,\kappa,\nu_q)\l(\frac{\delta t + \beta\sqrt{s\log(ed/s)}}{\delta Q}\r)^2+ \frac{4}{Q^2}\frac{\expect{\mu^2}^2\lambda_{\max}(\mf{\mf\Sigma})}{d^2}.$
 By Theorem \ref{thm:master}, we need to choose $\lambda = c_1(\nu,\kappa,\nu_q)\beta\sqrt{\log(ed/s)/N}$.
\end{proof}

\section{Conclusions}\label{sec:conclude}
In this paper, we introduce a truncation procedure to robustify the sparse recovery with general heavy-tailed measurements and structured single-index model with heavy-tailed elliptical measurements. We show that a new line of analysis leads to optimal sample and error rates regarding the two problem under rather minimal moment assumptions, thereby improving upon many of the previous results in the robust recovery area by relaxing assumptions on moments and prior knowledge of the measurements.

\section*{Acknowledgement}
The author thanks Stanislav Minsker and Larry Goldstein for helpful discussions related to the topic. 
The author is partially supported by Ming-Hsieh Scholarship of USC and National Science Foundation grant DMS-1712956.

\bibliographystyle{imsart-nameyear}
\bibliography{CS-Heavy-tail}

\newpage
\appendix
\section{Proof of Theorem \ref{thm:master}}

\begin{figure}[htbp]
 \centering
   \includegraphics[width=4in]{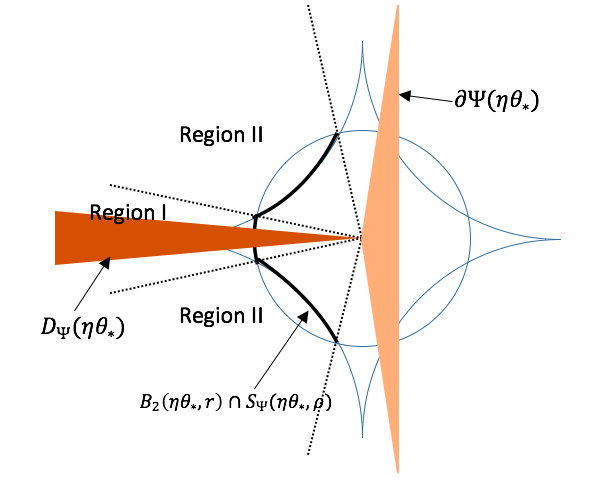} 
   \caption{A geometric interpretation that $\theta\not\in B_{2}(\eta\theta_*,r)\cap B_\Psi(\eta\theta_*,\rho)$ implies $\mc{P}_N\mathcal{L}^{\lambda}_{\theta-\eta\theta_*}>0$: When the set of sub-differentials $\partial\Psi(\eta\theta_*)$ is large, the set of descent directions $D_\Psi(\eta\theta_*)$ is small. Then, region I contains $D_\Psi(\eta\theta_*)$, in which $\Psi(\theta)\leq\Psi(\eta\theta_*)$, and the quadratic term $\mc{P}_N\mathcal{Q}_{\theta-\eta\theta_*}$ is expected to dominate $- 2\mc{P}_N\mc{M}_{\theta-\eta\theta_*} - 2 \mc{V}_{\theta-\eta\theta_*}$. On the other hand, any vector $\theta$ in region II has $\Psi(\theta)>\Psi(\eta\theta_*)$, which gives sufficient increase of norm values to dominate $- 2\mc{P}_N\mc{M}_{\theta-\eta\theta_*} - 2 \mc{V}_{\theta-\eta\theta_*}$.}
   \label{fig:geometry}
\end{figure}

\begin{proof}[Proof of Theorem \ref{thm:master}]
First of all, we have for any $\theta\in\mb{R}^d$
\[
\mc{P}_N\mathcal{L}^{\lambda}_{\theta-\eta\theta_*}
\geq \mc{P}_N\mathcal{Q}_{\theta-\eta\theta_*} - 2|\mc{P}_N\mc{M}_{\theta-\eta\theta_*}| - 2 |\mc{V}_{\theta-\eta\theta_*}| + \lambda\l( \Psi\l(\theta\r) - \Psi\l(\eta\theta_*\r) \r)
\]
\begin{enumerate}
\item Consider first that $\|\theta-\eta\theta_*\|_2>r(\rho)$ and $\Psi(\theta-\eta\theta_*)\leq\rho$. By definition of $r(\rho)$, we have
\[
 \mc{P}_N\mathcal{Q}_{\theta-\eta\theta_*} = \frac{\|\theta-\eta\theta_*\|_2^2}{r(\rho)^2}\cdot  
  \mc{P}_N\mathcal{Q}_{\frac{\theta-\eta\theta_*}{\|\theta-\eta\theta_*\|_2}r(\rho)} 
 \geq \Lambda_Q\|\theta-\eta\theta_*\|_2^2,
\]
with probability at least $1-p_{\mathcal{Q}}$, and 
\[
|\mc{P}_N\mc{M}_{\theta-\eta\theta_*}| = \l|\mc{P}_N\mc{M}_{\frac{\theta-\eta\theta_*}{\|\theta-\eta\theta_*\|_2}r(\rho)}  \r|\cdot\frac{\|\theta-\eta\theta_*\|_2}{r(\rho)}
\leq \Lambda_M\|\theta-\eta\theta_*\|_2r(\rho)\leq  \Lambda_M\|\theta-\eta\theta_*\|_2^2,
\]
with probability at least $1-p_{\mathcal{M}}$. Also,
\[
|\mc{V}_{\theta-\eta\theta_*}| = \l|\mc{V}_{\frac{\theta-\eta\theta_*}{\|\theta-\eta\theta_*\|_2}r(\rho)}  \r|\cdot\frac{\|\theta-\eta\theta_*\|_2}{r(\rho)}
\leq \Lambda_{\mc V}\|\theta-\eta\theta_*\|_2r(\rho)\leq  \Lambda_{\mc V}\|\theta-\eta\theta_*\|_2^2.
\]
For $\lambda\leq c_1\frac{r(\rho)^2}{\rho}$, we have 
$$\lambda(\Psi(\theta) - \Psi(\eta\theta_*))\geq -c_1\frac{r(\rho)^2}{\rho}\cdot\Psi(\theta-\eta\theta_*)
\geq -c_1r(\rho)^2\geq - c_1\|\theta-\eta\theta_*\|_2^2.$$
By the assumption that $\Lambda_Q> 2(\Lambda_M+\Lambda_{\mc{V}})+c_1$, we know that $ \mc{P}_N\mathcal{L}^{\lambda}_{\theta-\eta\theta_*}>0$ with probability at least $1-p_{\mathcal{Q}}-p_{\mathcal{M}}$ for $\|\theta-\eta\theta_*\|_2>r(\rho)$ and $\Psi(\theta-\eta\theta_*)\leq\rho$.

\item Consider the case $\|\theta-\eta\theta_*\|_2\leq r(\rho)$ and $\Psi(\theta-\eta\theta_*)>\rho$, then, for any specific $\theta$ satisfying the aforementioned conditions,
\begin{align*}
\mc{P}_N\mathcal{L}^{\lambda}_{\theta-\eta\theta_*}
\geq&   - 2|\mc{P}_N\mc{M}_{\theta-\eta\theta_*}| - 2 |\mc{V}_{\theta-\eta\theta_*}| + \lambda\l( \Psi\l(\theta\r) - \Psi\l(\eta\theta_*\r) \r)\\
=& \l(- 2\l|\mc{P}_N\mc{M}_{\frac{\theta-\eta\theta_*}{\Psi(\theta-\eta\theta_*)}\rho}\r|- 2 \l|\mc{V}_{\frac{\theta-\eta\theta_*}{\Psi(\theta-\eta\theta_*)}\rho}\r|\r)\cdot\frac{\Psi(\theta-\eta\theta_*)}{\rho}  + \lambda\l( \Psi\l(\theta\r) - \Psi\l(\eta\theta_*\r) \r)\\
\geq& - 2(\Lambda_M+\Lambda_{\mc{V}})r(\rho)^2\cdot\frac{\Psi(\theta-\eta\theta_*)}{\rho}  + \lambda\l( \Psi\l(\theta\r) - \Psi\l(\eta\theta_*\r) \r).
\end{align*}
Let $\mf u\in B_{\Psi}(\eta\theta_*,\rho/16)$ be the vector containing a sub-dfferential $\mf z\in\partial\Psi(\mf u)$ such that $\dotp{\mf z}{\theta-\eta\theta_*}\geq \frac34\Psi(\theta-\eta\theta_*) $. Note that this is possible because by the assumption that $\Delta(\eta\theta_*,\rho)\geq\frac34\rho$, we have there exists $\mf u\in B_{\Psi}(\eta\theta_*,\rho/16)$ with a sub-dfferential $\mf z\in\partial\Psi(\mf u)$ such that 
$
\dotp{\mf z}{\frac{\theta-\eta\theta_*}{\Psi(\theta-\eta\theta_*)}\rho}\geq\frac34\rho.
$
Thus, for the same choice of $\mf u$ and $\mf z$, $\Psi(\theta-\eta\theta_*)>\rho$ implies
\begin{equation}\label{sub-diff-bound}
\dotp{\mf z}{\theta-\eta\theta_*}=\dotp{\mf z}{\frac{\theta-\eta\theta_*}{\Psi(\theta-\eta\theta_*)}\rho}\cdot\frac{\Psi(\theta-\eta\theta_*)}{\rho} \geq\frac34\Psi(\theta-\eta\theta_*) .
\end{equation}
This implies 
\begin{align*}
\mc{P}_N\mathcal{L}^{\lambda}_{\theta-\eta\theta_*}
\geq&
- 2(\Lambda_M+\Lambda_{\mc{V}})r(\rho)^2\cdot\frac{\Psi(\theta-\eta\theta_*)}{\rho}  + \lambda\l( \Psi\l(\theta\r) - \Psi\l(\eta\theta_*+\mf u - \mf u\r) \r)\\
\geq& - 2(\Lambda_M+\Lambda_{\mc{V}})r(\rho)^2\cdot\frac{\Psi(\theta-\eta\theta_*)}{\rho}  + \lambda\l( \Psi\l(\theta\r) - \Psi\l(\mf u \r) - \frac{\rho}{16} \r)\\
\geq& - 2(\Lambda_M+\Lambda_{\mc{V}})r(\rho)^2\cdot\frac{\Psi(\theta-\eta\theta_*)}{\rho}  +\lambda\l( \dotp{\mf z}{\theta-\mf u} - \frac{\rho}{16} \r)\\
\geq& - 2(\Lambda_M+\Lambda_{\mc{V}})r(\rho)^2\cdot\frac{\Psi(\theta-\eta\theta_*)}{\rho}  +\lambda\l( \dotp{\mf z}{\theta-\eta\theta_*} - \frac{\rho}{8} \r)\\
\geq& \l(- 2(\Lambda_M+\Lambda_{\mc{V}})r(\rho)^2  +\lambda\l( \frac{3\rho}{4} - \frac{\rho}{8} \r)\r)\cdot\frac{\Psi(\theta-\eta\theta_*)}{\rho},
\end{align*}
where the second inequality follows from $\mf u\in B_{\Psi}(\eta\theta_*,\rho/16)$, the third inequality follows from the definition of sub-differential, the fourth inequality follows from Holder's inequality $\dotp{\mf z}{\eta\theta_*-\mf u}\leq \Psi^*(\mf z)\Psi(\eta\theta_*-\mf u)\leq\frac{\rho}{16} $ and the final inequality follows from the preceding argument \eqref{sub-diff-bound}. Now, we use the assumption that $\lambda\geq c_2\frac{r(\rho)^2}{\rho}$ and $c_2\geq4(\Lambda_M+\Lambda_{\mc{V}})$ to conclude that $\mc{P}_N\mathcal{L}^{\lambda}_{\theta-\eta\theta_*}>0$.

\item The case $\|\theta-\eta\theta_*\|_2> r(\rho)$ and $\Psi(\theta-\eta\theta_*)>\rho$. If $\frac{\|\theta-\eta\theta_*\|_2}{\Psi(\theta-\eta\theta_*)}>\frac{r(\rho)}{\rho}$, then, let $\alpha = \frac{\Psi(\theta-\eta\theta_*)}{\rho}$ and we have
\[
\mc{P}_N\mathcal{L}^{\lambda}_{\theta-\eta\theta_*}\geq \alpha^2\mc{P}_N\mc{Q}_{\frac{\theta-\eta\theta_*}{\Psi(\theta-\eta\theta_*)}\rho} 
- 2\alpha\l(\l| \mc{P}_N\mc{M}_{\frac{\theta-\eta\theta_*}{\Psi(\theta-\eta\theta_*)}\rho}  \r| + \l|\mc{V}_{\frac{\theta-\eta\theta_*}{\Psi(\theta-\eta\theta_*)}\rho}\r| 
+\lambda\rho \r) >0,
\]
by part 1.
On the other hand, if $\frac{\|\theta-\eta\theta_*\|_2}{\Psi(\theta-\eta\theta_*)}\leq\frac{r(\rho)}{\rho}$, then, let $\alpha = \frac{\|\theta-\eta\theta_*\|_2}{r(\rho)}$ and we have
\[
\mc{P}_N\mathcal{L}^{\lambda}_{\theta-\eta\theta_*}\geq 
- 2\alpha\l(\l| \mc{P}_N\mc{M}_{\frac{\theta-\eta\theta_*}{\|\theta-\eta\theta_*\|_2}r(\rho)}  \r| + \l|\mc{V}_{\frac{\theta-\eta\theta_*}{\|\theta-\eta\theta_*\|_2}r(\rho)}\r|\r)
+\lambda(\Psi(\theta) - \Psi(\eta\theta_*))>0,
\]
by part 2. 
\end{enumerate}
Overall, we finish the proof.
\end{proof}

\section{Sparse recovery with heavy-tailed measurements}\label{sec:sparse-recovery}
In this section, we focus on the proof of Lemma \ref{lem:bound-on-radiuses}. Our goal is to compute $r_{\mathcal{Q}},~r_{\mathcal{M}},~r_{\mc{V}}$ for specific constants 
$\Lambda_Q,~\Lambda_M,~\Lambda_{\mc V}$ satisfying the assumptions and determine the choice of $\rho$ so that 
$\Delta(\eta\theta_*,\rho)\geq \frac34\rho$.

\subsection{A truncated small-ball condition}
We start with the classical small-ball estimate:
\begin{lemma}\label{lem:small-ball}
Under Assumption \ref{assumption:moment}, let $\delta = \frac12\sqrt{\frac{\kappa}{2}}$ and $Q=\frac{\kappa^2}{8\nu}$, then, we have
\[
\inf_{\mathbf{v}\in\mathbb{R}^{d}} Pr\l( \Big| \dotp{\mathbf{x}_i}{\mathbf{v}} \Big| \geq 2\delta\|\mathbf{v}\|_2 \r) \geq 2Q.
\]
\end{lemma}
\begin{proof}
By Paley-Zygmund inequality, we know for any nonnegative real valued random variable $Z$,
\[
Pr(Z>t\expect{Z})\geq (1-t)^2\frac{\expect{Z}^2}{\expect{Z^2}},
\]
for any $t\geq0$. Now, fix any $\mathbf{v}\in\mathbb{R}^{d}$, we take $Z=| \dotp{\mathbf{x}_i}{\mathbf{v}} |^2$, $t = 1/2$, and obtain
\[
Pr\l(  \Big| \dotp{\mathbf{x}_i}{\mathbf{v}} \Big|^2 \geq \frac12 \expect{\Big| \dotp{\mathbf{x}_i}{\mathbf{v}} \Big|^2}\r)
\geq\frac14\frac{\expect{| \dotp{\mathbf{x}_i}{\mathbf{v}} |^2}^2}{\expect{| \dotp{\mathbf{x}_i}{\mathbf{v}} |^4}}
\] 
Recall from Assumption \ref{assumption:moment}, $\lambda_{\min}(\mathbf{\mathbf{\Sigma}}_X)>\kappa$, thus, $\expect{\Big| \dotp{\mathbf{x}_i}{\mathbf{v}} \Big|^2}\geq\kappa\|\mathbf{v}\|_2^2$
for any $\mathbf{v}\in\mathbb{R}^{d}$, and it follows,
\begin{align*}
\inf_{\mathbf{v}\in\mathbb{R}^{d}} Pr\l(  \Big| \dotp{\mathbf{x}_i}{\mathbf{v}} \Big| \geq \sqrt{\frac{\kappa}{2}}\|\mathbf{v}\|_2 \r)
&\geq \inf_{\mathbf{v}\in\mathbb{R}^{d}} Pr\l(  \Big| \dotp{\mathbf{x}_i}{\mathbf{v}} \Big|^2 \geq \frac12 \expect{\Big| \dotp{\mathbf{x}_i}{\mathbf{v}} \Big|^2}\r)\\
&\geq \inf_{\mathbf{v}\in\mathbb{R}^{d}}  \frac14\expect{\Big| \dotp{\mathbf{x}_i}{\mathbf{v}} \Big|^2}^2\left/\expect{\Big| \dotp{\mathbf{x}_i}{\mathbf{v}} \Big|^4}\right.\\
&\geq  \frac14\frac{ \inf_{\mathbf{v}\in\mathbb{S}_2^d} \expect{| \dotp{\mathbf{x}_i}{\mathbf{v}} |^2}^2}{\sup_{\mathbf{v}\in\mathbb{S}_2^d}
\expect{| \dotp{\mathbf{x}_i}{\mathbf{v}}  |^4}}\geq \frac{\kappa^2}{4\nu},
\end{align*}
where the last inequality follows from Assumption \ref{assumption:moment}. Taking $\delta = \frac12\sqrt{\frac{\kappa}{2}}$ and $Q=\frac{\kappa^2}{8\nu}$ finishes the proof.
\end{proof}

We see from Lemma \ref{lem:small-ball} that indeed such a small-ball condition is easily satisfied merely under a bounded moment assumption. The following lemma is the key to our analysis. It says a somewhat ``weaker'' small-ball condition is preserved under adaptive thresholding.

\begin{lemma}\label{lem:weak-small-ball}
Let $s_0$ be a positive integer such that $1\leq s_0\leq d$. Let $\mathcal{G}_{s_0}$ be the set of all vectors in $\mathbb{R}^d$ with $s_0$ cardinality of the support set. Suppose Assumption \ref{assumption:moment} holds and $N\geq \frac{\nu}{Q}s_0\log(ed) $, then,
for any $\mathbf{v}\in \mathcal{G}_{s_0}$,
\[
Pr\l( \Big| \dotp{\widetilde{\mathbf{x}}_i}{\mathbf{v}} \Big| \geq \delta\|\mathbf{v}\|_2 \r) \geq Q.
\]
\end{lemma}

\begin{proof}
First, note that for any vector $\mathbf{v}\in \mathcal{G}_{s_0}$,
\begin{align*}
\l|  \dotp{\widetilde{\mathbf{x}}_i}{\mathbf{v}} \r| = \l|  \dotp{\widetilde{\mathbf{x}}_i - \mathbf{x}_i}{\mathbf{v}} +  \dotp{\mathbf{x}_i}{\mathbf{v}} \r|
\geq\l| \dotp{\mathbf{x}_i}{\mathbf{v}} \r| - \l| \dotp{\widetilde{\mathbf{x}}_i - \mathbf{x}_i}{\mathbf{v}}  \r|.
\end{align*}
Thus, it follows
\begin{align}
Pr\l( \l| \dotp{\widetilde{\mathbf{x}}_i}{\mathbf{v}} \r| \geq \delta\|\mathbf{v}\|_2 \r) 
\geq& Pr\l( \l| \dotp{\mathbf{x}_i}{\mathbf{v}} \r| \geq \delta\|\mathbf{v}\|_2 +  \l| \dotp{\widetilde{\mathbf{x}}_i - \mathbf{x}_i}{\mathbf{v}}  \r| \r) \nonumber\\
\geq& Pr\l( \l\{ \l| \dotp{\mathbf{x}_i}{\mathbf{v}} \r| \geq 2\delta\|\mathbf{v}\|_2 \r\}   
  \cap   \l\{ \l| \dotp{\widetilde{\mathbf{x}}_i - \mathbf{x}_i}{\mathbf{v}}  \r|\leq \delta\|\mathbf{v}\|_2 \r\}  \r) \nonumber\\
\geq&  Pr\l( \l| \dotp{\mathbf{x}_i}{\mathbf{v}} \r| \geq 2\delta\|\mathbf{v}\|_2 \r)  
- 
Pr\l( \l| \dotp{\widetilde{\mathbf{x}}_i - \mathbf{x}_i}{\mathbf{v}}  \r|\geq\delta\|\mathbf{v}\|_2 \r), \label{inter-0}
\end{align}
where the last inequality follows from the fact that for any two measurable set $A,B$ in a probability space $(\Omega,\mathcal{E},\mathbb{P})$, $Pr(A\cap B) = Pr(A \setminus (B^c\cap A) )\geq Pr(A) - Pr(B^c\cap A)\geq Pr(A) - Pr(B^c)$. By Lemma \ref{lem:small-ball}, $Pr\l( \Big| \dotp{\mathbf{x}_i}{\mathbf{v}} \Big| \geq 
2\delta\|\mathbf{v}\|_2 \r)\geq 2Q$. It remains to bound $Pr\l( \l| \dotp{\widetilde{\mathbf{x}}_i - \mathbf{x}_i}{\mathbf{v}}  \r|\geq\delta\|\mathbf{v}\|_2 \r)$ from above. To this point, let $\mathcal{P}_{\mathbf{v}}\mathbf{x}$ be the orthogonal projection of a vector $\mathbf{x}\in\mathbb{R}^d$ onto the non-zero coordinates of $\mathbf{v}$. Then, by Holder's inequality, we have
\begin{align*}
Pr\l( \l| \dotp{\widetilde{\mathbf{x}}_i - \mathbf{x}_i}{\mathbf{v}}  \r|\geq\delta\|\mathbf{v}\|_2 \r)
\leq& Pr\l(  \|\mc{P}_{\mf v}(\widetilde{\mathbf{x}}_i - \mathbf{x}_i)\|_{\infty}\|\mathbf{v}\|_1  \geq\delta\|\mathbf{v}\|_2 \r)\\
=& Pr\l(  \| \mc{P}_{\mf v} (\widetilde{\mathbf{x}}_i - \mathbf{x}_i)\|_{\infty}  \geq\delta\frac{\|\mathbf{v}\|_2}{\|\mathbf{v}\|_1} \r)\\
\leq& Pr\l( \| \mc{P}_{\mf v}\mathbf{x}_i \|_{\infty}>\tau \r),
\end{align*}
where the last inequality follows from the definition of $\wt{\mf x}_i$ in \eqref{eq:trunc1} that if every entry of $\mc{P}_{\mf v}\mathbf{x}_i $ is bounded by $\tau$, then
$\mc{P}_{\mf v}\mathbf{x}_i = \mc{P}_{\mf v}\wt{\mathbf{x}}_i $. Furthermore,
\begin{multline*}
Pr\l( \| \mc{P}_{\mf v}\mathbf{x}_i \|_{\infty}>\tau \r)
\leq Pr\l( \l(\sum_{j\in\mc G_{\mf v}}x_{ij}^4\r)^{\frac14}> \tau \r)
= Pr\l( \sum_{j\in\mc G_{\mf v}}x_{ij}^4>\tau^4\r)\\
\leq\frac{\expect{\sum_{j\in\mc \nu G_{\mf v}}x_{ij}^4}}{\tau^4}
\leq \frac{s_0\nu\log (ed) }{N},
\end{multline*}
where the second from the last inequality follows from Markov inequality and the last inequality follows from the definition of $\tau = (N/\log (ed))^{1/4}$ and  the assumption that $\expect{x_{ij}^4}\leq \nu$. Since $N\geq \frac{\nu}{Q}s_0\log (ed) $ by assumption, we have $Pr\l( \| \mc{P}_{\mf v}\mathbf{x}_i \|_{\infty}>\tau \r)\geq Q$ and the proof is finished.
\end{proof}

\subsection{Computing the critical radiuses}
We set $\Lambda_Q := \delta^2Q/4$, $\Lambda_M:= \delta^2Q/64$, $\Lambda_{\mc V}:=  \delta^2Q/64$ and
 $\frac{\delta^2Q}{8}\frac{r(\rho)^2}{\rho}\leq \lambda\leq \frac{5\delta^2Q}{32}\frac{r(\rho)^2}{\rho}$. Then, we have $\Lambda_Q> 2(\Lambda_M+\Lambda_{\mc{V}})+\frac{5\delta^2Q}{32}$ 
and $\frac{\delta^2Q}{8}\geq4(\Lambda_M+\Lambda_{\mc{V}})$, satisfying the assumptions in Theorem \ref{thm:master}. We aim to bound the critical radiuses $r_{\mathcal{Q}},~r_{\mathcal{M}},~r_{\mc V}$ and show there exists $\rho>0$ such that $\Delta(\theta_*,\rho)\geq\frac34\rho$. 

\subsubsection{Bounding the radius $r_{\mathcal{Q}}$}

The following useful lower bound on the random quadratic form comes from \cite{lecue2017sparse}. Lower bounds of this sort via Maurey's empirical method originate from \cite{oliveira2013lower}.

\begin{lemma}[Lemma 2.7 of \cite{lecue2017sparse}]\label{lem:quad-form}
Let $\bold{\Gamma}: \mathbb{R}^{d}\rightarrow\mathbb{R}^N$. Let $s_0$ be a positive integer such that $1< s_0\leq d$. Assume for any $\mathbf{v}\in\mathcal{G}_{s_0}$, 
$\l\| \bold{\Gamma}\mathbf{v} \r\|_2\geq \xi\|\mathbf{v}\|_2$ for some absolute constant $\xi>0$. If $\mathbf{x}\in\mathbb{R}^d$ is a non-zero vector and $\mu_j = |x_j|/\|\mathbf{x}\|_1$, then,
\[
\l\| \bold{\Gamma}\mathbf{x} \r\|_2^2\geq \xi^2\|\mathbf{x}\|_2^2 - \frac{\|\mathbf{x}\|_1^2}{s_0-1}\l( \sum_{j=1}^d \l\|\bold{\Gamma}\mathbf{e}_j \r\|_2^2\mu_j - \xi^2 \r),
\] 
where $\{\mathbf{e}_j\}_{j=1}^d$ is the standard basis in $\mathbb{R}^d$.
\end{lemma}

Let $\wt{\mf{\Gamma}} := \l[\wt{\mf x}_1,~\wt{\mf x}_2,\cdots,~\wt{\mf x}_N\r]^T/\sqrt{N}$.
To use this lemma, we need to deduce a lower bound for $\inf_{\mathbf{v}\in\mathcal{G}_{s_0}}\l\| \widetilde{\bold{\Gamma}}\mathbf{v} \r\|_2^2$ as well as an upper bound for 
$\max_{1\leq j\leq d}\l\|  \widetilde{\bold{\Gamma}}\mathbf{e}_j \r\|_2^2$. The former is bounded via a book-keeping VC dimension argument, and the latter is bounded via a subgaussian concentration bound for a thresholded process.

\begin{lemma}\label{lem:VC}
Suppose $N\geq \frac{\nu}{Q}s_0\log(ed) $, then,
with probability at least $1-c\exp(-\beta)$,
\[
\inf_{\mathbf{v}\in\mathcal{G}_{s_0}\cap\mathbb{S}^{d-1}}\l\| \widetilde{\bold{\Gamma}}\mathbf{v} \r\|_2^2\geq \delta^2\l( Q - L\sqrt{\frac{s_0\log(ed)}{N}} - \sqrt{\frac{\beta}{N}} \r),
\]
where $L,c>0$ are absolute constants.
\end{lemma}

\begin{proof}[Proof of Lemma \ref{lem:VC}]
First of all, by Lemma \ref{lem:weak-small-ball}, for any $i\in\{1,2,\cdots,N\}$ and $\mathbf{v}\in\mathcal{G}_{s_0}\cap\mathbb{S}^{d-1}$, we have
\[
\expect{\mathbf{1}_{  \l\{ \l|\dotp{\widetilde{\mathbf{x}}_i}{\mathbf{v}}\r| \geq\delta \r\}  }} = Pr\l( \Big| \dotp{\widetilde{\mathbf{x}}_i}{\mathbf{v}} \Big| \geq \delta\|\mathbf{v}\|_2 \r) \geq Q.
\]
Let $\widetilde{\mathbf{x}}_1^N:= \l[ \widetilde{\mathbf{x}}_1,\cdots, \widetilde{\mathbf{x}}_N \r]$, and
define the following process parametrized by $\mathbf{v}\in\mathcal{G}_{s_0}\cap\mathbb{S}^{d-1}$:
\[
R\l( \widetilde{\mathbf{x}}_1^N,\mathbf{v} \r) = \frac1N\sum_{i=1}^N\mathbf{1}_{  \l\{ \l|\dotp{\widetilde{\mathbf{x}}_i}{\mathbf{v}}\r| \geq\delta/2 \r\}  }
- \expect{\mathbf{1}_{  \l\{ \l|\dotp{\widetilde{\mathbf{x}}_i}{\mathbf{v}}\r| \geq\delta/2 \r\}  }},
\]
and we aim to bound the following supremum
\[
\sup_{\mathbf{v}\in\mathcal{G}_{s_0}\cap\mathbb{S}^{d-1}}\l| R\l( \widetilde{\mathbf{x}}_1^N,\mathbf{v} \r) \r|.
\]
Define the following class of indicator functions:
\[
\mathcal{F} := \l\{ \mathbf{1}_{\l\{ |\dotp{\cdot}{\mathbf{v}}| \geq \delta/2\r\}}, ~\mathbf{v}\in \mathcal{G}_{s_0}\cap\mathbb{S}^{s_0-1}  \r\},
\]
By the standard symmetrization argument and then Dudley's entropy estimate (see, for example, \cite{van1996weak} for details of VC theory), we have
\begin{equation}\label{inter-2-1}
\expect{\sup_{\mathbf{v}\in\mathcal{G}_{s_0}\cap\mathbb{S}^{d-1}}\l| R\l( \widetilde{\mathbf{x}}_1^N,\mathbf{v} \r) \r|} 
\leq \frac{C_0}{\sqrt{N}}\int_0^2\sqrt{\log  \mathcal{N}\l( \varepsilon,\mathcal{F}, \|\cdot\|_{L_2(\mu_N)} \r) }d\varepsilon,
\end{equation}
where $C_0$ is a constant and $\mathcal{N}\l( \varepsilon,\mathcal{F}, \|\cdot\|_{L_2(\mu_N)} \r)$ is the $\varepsilon$-covering number of $\mathcal{F}$ under the norm 
$\|f-g\|_{L_2(\mu_N)}:= \sqrt{\frac{1}{N}\sum_{i=1}^N(f(\mathbf{x}_i) - g(\mathbf{x}_i))^2}$.

Consider, without loss of generality, a particular subspace $K_{s_0}$ of $\mathbb{R}^d$ consisting of all vectors whose first $s_0$ coordinates are non-zero. Note that for any fixed number $c\in \mathbb{R}$, the VC dimension of the set of halfspaces
$\mathcal{H}:=\{\dotp{\cdot}{\mathbf{v}}\geq c, ~\mathbf{v}\in K_{s_0}\cap\mathbb{S}^{s_0-1} \}$ is $VC(\mathcal{H}) = s_0$. Thus, by classical VC theorem, for any distinctive $p$ points in $\mathbb{R}^d$, the number distinctive projections from $\mathcal{H}$ to these $p$ points is $\sum_{i=0}^{s_0}{p \choose i}\leq (p+1)^{s_0}$.
Furthermore, any set in $\mathcal{H}':=\{|\dotp{\cdot}{\mathbf{v}}|\geq c, ~\mathbf{v}\in K_{s_0}\cap\mathbb{S}^{s_0-1} \}$ is the intersection of two sets in $\mathcal{H}$,
thus, the number of distinctive projections from $\mathcal{H}'$ to those $p$ points is at most 
$${(p+1)^{s_0} \choose 2}\leq \frac{e^2(p+1)^{2s_0}}{4}\leq 2(p+1)^{2s_0}.$$
This implies $VC(\mathcal{H}')\leq cs_0\log(s_0)$ for some absolute constant $c>0$.

Thus, the following class of indicator functions 
\[
\mathcal{F}_{\delta,K_{s_0}} := \l\{ \mathbf{1}_{\l\{ |\dotp{\cdot}{\mathbf{v}}| \geq \delta\r\}}, ~\mathbf{v}\in K_{s_0}\cap\mathbb{S}^{s_0-1} \r\}
\]
has VC dimension $VC(\mathcal{F}_{\delta,K_{s_0}})\leq cs_0\log(s_0)$. By Haussler's inequality, we have the $\varepsilon$ covering number of $\mathcal{F}_{\delta,K_{s_0}}$ can be bounded as
\begin{align*}
\mathcal{N}\l( \varepsilon,\mathcal{F}_{\delta,K_{s_0}}, \|\cdot\|_{L_2(\mu_N)} \r) 
\leq Cs_0(16e)^{cs_0\log(s_0)}\varepsilon^{-2cs_0\log(s_0)},
\end{align*}
where $C>0$ is an absolute constant.
Furthermore, $\mathcal{F}$ is the union of ${ d \choose s_0}$ different subspaces $K_{s_0}$. Thus, the $\varepsilon$ covering number of $\mathcal{F}$ can be bounded as
 \begin{align*}
 \mathcal{N}\l( \varepsilon,\mathcal{F}, \|\cdot\|_{L_2(\mu_N)} \r)
 &\leq { d \choose s_0}Cs_0(16e)^{cs_0\log(s_0)}\varepsilon^{-2cs_0\log(s_0)}\\
 &\leq \l(ed/s_0\r)^{s_0}Cs_0(16e)^{cs_0\log(s_0)}\varepsilon^{-2cs_0\log(s_0)}.
 \end{align*}
 Substituting this bound into \eqref{inter-2-1} gives 
 \begin{align*}
 \expect{\sup_{\mathbf{v}\in\mathcal{G}_{s_0}\cap\mathbb{S}^{d-1}}\l| R\l( \widetilde{\mathbf{x}}_1^N,\mathbf{v} \r) \r|}
\leq  L\sqrt{s_0\log(ed)/N},
 \end{align*}
 for some absolute constant $L>0$. By bounded difference inequality, we have
 \begin{align*}
 \sup_{\mathbf{v}\in\mathcal{G}_{s_0}\cap\mathbb{S}^{d-1}}\l| R\l( \widetilde{\mathbf{x}}_1^N,\mathbf{v} \r) \r| 
 \leq \expect{\sup_{\mathbf{v}\in\mathcal{G}_{s_0}\cap\mathbb{S}^{d-1}}\l| R\l( \widetilde{\mathbf{x}}_1^N,\mathbf{v} \r) \r|} 
 + \sqrt{u/N},
 \end{align*}
 with probability at least $1-ce^{-u}$ for some constant $c>0$ any $u\geq0$, which implies 
 \begin{align*}
 \inf_{\mathbf{v}\in \mathcal{G}_{s_0}\cap\mathbb{S}^{d-1}}\frac1N\sum_{i=1}^N\mathbf{1}_{  \l\{ \l|\dotp{\widetilde{\mathbf{x}}_i}{\mathbf{v}}\r| \geq\delta/2 \r\}  }
 \geq Q/2 - L\sqrt{s_0\log(ed)/N} - \sqrt{u/N},
 \end{align*}
 with probability at least $1-ce^{-u}$. This implies the claim of the lemma.
\end{proof}

\begin{lemma}\label{lem:upper-1}
For any $\beta\geq1$ chosen by the thresholding parameter $\tau$, we have with probability at least $1-e^{-\beta}$,
\[
\max_{1\leq j\leq d}\l\|  \widetilde{\bold{\Gamma}}\mathbf{e}_j \r\|_2^2 \leq \sqrt{\nu} + C (\sqrt{\nu}+1)\beta\sqrt{\frac{\log (ed)}{N}}, 
\]
where $C>0$ is an absolute constant.
\end{lemma}

\begin{proof}[Proof of Lemma \ref{lem:upper-1}]
By Bernstein's inequality, we have for any $t\geq0$,
\begin{align*}
Pr\l( \l| \frac1N\sum_{i=1}^N\widetilde{x}_{ij}^2 - \expect{\widetilde{x}_{ij}^2} \r| \geq C\l( \sqrt{\frac{2\sigma_j^2 t}{N}} + \frac{bt}{N} \r) \r)
\leq \exp(-t),
\end{align*}
where for i.i.d. measurements $\mf x_1,~\mf x_2,~\cdots,~\mf x_N$,
$$\sigma_j^2 = \expect{\l(\widetilde{x}_{ij}^2-\expect{\wt x_{ij}^2}\r)^2}\leq \expect{| x_{ij}|^4}\leq \sup_{\mathbf{v}\in\mathbb{S}^{d-1}} \expect{\l| \dotp{\mathbf{v}}{\mathbf{X}_i} \r|^4}\leq \nu,$$ 
$b = \tau^2 = \sqrt{\frac{N}{\log (ed)}}$, and $\expect{\widetilde{x}_{ij}^2}\leq \expect{|\widetilde{x}_{ij}|^4}^{1/2}\leq \sqrt{\nu}$. Thus, it follows for any $j\in\{1,2,\cdots,d\}$,
\[
 \frac1N\sum_{i=1}^N\widetilde{x}_{ij}^2  \leq \sqrt{\nu} + C\l(\sqrt{\frac{2\nu t}{N}} + \frac{t}{\sqrt{N\log (ed)}} \r),
\]
with probability at least $1-\exp(-t)$. Take a union bound over $j\in\{1,2,\cdots,d\}$ and let $t=\beta \log (ed) $ give
\[
 \max_{1\leq j\leq d} \frac1N\sum_{i=1}^N\widetilde{x}_{ij}^2\leq \sqrt{\nu} + C (\sqrt{\nu}+1)\beta\sqrt{\frac{\log (ed)}{N}},
\]
with probability at least $1-e^{-\beta}$,
for some absolute constant $C>0$. This finishes the proof.
\end{proof}

Combining the preceding three lemmas gives 
\begin{lemma}\label{lem:bound-Q}
Suppose $N\geq C_0\l( \frac{s_0}{Q^2} + \frac{\nu+1}{\nu} \r)\beta^2\log(ed)+ \frac{\nu}{Q}s_0\log(ed)$ for some absolute constant $C_0>0$, and $s_0 = \frac{8c_0\sqrt{\nu}}{\delta^2Q}s\leq d$ for some absolute constant $c_0>0$, then,
\[
r_{\mathcal{Q}}\leq\sqrt{\frac{2}{c_0s}}\rho
\]
when taking $p_{\mathcal{Q}} = ce^{-\beta}$ in the definition of $r_{\mathcal{Q}}$ for $\beta\geq1$.
\end{lemma}

\begin{proof}[Proof of Lemma \ref{lem:bound-Q}]
First of all, by Lemma \ref{lem:VC}, and the assumption $N\geq C_0 \frac{s_0}{Q^2}\beta^2\log(ed)+ \frac{\nu}{Q}s_0\log(ed)$ for some large enough absolute constant $C_0$, we have 
$$\inf_{\mf v\in \mc G_{s_0}\cap \mc S^{d-1}}\l\| \widetilde{\bold{\Gamma}}\mathbf{v} \r\|_2^2\geq\frac{\delta^2Q}{2},$$
with probability at least $1-e^{-\beta}$.
Thus, it follows from Lemma \ref{lem:quad-form} and \ref{lem:upper-1} that
\[
\inf_{\theta\in B_1(\theta_*,\rho)\cap S_2(\theta_*,r)}\mc P_N\mc Q_{\theta-\theta_*}\geq \frac{\delta^2Q}{2}r^2 - \frac{\rho^2}{s_0-1}
\l( \sqrt\nu + C\l(\sqrt{\nu}+1\r)\beta\sqrt{\frac{\log (ed)}{N}} \r),
\]
where $C>0$ is an absolute constant.
By assumption that $N\geq C_0\frac{\nu+1}{\nu}\beta^2\log (ed)$ for some $C_0$ large enough, then,
\[
\inf_{\theta\in B_1(\theta_*,\rho)\cap S_2(\theta_*,r)}\mc P_N\mc Q_{\theta-\theta_*}
\geq \frac{\delta^2Q}{2}r^2 - \frac{2\sqrt{\nu}}{s_0-1}\rho^2\geq \frac{\delta^2Q}{2}r^2 - \frac{4\sqrt{\nu}}{s_0}\rho^2.
\]
Using the assumption that $s_0 = \frac{8c_0\sqrt{\nu}}{\delta^2Q}s$, we obtain 
\[
\inf_{\theta\in B_1(\theta_*,\rho)\cap S_2(\theta_*,r)}\mc P_N\mc Q_{\theta-\theta_*}\geq \frac{\delta^2Q}{2}\l(r^2 - \frac{\rho^2}{c_0s}\r).
\]
The infimum of $r>0$ such that the right hand side is greater than $\frac{\delta^2Q}{4}r^2$ can be obtained by 
letting the right hand side equal to $\frac{\delta^2Q}{4}r^2$ and solve for $r$, which gives $r = \sqrt{\frac{2}{c_0s}}\rho$. It then follows from the definition of $r_{\mathcal{Q}}$ that $r_{\mathcal{Q}}$ must be bounded above by this value.
\end{proof}

\subsubsection{Bounding the radius $r_{\mathcal{M}}$}
The main objective is the following bound on $|\mc P_N \mc M_{\theta-\theta_*}|$:

\begin{lemma}\label{lem:bound-PM}
Suppose $N\geq cs\log(ed)$ and Assumption \ref{assume:sparse-recovery} holds.
For any $\beta,u,v,w>6$, we have with probability at least 
\begin{multline*}
1-2e^{-\beta}-2e^{-v^2}
-c'\l((u^{-q/4}+u^{-q'})(ed)^{-(c-1)}\r.\\
\l.+(eN)^{-\frac{q}{12}+1}(\log(eN))^{q/6}w^{-q/6} +  (eN)^{-\frac{q'}{4}+1}(\log(eN))^{q'/2}w^{-q'}\r).
\end{multline*}
where $c,c'>1$ are absolute constants,
\[
\sup_{\theta\in B_1(\theta_*,\rho)\cap B_2(\theta_*,r)}\l| \mc P_N \mc M_{\theta-\theta_*} \r|\leq C(\nu_q,\xi)\l(wu^2v+w\beta^{3/4}\r)
\sqrt{\frac{\log (ed)}{N}} \l( r\sqrt m  + \rho \r),
\]
for any $m\in\{1,2,\cdots,d\}$,~
where $C(\nu_q,\xi):= C\l(\nu_q^3+\nu_q^{5/2}+\nu_q^{3/2} + \|\xi\|_{L_{q'}}\l( \nu_q +1 \r)\r)$ for some absolute constant $C>0$
\end{lemma}


\begin{proof}[Proof of Lemma \ref{lem:bound-PM}]
First of all, by symmetrization inequality, it is enough to bound 
\[
\sup_{\theta\in B_1(\theta_*,\rho)\cap B_2(\theta_*,r)}
\l| \frac1N\sum_{i=1}^N\varepsilon_i(\wt y_i - \dotp{\wt{\mf x_i}}{\theta_*})\dotp{\wt{\mf x}_i}{\theta-\theta_*} \r| 
= \sup_{\mf v\in B_1(0,\rho)\cap B_2(0,r)}
\l| \frac1N\sum_{i=1}^N\varepsilon_i(\wt y_i - \dotp{\wt{\mf x}_i}{\theta_*})\dotp{\wt{\mf x}_i}{\mf v} \r|
\]
We define $\mf z := \frac1N\sum_{i=1}^N\varepsilon_i(\wt y_i - \dotp{\wt{\mf x}_i}{\theta_*})\wt{\mf x}_i$. Let $J$ be any group of coordinates in $\{1,2,\cdots,d\}$ with $m$ largest coordinates of $\l\{|z_j|\r\}_{j=1}^N$ for $m\in\{1,2,\cdots,d\}$. Then, it follows
\begin{multline}\label{eq:master-bound-0}
\sup_{\mf v\in B_1(0,\rho)\cap B_2(0,r)}
\dotp{\mf z}{\mf v} \leq \sup_{\mf v\in B_1(0,\rho)\cap B_2(0,r)}\sum_{j\in J}v_jz_j + \sup_{\mf v\in B_1(0,\rho)\cap B_2(0,r)}\sum_{j\in J^c}v_jz_j \\
\leq  \sup_{\mf v\in B_2(0,r)}\sum_{j\in J}v_jz_j +  \sup_{\mf v\in B_1(0,\rho)}\sum_{j\in J^c}v_jz_j
= r \cdot \l(\sum_{j\leq m}\l(z_j^{\sharp}\r)^2 \r)^{1/2} + \rho\cdot \max_{j>m}\l| z_j^\sharp \r|\\
\leq  \max_{j}\l| z_j \r|\cdot \l( r\sqrt m  + \rho \r)
\end{multline}
for any $m$, 
where $\l\{ z_j^{\sharp} \r\}_{j=1}^d$ denotes the non-increasing ordering of $\l\{ |z_j| \r\}_{j=1}^d$. Now for each $|z_j|$,
\begin{equation*}
N|z_j| = \l|\sum_{i=1}^N\varepsilon_i( \wt y_i - \dotp{\wt{\mf x}_i}{\theta_*})\wt x_{ij}  \r|
\leq \l|\sum_{i=1}^N\varepsilon_i(\wt y_i-y_i)\wt x_{ij} \r|  +
\l|\sum_{i=1}^N\varepsilon_i \xi_i \wt x_{ij}  \r| + \l|\sum_{i=1}^N\varepsilon_i\l\langle 
\mf x_i -\wt{\mf x}_i, \theta_* \r\rangle \wt x_{ij}  \r|
\end{equation*}
Thus, it follows 
\begin{multline}\label{eq:decomp-1}
N\cdot\max_{j\in\{1,2,\cdots,d\}}|z_j| \leq \max_{j\in\{1,2,\cdots,d\}}\l|\sum_{i=1}^N\varepsilon_i(\wt y_i-y_i)\wt x_{ij} \r|
+ \max_{j\in\{1,2,\cdots,d\}} \l|\sum_{i=1}^N\varepsilon_i \xi_i \wt x_{ij}  \r| \\
+ \max_{j\in\{1,2,\cdots,d\}}\l|\sum_{i=1}^N\varepsilon_i\l\langle 
\mf x_i -\wt{\mf x}_i, \theta_* \r\rangle \wt x_{ij}  \r|
\end{multline}

Then, we need to bound the three terms on the right hand side of \eqref{eq:decomp-1} separately.

\begin{enumerate}
\item \textbf{Bounding the terms} $\max_{j\in\{1,2,\cdots,d\}}\l|\sum_{i=1}^N\varepsilon_i\l\langle 
\mf x_i -\wt{\mf x}_i, \theta_* \r\rangle \wt x_{ij}  \r|$ \textbf{and}\\
 $\max_{j\in\{1,2,\cdots,d\}}\l|\sum_{i=1}^N\varepsilon_i(\wt y_i-y_i)\wt x_{ij} \r|$:

Let 
$\wt \phi_i = \l\langle \mf x_i -\wt{\mf x}_i, \theta_* \r\rangle$. A usual first step analyzing such a Rademacher sum (see, for example, \citep{mendelson2016upper, goldstein2016structured}) is to apply an inequality from \citep{montgomery1990distribution}, conditioned on $\mf{x}_i$, which results in
\[
\l|\sum_{i=1}^N\varepsilon_i\wt \phi_i \wt x_{ij}  \r|
\leq \sum_{i=1}^k\l|\wt\phi_i^{\sharp}\wt x_{ij}^{\sharp} \r|
+v\l( \sum_{i>k}\l|\wt\phi_i^{\sharp}\wt x_{ij}^{\sharp} \r|^2 \r)^{1/2},
\]
with probability at least $1-e^{-v^2}$, 
where $k$ is any chosen integer within $\l\{0,1,2,\cdots,N \r\}$ and $\l(\wt\phi_i^{\sharp}\r)_{i=1}^N$, $\l(\wt x_{ij}^{\sharp}\r)_{i=1}^N$ are non-increasing rearrangements of  $\l(|\wt\phi_i|\r)_{i=1}^N$, $\l(|\wt x_{ij}|\r)_{i=1}^N$. We define the former sum to be 0 when $k=0$.

By Holder's inequality, we have
\begin{equation*}
\l|\sum_{i=1}^N\varepsilon_i\wt \phi_i \wt x_{ij}  \r|
\leq \l(\sum_{i=1}^k\l|\wt\phi_i^{\sharp}\r|^2\r)^{1/2}\l(\sum_{i=1}^k\l|\wt x_{ij}^{\sharp} \r|^2\r)^{1/2}
+v\l( \sum_{i>k}\l|\wt\phi_i^{\sharp} \r|^{2r} \r)^{1/(2r)} 
\l(\sum_{i>k}\l|\wt x_{ij}^{\sharp}\r|^{2r'}\r)^{1/(2r')},
\end{equation*}
for some positive constants $r,r'$ such that $\frac1r+\frac{1}{r'}=1$. Take a union bound for all $j\in\{1,2,\cdots,d\}$, gives with probsability at least $1-e^{-v^2}$,
\begin{multline}\label{eq:master-bound-2}
\max_{j\in\{1,2,\cdots,d\}}\l|\sum_{i=1}^N\varepsilon_i\wt \phi_i \wt x_{ij}  \r|
\leq \l(\sum_{i=1}^k\l|\wt\phi_i^{\sharp}\r|^2\r)^{1/2}
\max_{j\in\{1,2,\cdots,d\}}\l(\sum_{i=1}^k\l|\wt x_{ij}^{\sharp} \r|^2\r)^{1/2}\\
+v\sqrt{\log d}\l( \sum_{i>k}\l|\wt\phi_i^{\sharp} \r|^{2r} \r)^{1/(2r)} 
\max_{j\in\{1,2,\cdots,d\}}\l(\sum_{i>k}\l|\wt x_{ij}^{\sharp}\r|^{2r'}\r)^{1/(2r')},
\end{multline}

Now we bound the four terms in \eqref{eq:master-bound-2} respectively.

\begin{lemma}\label{lem:supp-11}
Let $k = \lfloor\frac{c\log(ed)}{\log(eN/c\log(ed))}\rfloor$ for some absolute constant $c>1$, and suppose 
$N\geq cs\log(ed)$, then, we have
\[
\l(\sum_{i=1}^k\l|\wt\phi_i^{\sharp}\r|^2\r)^{1/2}
\leq C\nu_q^{3/2}w\sqrt{e\log(ed)},
\]
with probability at least $1 - c' (eN)^{-\frac{q}{12}+1}(\log(eN))^{\frac q6}w^{-\frac q6}$ for any $w>6$ and some absolute constant $C, c'>1$.
\end{lemma}

\begin{proof}[Proof of Lemma \ref{lem:supp-11}]
First of all, using Binomial estimates, we have for any $i$, and any positive constant $c_i$,
\begin{align*}
Pr\l(\l|\wt\phi_i^{\sharp}\r|\geq c_i\|\wt \phi_i\|_{L_p}\r)
&\leq {N \choose i}Pr(\l|\wt\phi_i\r|\geq c_k\|\wt \phi_i\|_{L_p})^i\\
\leq& \l(\frac{eN}{i}\r)^iPr(\l|\wt\phi_i\r|\geq c_k\|\wt \phi_i\|_{L_p})^i\\
\leq&\l(\frac{eN}{i}\r)^i \frac{\expect{\l|\wt\phi_i\r|^p}^i}{c_i^{pi}\l\|\wt\phi_i\r\|_{L_p}^{pi}}
=\l(\frac{eN}{i}\r)^ic_i^{-pi},
\end{align*}
where we define $\l\|\wt\phi_i\r\|_{L_p}:=\expect{\l|\wt\phi_i\r|^p}^{1/p}$ and $p>2$ is a chosen positive constant. Then, we choose $c_i:=\frac{w}{\log(eN/i)}\l(\frac{eN}{i}\r)^{\frac12}$, which implies
\[
Pr\l(\l|\wt\phi_i^{\sharp}\r|\geq \frac{w}{\log(eN/i)}\l(\frac{eN}{i}\r)^{1/2}\|\wt \phi_i\|_{L_p}\r)
\leq \l(\frac{i}{eN}\r)^{i\l(\frac{p}{2}-1\r)}w^{-pi}\l(\log(eN/i)\r)^{pi}.
\]
Thus, it follows,
\begin{multline}\label{eq:inter-111}
\sum_{i=1}^k\l| \wt\phi_i^{\sharp} \r|^2\leq \sum_{i=1}^N\l| \wt\phi_i \r|^2\leq \sum_{i=1}^N\frac{w^2}{(\log(eN/i))^2}\l(\frac{eN}{i}\r)\|\wt \phi_i\|_{L_p}^2\\
\leq w^2\|\wt \phi_i\|_{L_p}^2eN \int_0^N\frac{1}{x(\log(eN)-\log x)^2}dx 
\leq Cw^2\|\wt \phi_i\|_{L_p}^2eN
\end{multline}
with probability at least 
$$
1 - \sum_{i=1}^N\l(\frac{i}{eN}\r)^{i\l(\frac{p}{2}-1\r)}w^{-pi}\l(\log(eN/i)\r)^{pi}.
$$
Note that for $w>6$ and $p$ chosen to be $p := q/6>10/3$, the above sum is a geometrically decreasing sequence, specifically, it is easy to verify that 
$\l(\frac{i}{eN}\r)^{\l(\frac{p}{2}-1\r)}w^{-p}\l(\log(eN/i)\r)^{p}<(6/5)^{-p},~\forall i \in\{1,2,3,4,\cdots,N\}$.
Thus, 
it follows the above probability is at least 
\[
1 - c'\l(eN\r)^{-\l(\frac{p}{2}-1\r)}\l(\log(eN)\r)^{p}w^{-p},
\]
for some absolute constant $c'>1$. Now, we bound the term $\|\wt \phi_i\|_{L_p}$. We choose $p= \frac{q}{6}$. Then, under the condition that $q>20$, $p=\frac q6>2$, and $\expect{|x_{ij}|^{6p}}<\infty,~\forall i\in\{1,2,\cdots,N\},~j\in\{1,2,\cdots,d\}$. Furthermore, without loss of generality, assume the first $s$ coordinates of $\theta_*$ is non-zero. Then, we have
\[
\|\wt \phi_i\|_{L_p} = \|\dotp{\mf x_i-\wt{\mf x}_i}{\theta_*}\|_{L_p}\leq \l\| \sqrt{\sum_{n=1}^s(x_{in}-\wt x_{in})^2}  \r\|_{L_p}
\leq   \sqrt{\sum_{n=1}^s\l\| (x_{in}-\wt x_{in})^2  \r\|_{L_p}},
\]
where the last inequality follows from Jensen's inequality and then triangle inequality. Now, for each $n$, we have
\begin{multline*}
\l\| (x_{in}-\wt x_{in})^2  \r\|_{L_p}\leq\l\| x_{in}^2\cdot1_{\l\{|x_{in}|>\tau\r\}}  \r\|_{L_p} 
\leq \expect{x_{in}^{2p}\cdot1_{\l\{|x_{in}|>\tau\r\}}}^{1/p}\\
\leq \expect{x_{in}^{6p}}^{1/3p}Pr(|x_{in}|>\tau)^{2/3p}\leq  \expect{x_{in}^{6p}}^{1/3p}\l(\frac{\expect{x_{in}^{6p}}}{\tau^{6p}}\r)^{2/3p},
\end{multline*}
where the second from the last inequality follows from Holder's inequality and the last inequality follows from Markov inequality. Thus, we obtain,
\[
\|\wt \phi_i\|_{L_p} \leq \sqrt{\sum_{n=1}^s\frac{\expect{x_{in}^{6p}}^{1/p}}{\tau^4}}\leq C\nu_q^3\frac{\sqrt{s}}{\tau^2}\leq C\nu_q^3\sqrt{\frac{\log(ed)}{N}},
\]
for some constant $C$ and $\tau = \l( \frac{N}{\log(ed)} \r)^{1/4}\geq s^{1/4}$. Overall, substituting the above bound into \eqref{eq:inter-111}, we have with probability at least $1 - c'\l(eN\r)^{-\l(\frac{p}{2}-1\r)}\l(\log(eN)\r)^{p}w^{-p}$, where $p=q/6$,
\[
\sum_{i=1}^k\l| \wt\phi_i^{\sharp} \r|^2\leq C\nu_q^3w^2eN\cdot \frac{\log(ed)}{N} = C\nu_q^3w^2e\log(ed),
\]
for some constant $C>1$.
\end{proof}

\begin{lemma}\label{lem:supp-12}
Let $k = \lfloor\frac{c\log(ed)}{\log(eN/c\log(ed))}\rfloor$ for some absolute constant $c>1$, and suppose 
$N\geq cs\log(ed)$, then, we have
$$\max_{j\in\{1,2,\cdots,d\}}\l(\sum_{i=1}^k\l|\wt x_{ij}^{\sharp} \r|^2\r)^{1/2}
\leq C\l(\nu_q^2\log(ed) + \nu_q^2\sqrt{\beta}\log(ed)+\sqrt{\frac{N}{\log(ed)}}(\beta+\log(ed))\r)^{1/2},$$
with probability at least $1-e^{-\beta}$ for any $\beta>1$ and some constant $C>1$.
\end{lemma}
\begin{proof}[Proof of Lemma \ref{lem:supp-12}]
First, for any set of $k$ random variables $x_{1j},~x_{2j},~\cdots,~x_{kj}$ we have by Bernstein's inequality,
\[
Pr\l( \sum_{i=1}^k\l|\wt x_{ij} \r|^2\geq k\expect{\wt x_{ij}^2} + C\l( \sqrt{2\sigma_2^2kt}+b_2t\r) \r)\leq e^{-t},
\]
for some constant $C$,
where $\sigma_2^2 := \expect{  \l(\wt x_{ij}^2-\expect{\wt x_{ij}^2} \r)^2  }\leq \expect{x_{ij}^4}\leq\nu_q^4$, $b_2:= \l( N/\log(ed) \r)^{1/2}$ and 
$\expect{\wt x_{ij}^2}\leq \expect{x_{ij}^2}\leq \nu_q^2$. Take a union bound over all ${N \choose k}$ different combinations from $x_{1j},~x_{2j},\cdots,~x_{Nj}$, we obtain,
\[
Pr\l( \sum_{i=1}^k\l|\wt x_{ij}^{\sharp} \r|^2\geq k\expect{\wt x_{ij}^2} + C\l(\sqrt{2\sigma_2^2kt}+b_2t \r) \r)\leq {N \choose k}e^{-t}\leq \l(\frac{eN}{k}\r)^ke^{-t}.
\]
Taking a union bound over all $j\in\{1,2,\cdots,d\}$, we get
\[
Pr\l( \max_{j\in\{1,2,\cdots,d\}}\sum_{i=1}^k\l|\wt x_{ij}^{\sharp} \r|^2\geq k\expect{\wt x_{ij}^2} + C\l( \sqrt{2\sigma_2^2kt}+b_2t \r) \r)\leq
d\l(\frac{eN}{k}\r)^ke^{-t}
\]
Substituting the definition of $k =\lfloor\frac{c\log(ed)}{\log(eN/\log(ed))}\rfloor\leq \frac{c\log(ed)}{\log(eN/\log(ed))}$, we get 
\begin{multline*}
d\l(\frac{eN}{k}\r)^ke^{-t} = \exp\l(-t + k\log(eN/k) + \log d\r)\\
\leq \exp\l(-t + \frac{c\log(ed)}{\log(eN/c\log(ed))}\log\l(\frac{eN}{c\log(ed)}\cdot \log\l(\frac{eN}{c\log(ed)}\r)\r)  + \log d\r)\\
\leq\exp(-t+(2c+1)\log(ed)).
\end{multline*}
Setting $\beta = t - (2c+1)\log(ed)$ and rearranging the terms gives the claim.
\end{proof}

\begin{lemma}\label{lem:supp-13}
Let $k = \lfloor\frac{c\log(ed)}{\log(eN/c\log(ed))}\rfloor$ for some absolute constant $c>1$, and suppose 
$N\geq cs\log(ed)$, then, we have with probability at least $1-c'u^{-q/4}(ed)^{-c}$, for some absolute constant $c'>0$,
\[
\l(\sum_{i> k}\l|\wt\phi_i^{\sharp}\r|^{2r}\r)^{1/2r}\leq  Cu \nu_q^2 N^{1/2r},
\]
for $5/4\leq r<q/16$, any $u>2$, and some absolute constant $C>0$. 
\end{lemma}
\begin{proof}[Proof of Lemma \ref{lem:supp-13}]
Let $p = q/4$, then, $p>4r$. 
Using Binomial estimates, we have for any $i>k$, and any $\alpha>0$,
\[
Pr\l(\l|\wt\phi_i^{\sharp}\r| > \alpha\r)\leq {N\choose i} Pr(|\wt\phi_i|>\alpha)^i\leq   {N\choose i}\l(\frac{\expect{|\wt\phi_i|^{p}}}{\alpha^{p}}\r)^i
\leq \l(\frac{eN}{ i}\frac{\expect{|\wt\phi_i|^{p}}}{\alpha^{p}}\r)^i,
\]
where the second inequality follows from Markov inequality. We choose $\alpha = \|\wt\phi\|_{L_{p}}u\l(\frac{eN}{i}\r)^{2/p}$ and get
\[
Pr\l(\l|\wt\phi_i^{\sharp}\r| > \|\wt\phi\|_{L_{p}}u\l(\frac{eN}{i}\r)^{2/p} \r)\leq u^{-pi}\l(\frac{eN}{i}\r)^{-i}.
\]
Thus, it follows
\begin{multline*}
Pr\l(\exists i >k,~s.t. \l|\wt\phi_i\r|> \|\wt\phi\|_{L_{p}}u\l(\frac{eN}{i}\r)^{2/p} \r)\leq \sum_{i>k}u^{-pi}\l(\frac{eN}{i}\r)^{-i}\\
\leq c'u^{-(k+1)p}\l(\frac{eN}{k+1}\r)^{-(k+1)}\leq c'u^{-p}\l(\frac{eN}{k+1}\r)^{-(k+1)},
\end{multline*}
for some absolute constant $c'>0$, 
where the second from the last inequality follows from the fact that for any $u>2$, the summand is a geometrically decreasing sequence since $N\geq i$. Plugging in $k+1 \geq \frac{c\log(ed)}{\log(eN/\log(ed))}$ and using the fact that $N\geq k+1$ give 
\begin{multline*}
\l(\frac{eN}{k+1}\r)^{-(k+1)} \leq \exp\l(-\frac{c\log(ed)}{\log(eN/c\log(ed))}  \log\l(\frac{eN}{c\log(ed)}\log\l(\frac{eN}{c\log(ed)}\r)\r)    \r)\\
\leq \exp(-c\log(ed)) = (ed)^{-c},
\end{multline*}
Thus, it follows with probability at least $1- c_0u^{-p}(ed)^{-c}$, we have 
\begin{equation}\label{eq:inter-111}
\l(\sum_{i>k}\l|\wt\phi_i^{\sharp}\r|^{2r}\r)^{1/2r}\leq \|\wt\phi\|_{L_{p}}u\l(\sum_{i>k}\l(\frac{eN}{i}\r)^{4r/p} \r)^{1/2r}
\end{equation}
Since $p = q/4>4r$, it follows 
\[
\sum_{i>k}\l(\frac{1}{i}\r)^{4r/p}\leq \int_0^N\l(\frac{1}{x}\r)^{4r/p}dx = \frac{1}{1-4r/p}N^{1-\frac{4r}{p}}.
\]
Thus, with probability at least $1- c_0u^{-q/4}(ed)^{-c}$, 
\begin{equation}\label{eq:inter-13}
\l(\sum_{i>k}\l|\wt\phi_i^{\sharp}\r|^{2r}\r)^{1/2r}\leq C\|\wt\phi\|_{L_{p}}u N^{1/2r},
\end{equation}
for some constant $C$. It remains to bound $\|\wt\phi\|_{L_{p}}$. Again, without loss of generality, we assume the first $s$ entries of $\theta_*$ is non-zero,
\[
\|\wt \phi_i\|_{L_p} = \|\dotp{\mf x_i-\wt{\mf x}_i}{\theta_*}\|_{L_p}\leq \l\| \sqrt{\sum_{n=1}^s(x_{in}-\wt x_{in})^2}  \r\|_{L_p}
\leq   \sqrt{\sum_{n=1}^s\l\| (x_{in}-\wt x_{in})^2  \r\|_{L_p}},
\]
where the last inequality follows from Jensen's inequality and then triangle inequality. Now, for each $n$, we have
\begin{multline*}
\l\| (x_{in}-\wt x_{in})^2  \r\|_{L_p}\leq\l\| x_{in}^2\cdot1_{\l\{|x_{in}|>\tau\r\}}  \r\|_{L_p} 
\leq \expect{x_{in}^{2p}\cdot1_{\l\{|x_{in}|>\tau\r\}}}^{1/p}\\
\leq \expect{x_{in}^{4p}}^{1/2p}Pr(|x_{in}|>\tau)^{1/2p}\leq  \expect{x_{in}^{4p}}^{1/2p}\l(\frac{\expect{x_{in}^{4p}}}{\tau^{4p}}\r)^{1/2p},
\end{multline*}
where we use the fact that $p=q/4$ and thus $\expect{x_{in}^{4p}}$ is bounded. The second from the last inequality follows from Holder's inequality and the last inequality follows from Markov inequality. Using the fact that $\tau^2 = \sqrt{\frac{N}{\log(ed)}}\geq \sqrt{s}$, we get $\|\wt \phi_i\|_{L_p}\leq \nu_q^2$ . Combining this bound with \eqref{eq:inter-13} finishes the proof.
\end{proof}

\begin{lemma}\label{lem:supp-14}
Let $k = \lfloor\frac{c\log(ed)}{\log(eN/c\log(ed))}\rfloor$ for some absolute constant $c>1$, and suppose 
$N\geq cs\log(ed)$, then, we have with probability at least $1-c'u^{-q}(ed)^{-(c-1)}$, for some absolute constant $c'>0$.
$$\max_{j\in\{1,2,\cdots,d\}}\l(\sum_{i>k}\l|\wt x_{ij}^{\sharp} \r|^{2r'}\r)^{1/2r'}
\leq Cu\nu_qN^{1/2r'},$$
for some constant absolute constant $C>0$ and $r'\in(\frac{q}{q-16},5]$.
\end{lemma}
\begin{proof}
First, following the same procedure as that of Lemma \ref{lem:supp-13} up to \eqref{eq:inter-111}, with $p = q$, we have with probability at least $1-c'u^{-q}(ed)^{-c}$,
\[
\l(\sum_{i>k}\l|\wt x_{ij}^{\sharp}\r|^{2r'}\r)^{1/2r'}\leq \|\wt x_{ij}\|_{L_{q}}u\l(\sum_{i>k}\l(\frac{eN}{i}\r)^{4r'/q} \r)^{1/2r'}.
\]
Note that  $\|\wt x_{ij}\|_{L_{q}}\leq\| x_{ij}\|_{L_{q}}\leq\nu_q$ by the assumption and $r'\in(\frac{q}{q-16},5]$, thus, $4r'/q<1$ and we have with probability at least $1-u^{-q}(ed)^{-c}$,
\[
\l(\sum_{i>k}\l|\wt x_{ij}^{\sharp}\r|^{2r'}\r)^{1/2r'}\leq Cu\nu_qN^{1/2r'}.
\]
Finally, taking a union bound over all $j\in\{1,2,\cdots,d\}$ finishes the proof.
\end{proof}

Finally, substituting Lemma \ref{lem:supp-11},~\ref{lem:supp-12},~\ref{lem:supp-13},~\ref{lem:supp-14} into \eqref{eq:master-bound-2} with $r=5/4, r'=5$ gives with probability at least $1-e^{-\beta}-e^{-v^2}-c'\l(u^{-q}(ed)^{-(c-1)}+u^{-q/4}(ed)^{-c}+e^{-\frac{q}{12}}N^{-\frac{q}{12}+1}(\log(eN))^{q/6}w^{-q/6}\r)$,
\begin{multline}\label{eq:master-bound-3.5}
\max_{j\in\{1,2,\cdots,d\}}\l|\sum_{i=1}^N\varepsilon_i\wt \phi_i \wt x_{ij}  \r|\\ 
\leq C\l(\nu_q^3+\nu_q^{5/2}+\nu_q^{3/2}\r)w\l( \log(ed)\beta^{1/4} + N^{1/4}(\log(ed))^{3/4}\beta^{1/2} + vu^2\sqrt{N\log d}  \r).
\end{multline}
Similarly, one can show that the Rademacher sum  $\max_{j\in\{1,2,\cdots,d\}}\l|\sum_{i=1}^N\varepsilon_i(\wt y_i-y_i)\wt x_{ij} \r|$ satisfies the same bound.

\item \textbf{Bounding the term} $\max_{j\in\{1,2,\cdots,d\}} \l|\sum_{i=1}^N\varepsilon_i \xi_i \wt x_{ij}  \r|$: 

The proving techniques in this part is essentially the same as that of part 1 but with a slight change of exponents when applying Holder's inequality adapting to the moment condition of the noise $\xi_i$.
Similar as before, one can employ the inequality from \citep{montgomery1990distribution}, conditioned on $\mf{x}_i$, which results in
\[
\l|\sum_{i=1}^N\varepsilon_i \xi_i \wt x_{ij}  \r|
\leq
\sum_{i=1}^k\l|\xi_i^{\sharp}\wt x_{ij}^{\sharp} \r|
+v\l( \sum_{i>k}\l|\xi_i^{\sharp}\wt x_{ij}^{\sharp} \r|^2 \r)^{1/2},
\]
with probability at least $1-e^{-v^2}$, where $k$ is any chosen integer within $\l\{0,1,2,\cdots,N \r\}$ and $\l(\xi_i^{\sharp}\r)_{i=1}^N$, $\l(\wt x_{ij}^{\sharp}\r)_{i=1}^N$ are non-increasing rearrangements of  $\l(|\xi_i|\r)_{i=1}^N$, $\l(|\wt x_{ij}|\r)_{i=1}^N$. We define the former sum to be 0 when $k=0$. 
By Holder's inequality, we have
\begin{equation*}
\l|\sum_{i=1}^N\varepsilon_i\xi_i \wt x_{ij}  \r|
\leq \l(\sum_{i=1}^k\l|\xi_i^{\sharp}\r|^4\r)^{1/4}\l(\sum_{i=1}^k\l|\wt x_{ij}^{\sharp} \r|^{4/3}\r)^{3/4}
+v\l( \sum_{i>k}\l|\xi_i^{\sharp} \r|^{2r} \r)^{1/(2r)} 
\l(\sum_{i>k}\l|\wt x_{ij}^{\sharp}\r|^{2r'}\r)^{1/(2r')},
\end{equation*}
for some positive exponents $r,r'$ such that $\frac1r+\frac{1}{r'}=1$. Take a union bound for all $j\in\{1,2,\cdots,d\}$, gives with probsability at least $1-e^{-v^2}$,
\begin{multline}\label{eq:master-bound-3}
\max_{j\in\{1,2,\cdots,d\}}\l|\sum_{i=1}^N\varepsilon_i\xi_i \wt x_{ij}  \r|
\leq \l(\sum_{i=1}^k\l|\xi_i^{\sharp}\r|^4\r)^{1/4}
\max_{j\in\{1,2,\cdots,d\}}\l(\sum_{i=1}^k\l|\wt x_{ij}^{\sharp} \r|^{4/3}\r)^{3/4}\\
+v\sqrt{\log d}\l( \sum_{i>k}\l|\xi_i^{\sharp} \r|^{2r} \r)^{1/(2r)} 
\max_{j\in\{1,2,\cdots,d\}}\l(\sum_{i>k}\l|\wt x_{ij}^{\sharp}\r|^{2r'}\r)^{1/(2r')},
\end{multline}
Again, our goal is to bound the four terms in \eqref{eq:master-bound-3} separately.

\begin{lemma}\label{lem:xi-1}
Let $k = \lfloor\frac{c\log(ed)}{\log(eN/c\log(ed))}\rfloor$ for some absolute constant $c>1$, and suppose 
$N\geq cs\log(ed)$, then, we have
\[
\l(\sum_{i=1}^k\l|\xi_i^{\sharp}\r|^4\r)^{1/4}
\leq C\|\xi\|_{L_{q'}}w N^{1/4},
\]
with probability at least $1 - c' (eN)^{-\frac{q'}{4}+1}(\log(eN))^{\frac {q'}{2}}w^{-q'}$ for any $w>4$ and some absolute constant $C, c'>1$, where $q'>5$ is defined in Assumption \ref{assume:sparse-recovery}.
\end{lemma}
\begin{proof}[Proof of Lemma \ref{lem:xi-1}]
First of all, by Markov inequality, 
\begin{multline*}
Pr\l(\l|\xi_i^{\sharp}\r|\geq c_i\|\xi\|_{L_{q'}}\r)\leq {N \choose i}Pr\l(\l|\xi_i\r|\geq c_k\|\xi\|_{L_{q'}}\r)^i\\
\leq \l(\frac{eN}{i}\r)^iPr\l(\l|\xi_i\r|\geq c_k\|\xi\|_{L_{q'}}\r)^i
\leq  \l(\frac{eN}{i}\r)^i \frac{\expect{\l|\xi_i\r|^{q'}}^i}{c_i^{q'i}\l\|\xi\r\|_{L_{q'}}^{q'i}}
=\l(\frac{eN}{i}\r)^ic_i^{-q'i}.
\end{multline*}
Choosing $c_i = w(eN/i)^{1/4}(\log(eN/i))^{1/2}$ gives
\[
Pr\l( \l|\xi_i^{\sharp}\r|\geq \l(\frac{eN}{i}\r)^{1/4}\frac{w}{(\log(eN/i))^{1/2}}\l\|\xi\r\|_{L_{q'}} \r)
\leq \l( \frac{i}{eN} \r)^{i(\frac{q'}{4}-1)}w^{-q'i}\l(\log\frac{eN}{i}\r)^{\frac{q'}{2}i}.
\]
Thus, it follows
\[
\sum_{i=1}^k\l|\xi_i^{\sharp}\r|^4\leq \sum_{i=1}^N\l|\xi_i^{\sharp}\r|^4\leq \sum_{i=1}^N\frac{eN}{i}\frac{w^4}{(\log(eN/i))^2}\l\|\xi\r\|_{L_{q'}} 
\leq Cw^4\l\|\xi\r\|_{L_{q'}}eN, 
\]
with probability at least 
\[
1 - \sum_{i=1}^N\l( \frac{i}{eN} \r)^{i(\frac{q'}{4}-1)}w^{-q'i}\l(\log\l(\frac{eN}{i}\r)\r)^{\frac{q'}{2}i}.
\]
Since for any $w>4$ and $q'>5$, the above summand is a geometrically decreasing sequence. Specifically, it is easy to show that 
$\l( \frac{i}{eN} \r)^{(\frac{q'}{4}-1)}w^{-q'}\l(\log\l(\frac{eN}{i}\r)\r)^{\frac{q'}{2}}<\l(4/\sqrt{10}\r)^{-q'},~\forall i\in\{1,2,\cdots,N\}$. 
Thus, it follows the probability is at least
\[
1- c'\l( eN \r)^{-(\frac{q'}{4}-1)}w^{-q'}\l(\log\l(eN\r)\r)^{\frac{q'}{2}}
\]
for some absolute constant $c'>0$.
\end{proof}

\begin{lemma}\label{lem:x-1}
Let $k = \lfloor\frac{c\log(ed)}{\log(eN/c\log(ed))}\rfloor$ for some absolute constant $c>1$, then, we have
$$\max_{j\in\{1,2,\cdots,d\}}\l(\sum_{i=1}^k\l|\wt x_{ij}^{\sharp} \r|^{4/3}\r)^{3/4}
\leq C\l(\nu_q^{4/3}\log(ed) + \nu_q^{4/3}\sqrt{\beta}\log(ed)+\l(\frac{N}{\log(ed)}\r)^{1/3}(\beta+\log(ed))\r)^{3/4},$$
with probability at least $1-e^{-\beta}$ for any $\beta>1$ and some constant $C>1$.
\end{lemma}
\begin{proof}[Proof of Lemma \ref{lem:x-1}]
First, for any set of $k$ random variables $x_{1j},~x_{2j},~\cdots,~x_{kj}$ we have by Bernstein's inequality,
\[
Pr\l( \sum_{i=1}^k\l|\wt x_{ij} \r|^{4/3}\geq k\expect{|\wt x_{ij}|^{4/3}} + C\l( \sqrt{2\sigma_2^2kt}+b_2t\r) \r)\leq e^{-t},
\]
for some constant $C$,
where $\sigma_2^2 := \expect{  \l(|\wt x_{ij}|^{4/3}-\expect{|\wt x_{ij}|^{4/3}} \r)^2  }\leq \expect{|x_{ij}|^{8/3}}\leq\nu_q^{8/3}$, $b_2:= \l( N/\log(ed) \r)^{1/3}$ and 
$\expect{|\wt x_{ij}|^{4/3}}\leq \expect{|x_{ij}|^{4/3}}\leq \nu_q^{4/3}$. Take a union bound over all ${N \choose k}$ different combinations from $x_{1j},~x_{2j},\cdots,~x_{Nj}$, we obtain,
\[
Pr\l( \sum_{i=1}^k\l|\wt x_{ij}^{\sharp} \r|^{4/3}\geq k\expect{|\wt x_{ij}|^{4/3}} + C\l(\sqrt{2\sigma_2^2kt}+b_2t \r) \r)\leq {N \choose k}e^{-t}\leq \l(\frac{eN}{k}\r)^ke^{-t}.
\]
Taking a union bound over all $j\in\{1,2,\cdots,d\}$, we get
\[
Pr\l( \max_{j\in\{1,2,\cdots,d\}}\sum_{i=1}^k\l|\wt x_{ij}^{\sharp} \r|^{4/3}\geq k\expect{|\wt x_{ij}|^{4/3}} + C\l( \sqrt{2\sigma_2^2kt}+b_2t \r) \r)\leq
d\l(\frac{eN}{k}\r)^ke^{-t}
\]
Substituting the definition of $k =\lfloor\frac{c\log(ed)}{\log(eN/\log(ed))}\rfloor\leq \frac{c\log(ed)}{\log(eN/\log(ed))}$, we get 
\begin{multline*}
d\l(\frac{eN}{k}\r)^ke^{-t} = \exp\l(-t + k\log(eN/k) + \log d\r)\\
\leq \exp\l(-t + \frac{c\log(ed)}{\log(eN/c\log(ed))}\log\l(\frac{eN}{c\log(ed)}\cdot \log\l(\frac{eN}{c\log(ed)}\r)\r)  + \log d\r)\\
\leq\exp(-t+(2c+1)\log(ed)).
\end{multline*}
Setting $\beta = t - (2c+1)\log(ed)$ and rearranging the terms gives the claim.
\end{proof}

\begin{lemma}\label{lem:xi-2}
Let $k = \lfloor\frac{c\log(ed)}{\log(eN/c\log(ed))}\rfloor$ for some absolute constant $c>1$, then, we have with probability at least $1-c'u^{-q'}(ed)^{-c}$, for some absolute constant $c'>0$,
\[
\l(\sum_{i> k}\l|\xi_i^{\sharp}\r|^{2r}\r)^{1/2r}\leq  Cu \|\xi\|_{L_{q'}} N^{1/2r},
\]
for $r\leq5/4$, any $u>2$, and some absolute constant $C>0$. 
\end{lemma}

\begin{proof}
Following from the same proof as that of Lemma \ref{lem:supp-13} up to \eqref{eq:inter-111} with $p=q'$, we have with probability at least $1- c_0u^{-q'}(ed)^{-c}$,  
\begin{equation}
\l(\sum_{i>k}\l|\xi_i^{\sharp}\r|^{2r}\r)^{1/2r}\leq \|\xi\|_{L_{q'}}u\l(\sum_{i>k}\l(\frac{eN}{i}\r)^{4r/q'} \r)^{1/2r}.
\end{equation}
Since $q'>5\geq4r$ by assumption, it follows,
\[
\sum_{i>k}\l(\frac{1}{i}\r)^{4r/q'}\leq \int_0^N\l(\frac{1}{x}\r)^{4r/q'}dx = \frac{1}{1-4r/q'}N^{1-\frac{4r}{q'}},
\]
which implies the claim.
\end{proof}

Also, by Lemma \ref{lem:supp-14}, we have with probability at least $1-c'u^{-q}(ed)^{-(c-1)}$, for some absolute constant $c'>0$.
\begin{equation}\label{eq:x-2}
\max_{j\in\{1,2,\cdots,d\}}\l(\sum_{i>k}\l|\wt x_{ij}^{\sharp} \r|^{2r'}\r)^{1/2r'}
\leq Cu\nu_qN^{1/2r'},
\end{equation}
for some constant absolute constant $C>0$ and $r'\in(\frac{q}{q-16},5]$.

Overall, substituting Lemma \ref{lem:xi-1}, \ref{lem:x-1}, \ref{lem:xi-2}, and \eqref{eq:x-2} into \eqref{eq:master-bound-3} with $r = 5/4, r' = 5$ gives with probability at least 
$1- e^{-\beta} - e^{-v^2} - c'\l( (eN)^{-(\frac{q'}{4}-1)}(\log(eN))^{q'/2}w^{-q'} +u^{-q}(ed)^{-(c-1)} + u^{-q'}(ed)^{-c} \r)$,
\begin{multline}\label{eq:master-bound-4}
\max_{j\in\{1,2,\cdots,d\}}\l|\sum_{i=1}^N\varepsilon_i\xi_i \wt x_{ij}  \r|
\leq C\|\xi\|_{L_{q'}}\l( vu^2\nu_qN^{1/2}(\log(ed))^{1/2} + w\nu_q(\log(ed))^{3/4}N^{1/4} \r.\\
\l.+ w\nu_q\beta^{3/8}N^{1/4}(\log(ed))^{3/4} + w\beta^{3/4}N^{1/2}(\log(ed))^{1/2} \r)
\end{multline}

\end{enumerate}

Overall, substituting the bounds \eqref{eq:master-bound-3.5} and \eqref{eq:master-bound-4} into \eqref{eq:decomp-1} gives
\begin{multline*}
N\cdot\max_{j\in\{1,2,\cdots,d\}}|z_j| \leq 
C\l(\nu_q^3+\nu_q^{5/2}+\nu_q^{3/2}\r)w\l( \log(ed)\beta^{1/4} + N^{1/4}(\log(ed))^{3/4}\beta^{1/2} 
+ vu^2\sqrt{N\log d}  \r) \\
+ C\|\xi\|_{L_{q'}}\l( \nu_q +1 \r)(vu^2+w+w\beta^{3/8}+w\beta^{3/4})\l(\sqrt{\beta N\log(ed)} + \beta N^{1/4}\l(\log(ed)\r)^{3/4} \r),
\end{multline*}
with probability at least 
\begin{multline*}
1-2e^{-\beta}-2e^{-v^2}
-c'\l(u^{-q}(ed)^{-(c-1)}+(u^{-q/4}+u^{-q'})(ed)^{-c}\r.\\
\l.+(eN)^{-\frac{q}{12}+1}(\log(eN))^{q/6}w^{-q/6} +  (eN)^{-(\frac{q'}{4}-1)}(\log(eN))^{q'/2}w^{-q'}\r).
\end{multline*}
This implies the claim when combining \eqref{eq:master-bound-0} and the fact that $N\geq cs\log(ed)$.
\end{proof}

The following lemma gives a bound on $r_{\mathcal{M}}$ in terms of $\rho$.
\begin{lemma}\label{lem:bound-M}
Suppose $N\geq cs\log(ed)$ for some constant $c>1$ and Assumption \ref{assume:sparse-recovery} holds, then, we have
\[
r_{\mathcal{M}}\leq C(\nu_q,\xi)\l( \frac{wu^2v+w\beta^{3/4}}{\delta^2Q}\sqrt{\frac{s\log(ed)}{N}} + \sqrt{\rho\frac{wu^2v+w\beta^{3/4}}{\delta^2Q}}\l(\frac{s\log(ed)}{N}\r)^{1/4} \r),
\]
when taking 
\begin{multline*}
p_{\mathcal{M}} = 2e^{-\beta}+2e^{-v^2}
+c'\l((u^{-q/4}+u^{-q'})(ed)^{-(c-1)}\r.\\
\l.+(eN)^{-\frac{q}{12}+1}(\log(eN))^{q/6}w^{-q/6} +  (eN)^{-\frac{q'}{4}+1}(\log(eN))^{q'/2}w^{-q'}\r),
\end{multline*}
for some absolute constant $c'>1$ and any $\beta,u,v,w>6$,
where $C(\nu_q,\xi) := C\l(\nu_q^3+\nu_q^{5/2}+\nu_q^{3/2} + \|\xi\|_{L_{q'}}(\nu_q+1)\r)$, for some absolute constant $C>0$.
\end{lemma}

\begin{proof}[Proof of Lemma \ref{lem:bound-M}]
Since $\Lambda_M = \frac{\delta^2Q}{64}$, let $m=s$ in Lemma \ref{lem:bound-PM} and the infimum of the $r>0$ such that the right hand side of Lemma \ref{lem:bound-PM} is less than $\frac{\delta^2Q}{64}r^2$ can be achieved by setting the right hand side equal to $\frac{\delta^2Q}{64}r^2$, which gives,
\[
\frac{\delta^2Q}{64}r^2 = C(\nu_q,\xi)(wu^2v+w\beta^{3/4})\sqrt{\frac{\log (ed)}{N}} \l( r\sqrt s  + \rho \r).
\]
Solving the above quadratic equation gives
\[
r = C(\nu_q,\xi)\l( \frac{wu^2v+w\beta^{3/4}}{\delta^2Q}\sqrt{\frac{s\log(ed)}{N}} + \sqrt{\rho\frac{wu^2v+w\beta^{3/4}}{\delta^2Q}}\l(\frac{s\log(ed)}{N}\r)^{1/4} \r).
\]
Thus, the defined $r_{\mathcal{M}}$ must be bounded above by this value and the lemma is proved.
\end{proof}

\subsubsection{Bounding the radius $r_{\mathcal{V}}$}
\begin{lemma}\label{lem:bound-V}
Suppose $N\geq s\log(ed)$, then, 
\[
r_{\mathcal{V}}\leq 8\l(\nu_q^2+\nu_q^4\r)  \l( \frac{\rho}{\delta^2Q} \r)^{1/2}\l( \frac{\log(ed)}{N} \r)^{1/4},
\]
\end{lemma}

\begin{proof}[Proof of Lemma \ref{lem:bound-V}]
First of all,
\[
\sup_{\theta\in B_{2}(\theta_*,r)\cap B_\Psi(\theta_*,\rho)}  
\l|\mc{V}_{\theta-\theta_*}\r| :=
\sup_{\mf v\in B_{2}(0,r)\cap B_\Psi(0,\rho)} \expect{\l(\wt{y} - \dotp{\wt {\mf x}}{\eta\theta_*}\r)\dotp{\wt{\mf{x}}}{\mf{v}}}. 
\]
For each $\mf v$, we have
\begin{multline*}
\expect{\l(\wt{y} - \dotp{\wt {\mf x}}{\eta\theta_*}\r)\dotp{\wt{\mf{x}}}{\mf{v}}}
=|\expect{\l(\wt{y} - y\r)\dotp{\wt{\mf{x}}}{\mf{v}}}| 
+| \expect{\l(y-\dotp{\mf x}{\theta_*}\r)\dotp{\wt{\mf{x}}}{\mf{v}}}| +
|\expect{\dotp{\mf x - \wt{\mf x}}{\theta_*}\dotp{\wt{\mf{x}}}{\mf{v}}}|\\
\leq \rho\|\expect{\l(\wt{y} - y\r)\wt{\mf{x}}}\|_{\infty} +
\rho\|\expect{\dotp{\mf x - \wt{\mf x}}{\theta_*}\wt{\mf{x}}}\|_{\infty},
\end{multline*}
where we use the fact that $y-\dotp{\mf x}{\theta_*}=\xi$ is independent of $\mf x$. 
Note that for any $j\in\{1,2,\cdots,d\}$,
\begin{multline*}
|\expect{\l(\wt{y} - y\r)\wt{x}_j}|\leq \expect{\l|y\r|\cdot\l|\wt{x}_j\r|1_{\{|y|>\tau\}}}\leq\expect{\l|y\r|^2\cdot\l|\wt{x}_j\r|^2}^{1/2}
Pr(|y|>\tau)^{1/2}  \\
\leq  \expect{\l|y\r|^2\cdot\l|\wt{x}_j\r|^2}^{1/2}  \l( \frac{\expect{|y|^4}}{\tau^4}\r)^{1/2}
\leq \nu_q^4\sqrt{\frac{\log(ed)}{N}},
\end{multline*}
for some constant $C_1>0$. Next, Let $\mathcal{G}_s$ be the set of nonzero coordinates of $\theta_*$ and $\mathcal{P}_{\mc G_s}$ be the orthogonal projection onto these coordinates, then,
\begin{multline*}
|\expect{\dotp{\mf x - \wt{\mf x}}{\theta_*}\wt{x}_j}|
=|\expect{\dotp{(\mf x - \wt{\mf x})\wt{x}_j}{\theta_*}}|
\leq \l\|\mathcal{P}_{\mc G_s}\expect{(\mf x - \wt{\mf x})\wt{x}_j} \r\|_2\\
\leq \l(\sum_{i\in\mc G_s}\l(\expect{(x_i - \wt{x}_i)\wt{x}_j}\r)^2\r)^{1/2}
\leq  \l(\sum_{i\in\mc G_s}\expect{|x_i| | x_j| 1_{\{|x_i|>\tau\}}}^2\r)^{1/2}\\
\leq  \l(\sum_{i\in\mc G_s}\expect{|x_i|^2 | x_j|^2} Pr(|x_i|>\tau)\r)^{1/2},
\end{multline*}
where the last inequality follows from Holder's inequality. Now, By Markov inequality  
\[
Pr(|x_i|>\tau)\leq \frac{\expect{|x_i|^8}}{\tau^8}=\expect{|x_i|^8}\l( \frac{\log (ed)}{N} \r)^{2} ,
\]
Thus, it follows,
\[
|\expect{\dotp{\mf x - \wt{\mf x}}{\theta_*}\wt{x}_j}|
\leq \l(\sum_{i\in\mc G_s}\expect{|x_i|^2 | x_j|^2}\expect{|x_i|^8}\l( \frac{\log (ed)}{N} \r)^{2} \r)^{1/2}\leq \nu_q^8\sqrt{\frac{\log(ed)}{N}},
\]
where the last inequality follows from the fact that $N\geq s\log(ed)$ and $\expect{|x_i|^8}$ is bounded. Overall, we get 
\[
\sup_{\theta\in B_{2}(\theta_*,r)\cap B_\Psi(\theta_*,\rho)}  
\l|\mc{V}_{\theta-\theta_*}\r|
\leq \l(\nu_q^4+\nu_q^8\r)\rho\sqrt{\frac{\log(ed)}{N}}
\]
Since $\Lambda_{\mc V}=\frac{\delta^2Q}{64}$, let
\[
\frac{\delta^2Q}{64}r^2 = \l(\nu_q^4+\nu_q^8\r)\rho\sqrt{\frac{\log(ed)}{N}},
\]
which results in $r=8\l(\nu_q^2+\nu_q^4\r)\l( \frac{\rho}{\delta^2Q} \r)^{1/2}\l( \frac{\log(ed)}{N} \r)^{1/4}$ and $r_{\mc V}$ must be bounded above by this value.
\end{proof}

\begin{proof}[Proof of Lemma \ref{lem:sparse-eq}]
Let $\mc G_s$ be the set of nonzero coordinates of $\theta_0$, then, for any vector $\mf v\in B_{2}(0,r)\cap S_\Psi(0,\rho)$, we have
$\mf v = \mc P_{\mc G_s}\mf v + \mc P_{\mc G_s^c}\mf v$ and since $\|\eta\theta_* - \theta_0\|_1\leq\rho/16$, by definition of $\Gamma_{\Psi}(\theta_*,\rho)$ in \eqref{eq:sub-diff},
 there exists a sub-differential $\mf z^*\in\Gamma_{\Psi}(\theta_*,\rho)$ such that $\dotp{\mf z^*}{\theta_0} = \|\theta_0\|_1$ and
$\dotp{\mf z^*}{\mc P_{\mc G_s^c}\mf v} = \|\mc P_{\mc G_s^c}\mf v\|_1$. Thus, it follows,
\begin{multline*}
\dotp{\mf z^*}{\mf v} = \dotp{\mf z^*}{\mc P_{\mc G_s}\mf v} + \dotp{\mf z^*}{\mc P_{\mc G_s^c}\mf v} \geq  \|\mc P_{\mc G_s^c}\mf v\|_1 - \|\mc P_{\mc G_s}\mf v\|_1\\
\geq \|\mf v\|_1 - 2\|\mc P_{\mc G_s}\mf v\|_1 \geq \rho - 2\sqrt{s}\|\mc P_{\mc G_s}\mf v\|_2
\geq  \rho - 2r(\rho)\sqrt{s},
\end{multline*}
where the second from the last inequality follows from $\mf v\in B_{2}(0,r)\cap S_\Psi(0,\rho)$ that $\|\mf v\|_1= \rho$ and the last inequality follows from $\|\mc P_{\mc G_s}\mf v\|_2\leq \|\mf v\|_2\leq r(\rho)$. The above bound is greater than $3\rho/4$ when $\rho\geq 8r(\rho)\sqrt{s}$.
\end{proof}

\section{Single-index model with heavy-tailed elliptical measurements}\label{sec:single-index}
 According to Theorem \ref{thm:master} our goal is to compute $r_{\mathcal{Q}},~r_{\mathcal{M}},~r_{\mc{V}}$ for specific constants 
$\Lambda_Q,~\Lambda_M,~\Lambda_{\mc V}$ satisfying the assumptions and determine the choice of $\rho$ so that 
$\Delta(\eta\theta_*,\rho)\geq \frac34\rho$. We start with some preliminaries on elliptical distributions.

\subsection{Basic properties of elliptical symmetric distribution}

The following elliptical symmetry property, generalizing the well known fact for the conditional distribution of the multivariate Gaussian, plays an important role in our subsequent analysis (see, for example, \citep{goldstein2016structured}, for the proof):

\begin{lemma}
\label{lem:elliptical}
If $\mathbf{x}\sim\mathcal{E}_d(0,\mf I_{d\times d},F_\mu)$, then for any two fixed vectors $\mathbf{y}_1,\mathbf{y}_2\in\mathbb{R}^d$ with $\|\mathbf{y}_2\|_2=1$,
\[\expect{\langle\mathbf{x},\mathbf{y}_1\rangle~|~\langle\mathbf{x},\mathbf{y}_2\rangle}
=\langle\mathbf{y}_1,\mathbf{y}_2\rangle\langle\mathbf{x},\mathbf{y}_2\rangle.\]
\end{lemma}

Furthermore, we need the following lemma which basically states that $\sqrt{d}U$ is a sub-Gaussian random vector:
\begin{lemma}[Lemma 2.2 of \citep{ball1997elementary}]\label{lem:sub-gaussian}
Let $U$ have the uniform distribution on $\mathbb{S}^{d-1}$, then, for any $t\in(0,1)$ and any fixed vector $\mf v\in\mb R^d$,
\[
Pr(|\dotp{U}{\mf v}|\geq t)\leq e^{-dt^2}.
\]
\end{lemma}

\begin{lemma}\label{lem:bound-mu}
Suppose $\mf x$ is an elliptical symmetric vector with the decomposition \eqref{elliptical-definition} satisfying Assumption \ref{assumption:moment} for some $q>4$, then, $\expect{\l|\mu/\sqrt{d}\r|^q}\leq \frac{\nu_q^q}{\kappa^{q/2}}$.
\end{lemma}
\begin{proof}[Proof of Lemma \ref{lem:bound-mu}]
Note that for the $\mf x$ such that $\expect{|x_i|^q}\leq \nu_q^q,~i=1,2,\cdots,d$, we have by \eqref{elliptical-definition},
$\expect{|\mu\dotp{U}{\mf B^T\mf e_i}|^q}\leq\nu_q^q$, for any unit coordinate vector $\mf e_i,~i=1,~2,\cdots,~d$. Since $\mu$ and $U$ are independent, we have $\expect{\l|\mu\r|^q}\expect{|\dotp{U}{\mf B^T\mf e_i}|^q}\leq\nu_q^q$ and thus, 
$$\expect{\l|\mu/\sqrt{d}\r|^q}\leq \nu_q^q \l/\expect{\l|\dotp{\sqrt{d}U}{\mf B^T\mf e_i}\r|^q}\r..$$ 
Since 
$$\expect{\l|\dotp{\sqrt{d}U}{\mf B^T\mf e_i}\r|^q}^{1/q}\geq\expect{\l|\dotp{\sqrt{d}U}{\mf B^T\mf e_i}\r|^2}^{1/2} = \|\mf B^T\mf e_i\|_2\geq\sqrt{\kappa},$$
where the last inequality follows from the non-degeneracy property of the covariance matrix $\mf\Sigma= \mf B\mf B^T$.
\end{proof}

\begin{lemma}\label{lem:normalization}
Suppose $\mf x$ is an elliptical symmetric vector with the decomposition \eqref{elliptical-definition}, then, $\sqrt{d}\mf x/\|\mf x\|_2$ is a sub-Gaussian random vector.
\end{lemma}
\begin{proof}[Proof of Lemma \ref{lem:normalization}]
It is enough to check $\l\|\sqrt{d}\mf x/\|\mf x\|_2  \r\|_{\psi_2}$ is bounded.
For any $p\geq1$ and any $\mf v\in\mb R^d$,
\begin{align*}
\expect{\l| \dotp{\frac{\sqrt{d}\mf x}{\|\mf x\|_2}}{\mf v} \r|^p}^{1/p}=& \expect{\l| \dotp{\frac{\sqrt{d}\mu\mf B U}{|\mu|\|\mf B U\|_2}}{\mf v} \r|^p}^{1/p}\\
=&\expect{\l| \dotp{ \sqrt{d} U}{\mf B^T\mf v}\cdot \frac{1}{\|\mf B U\|_2} \r|^p}^{1/p}\\
\leq&\expect{\l| \dotp{ \sqrt{d} U}{\mf B^T\mf v} \r|^p}^{1/p}\cdot\frac{1}{\sqrt{\lambda_{\min}(\mf{\Sigma})}}\\
\leq& \sqrt{p}\|\sqrt{d}U\|_{\psi_2}\l\|\mf B^T\mf v\r\|_2\cdot\frac{1}{\sqrt{\lambda_{\min}(\mf{\Sigma})}}\leq \sqrt{p} \|\sqrt{d}U\|_{\psi_2}\sqrt{\frac{\lambda_{\max}(\mf{\Sigma})}{\lambda_{\min}(\mf{\Sigma})}}.
\end{align*}
By Lemma \ref{lem:sub-gaussian}, $\sqrt{d}U$ is sub-Gaussian and thus $\|\sqrt{d}U\|_{\psi_2}$ is bounded by a constant. This finishes the proof.
\end{proof}

\subsection{Computing critical radiuses}
Let $\delta$ and $Q$ be the same as that of Lemma \ref{lem:small-ball}, i.e. $\delta = \frac12\sqrt{\frac{\kappa}{2}}$ and $Q=\frac{\kappa^2}{8\psi}$,  then 
we set $\Lambda_Q := \delta^2Q^2/4$, $\Lambda_M:= \delta^2Q^2/64$, $\Lambda_{\mc V}:=  \delta^2Q^2/64$ and
 $\frac{\delta^2Q^2}{8}\frac{r(\rho)^2}{\rho}\leq \lambda\leq \frac{5\delta^2Q^2}{32}\frac{r(\rho)^2}{\rho}$. Then, we have $\Lambda_Q> 2(\Lambda_M+\Lambda_{\mc{V}})+\frac{5\delta^2Q^2}{32}$ 
and $\frac{\delta^2Q^2}{8}\geq4(\Lambda_M+\Lambda_{\mc{V}})$, satisfying the assumptions in Theorem \ref{thm:master}. We aim to bound the critical radiuses $r_{\mathcal{Q}},~r_{\mathcal{M}},~r_{\mc V}$ and show there exists $\rho>0$ such that $\Delta(\theta_*,\rho)\geq\frac34\rho$.

\subsubsection{Bounding the radius $r_{\mathcal{Q}}$}
We start with a version of truncated small-ball for elliptical symmetric vectors which is stronger than Lemma \ref{lem:weak-small-ball} in the last section.

\begin{lemma}\label{lem:small-ball-2}
Suppose $\mf x$ is an elliptical symmetric vector with the decomposition \eqref{elliptical-definition} satisfying Assumption \ref{assumption:moment} for some $q>4$. 
Suppose $N\geq\l.\frac{4}{Q^2}\expect{\mu^2}^2\lambda_{\max}(\mf{\Sigma})\r/d^2$, then, we have for any $\mf v\in\mb R^d$,
\[
Pr(\l| \dotp{\wt{\mf x}}{\mf v} \r|\geq\delta\|\mf v\|_2)\geq Q,
\]
where 
$\wt{\mf x}=\frac{\sqrt{d}\mathbf{x}}{\|\mathbf{x}\|_2}\cdot\l( \frac{\|\mathbf{x}\|_2}{\sqrt{d}}\wedge\tau \r)$ with $\tau = N^{2/(q+4)}$.
\end{lemma}

\begin{proof}[Proof of Lemma \ref{lem:small-ball-2}]
For any $\mf v\in\mc R^d$, we have
\begin{align*}
\l| \dotp{\wt{\mf x}}{\mf v} \r| = \l| \dotp{\wt{\mf x} - \mf x}{\mf v} + \dotp{\mf x}{\mf v} \r|
\geq \l| \dotp{\mf x}{\mf v} \r| -  \l| \dotp{\wt{\mf x} - \mf x}{\mf v}\r|.
\end{align*}
By \eqref{inter-0}, we have
\begin{align}
Pr\l( \l| \dotp{\widetilde{\mathbf{x}}_i}{\mathbf{v}} \r| \geq \delta\|\mathbf{v}\|_2 \r) \geq
Pr\l( \l| \dotp{\mathbf{x}_i}{\mathbf{v}} \r| \geq 2\delta\|\mathbf{v}\|_2 \r)  
- 
Pr\l( \l| \dotp{\widetilde{\mathbf{x}}_i - \mathbf{x}_i}{\mathbf{v}}  \r|\geq\delta\|\mathbf{v}\|_2 \r), \label{inter-1}
\end{align}
Then, 
\begin{align*}
Pr\l(\l| \dotp{\wt{\mf x} - \mf x}{\mf v}\r| \geq\delta\|\mf v\|_2 \r)
=& Pr\l(\l| \dotp{\frac{\sqrt{d}\mathbf{x}}{\|\mathbf{x}\|_2}}{\mf v}\r|\cdot
\l|\l( \frac{\|\mathbf{x}\|_2}{\sqrt{d}}\wedge\tau \r) - \frac{\|\mathbf{x}\|_2}{\sqrt{d}} \r| \geq\delta\|\mf v\|_2 \r)\\
\leq& Pr\l(\|\mf x\|_2\geq \sqrt{d}\tau\r),
\end{align*}
where the inequality follows from the fact that if $\|\mf x\|_2< \sqrt{d}\tau$, then, the truncation does not activate and the difference on the left hand side is equal to 0. By Markov inequality, we can bound the probability as
\begin{multline*}
Pr\l(\|\mf x\|_2\geq \sqrt{d}\tau\r)\leq \frac{\expect{\|\mf x\|_2^2}}{d\tau^2} = \frac{\expect{\mu^2\|\mf B U\|_2^2}}{d\tau^2}
\leq\frac{\expect{\mu^2}}{d\tau^2}\lambda_{\max}(\mf{\Sigma}) \leq \frac{\expect{\mu^2}}{d\sqrt{N}}\lambda_{\max}(\mf{\Sigma}),
\end{multline*}
Thus, when $N\geq\l.\frac{4}{Q^2}\expect{\mu^2}^2\lambda_{\max}(\mf{\Sigma})\r/d^2$, 
\[
Pr\l(\l| \dotp{\wt{\mf x} - \mf x}{\mf v}\r| \geq\delta\|\mf v\|_2 \r)\leq
Pr\l(\|\mf x\|_2\geq \sqrt{d}\tau\r)\leq  Q.
\]
On the other hand, from Lemma \ref{lem:small-ball}, we have $Pr\l( \l| \dotp{\mathbf{x}_i}{\mathbf{v}} \r| \geq 2\delta\|\mathbf{v}\|_2 \r)\geq2Q$, which implies the claim by substituting the above two bounds into \eqref{inter-1}.
\end{proof}

Next, we upgrade the small-ball probability to the lower-tail estimate of the quadratic form. Since the small ball condition in Lemma \ref{lem:small-ball-2} is stronger than that of Lemma \ref{lem:weak-small-ball} in the last section, instead of the VC dimension argument, we can simply invoke the following lemma:

\begin{lemma}[\citep{mendelson2014learning}]\label{lem:mendelson}
Let $\mc H \subseteq \mb S^{d-1}$ and define
\[
\omega_N(\mc H) := \expect{\sup_{\mf h\in \mc H}\frac{1}{\sqrt{N}} \sum_{i=1}^N\varepsilon_i\dotp{\mf x_i}{\mf h}}.
\]
Suppose $Pr(\l| \dotp{\mf x}{\mf h} \r|\geq\delta\|\mf h\|_2)\geq Q,~\forall \mf h\in\mathbb{R}^d$, then, it follows
\[
\inf_{\mf h\in \mc H}\l( \sum_{i=1}^N\dotp{\mf x_i}{\mf h}^2  \r)^{1/2}\geq \delta Q\sqrt{N} - 2\omega_N(\mc H) - \frac{\delta t}{2},
\]
with probability at least $1-ce^{-t^2}$ for any $t>0$.
\end{lemma}

The main Lemma leading to the bound on $r_{\mathcal{Q}}$ is the following,
\begin{lemma}\label{lem:lower-tail-Q}
Suppose $N\geq\frac{4}{Q^2}\frac{\expect{\mu^2}^2\lambda_{\max}(\mf{\Sigma})}{d^2}+\beta^2(\omega(S_{2}(0,r)\cap B_\Psi(0,\rho))+r)^2$, then, with probability at least $1-ce^{-t^2}$
for every $\theta \in S_{2}(\eta\theta_*,r)\cap B_\Psi(\eta\theta_*,\rho)$,
\[
\l| \frac1N\sum_{i=1}^N\dotp{\wt {\mf x}_i}{\theta - \eta \theta_*}^2 \r|^{1/2}\geq \l(\delta Q  - \frac{\delta t}{\sqrt{N}}\r)r - C(\nu_q,\kappa)
\frac{\omega(S_{2}(0,r)\cap B_\Psi(0,\rho)) + r}{\sqrt{N}},
\]
where $C(\nu_q,\kappa)$ is a constant depending only on $\nu_q$ and $\kappa$ in Assumption \ref{assumption:moment}.
\end{lemma}

To prove this lemma, we need the following bound on the truncated multiplier process whose proof is similar to Lemma 5.9 of \citep{goldstein2016structured} via an improved generic chaining technique. For simplicity, we omitted the details of the proof here.

\begin{lemma}[\citep{goldstein2016structured}]\label{lem:trunc-multi}
Suppose $X_i,~i=1,2,\cdots,N$ are i.i.d. sub-Gaussian random vectors in $\mb R^d$ and $q_i,~i=1,2,\cdots,N$ are i.i.d. random variables in $\mb R$ such that 
$\expect{|q_i|^{2(1+\epsilon)}}\leq \tau_1$ for some constant $\tau_1>0$ and some $\epsilon>0$ and $|q_i|\leq \tau_2N^{1/2(1+\epsilon)}$ for some constant $\tau_2>0$. Then, for any compact set $T\in\mb R^d$, we have
\[
Pr\l( \sup_{\mf v\in T}\l| \frac1N\sum_{i=1}^N\varepsilon_i q_i\dotp{X_i}{\mf v} \r|
\geq C(\tau_1,\tau_2)\frac{\omega(T) + D_d(T)}{\sqrt{N}}\beta \r)\leq ce^{-\beta},
\] 
for absolute constant $c>0$, any $\beta \geq2$ and any $N\geq \beta^2(\omega(T) + D_d(T))^2$, where $D_d(T):= \sup_{\mf v\in T}\|\mf v\|_2$ and $C(\tau_1,\tau_2)$ is a constant depending only on $\tau_1$ and $\tau_2$.
\end{lemma}

\begin{proof}[Proof of Lemma \ref{lem:lower-tail-Q}]
First of all, writing $\wt{q}_i:= \l(\frac{\|\mf x_i\|_2}{\sqrt{d}}\r)\wedge\tau$, we have 
$$\wt{\mf x}_i = \wt{q}_i\cdot \frac{\sqrt{d}\mf x_i}{\|\mf x_i\|_2},$$
then, by Lemma \ref{lem:bound-mu}, we have $\expect{\l|\mu/\sqrt{d}\r|^q}\leq \frac{\nu_q^q}{\kappa^{q/2}}$ for some $q>4$, and thus,
\[
\expect{|\wt{q}_i|^4}\leq
\expect{\l(\frac{\|\mf x_i\|_2}{\sqrt{d}}\r)^4} 
= \expect{\l(\frac{|\mu_i|}{\sqrt{d}}\|\mf BU_i\|_2\r)^4}\leq \lambda_{\max}(\mathbf{\Sigma})^{2}\expect{\l(\frac{|\mu|}{\sqrt d}\r)^4}
:= c(\nu_q,\kappa)
\]
Furthermore, it is obvious that $|\wt q_i|\leq N^{1/4}$ and by Lemma \ref{lem:normalization}, we have $\sqrt{d}\mf x_i/\|\mf x_i\|_2$ is a sub-Gaussian random vector.
Thus, 
by Lemma \ref{lem:trunc-multi}, we have
\[
Pr\l(\sup_{\mf v\in S_{2}(0,r)\cap B_\Psi(0,\rho)} 
\l|\frac1N\sum_{i=1}^N\varepsilon_i\dotp{\wt{\mf x}_i}{\theta-\eta\theta_*} \r|
\geq C(\nu_q,\kappa)\beta\frac{\omega(S_{2}(0,r)\cap B_\Psi(0,\rho)) + r}{\sqrt{N}}\r)
\leq ce^{-\beta},
\]
for any $N\geq \beta^2(\omega(S_{2}(0,r)\cap B_\Psi(0,\rho))+r)^2$. Integrating the tails gives
\[
\expect{\sup_{\mf v\in S_{2}(0,r)\cap B_\Psi(0,\rho)} 
\l|\frac1N\sum_{i=1}^N\varepsilon_i\dotp{\wt{\mf x}_i}{\theta-\eta\theta_*} \r|}\leq 
C(\nu_q,\kappa)\frac{\omega(S_{2}(0,r)\cap B_\Psi(0,\rho)) + r}{\sqrt{N}},
\]
Thus, Combining Lemma \ref{lem:small-ball-2} and Lemma \ref{lem:mendelson} with $N\geq\l.\frac{4}{Q^2}\expect{\mu^2}^2\lambda_{\max}(\mf{\Sigma})\r/d^2$
and $\mc H = S_{2}(0,1)\cap B_\Psi(0,\rho/r)$ give
\[
\inf_{\mf v\in S_{2}(0,1)\cap B_\Psi(0,\rho/r)}\l(\frac1N \sum_{i=1}^N\dotp{\wt{\mf x}_i}{\mf v}^2  \r)^{1/2}\geq
 \delta Q - C(\nu_q,\kappa)\frac{\omega(S_{2}(0,r)\cap B_\Psi(0,\rho)) + r}{r\sqrt{N}} - \frac{\delta t}{2\sqrt{N}},
\]
with probability at least  $1-e^{-t^2}$, which implies the claim.
\end{proof}

The following corollary on the bound of $r_{\mathcal{Q}}$ readily follows from the above lemma, and the fact that we take $\Lambda_Q = \delta^2Q^2/4$ in the definition of $r_{\mathcal{Q}}$.
\begin{corollary}\label{lem:bound-Q-2}
Consider $\Omega_{\mathcal{Q}}:=\l\{r>0:  N\geq\frac{4}{Q^2}\frac{\expect{\mu^2}^2\lambda_{\max}(\mf{\Sigma})}{d^2} + (\omega(S_{2}(0,r)\cap B_\Psi(0,\rho) )+ r)^2\r\}$, then, by taking $p_{\mathcal{Q}}=ce^{-t^2}$, we have
\[
r_{\mathcal{Q}}\leq \inf\l\{ r\in\Omega_{\mathcal{Q}}: \l(\frac{\delta Q}{2} - \frac{\delta t + C(\nu,\kappa)}{\sqrt{N}}\r)r\geq C(\nu_q,\kappa)\frac{\omega(S_{2}(0,r)\cap B_\Psi(0,\rho))}{\sqrt{N}} \r\},
\]
where $C(\nu_q,\kappa)$ is a constant depending only on $\nu_q$ and $\kappa$ in Assumption \ref{assumption:moment}.
\end{corollary}

\subsubsection{Bounding the radiuses $r_{\mathcal{M}}$.}
\begin{lemma}\label{lem:bound-M-2}
Suppose $\Omega_{\mathcal{M}}:=\l\{r>0: N\geq (\omega(S_{2}(0,r)\cap B_\Psi(0,\rho)) + r)^2\r\} $, then, by taking $p_{\mathcal{M}} = e^{-\beta}$ and $\Lambda_M:= \delta^2Q^2/64$,
\[
r_{\mathcal{M}}\leq \inf\l\{ r\in\Omega_{\mathcal{M}}:  C(\nu,\kappa,\nu_q)\beta\frac{\omega(S_{2}(0,r)\cap B_\Psi(0,\rho)) + r}{\sqrt{N}}\leq \frac{\delta^2Q^2}{64}r^2 \r\},
\]
where $C(\nu,\kappa,\nu_q)$ is a constant depending only on $\nu,\kappa,\nu_q$ in Assumption \ref{assumption:moment}.
\end{lemma}

\begin{proof}[Proof of Lemma \ref{lem:bound-M-2}]
First if all, by symmetrization inequality, it is enough to bound 
\begin{align*}
&\sup_{\theta\in B_\Psi(\eta\theta_*,\rho)\cap B_2(\eta\theta_*,r)}
\l| \frac1N\sum_{i=1}^N\varepsilon_i(\wt y_i - \dotp{\wt{\mf x}_i}{\eta\theta_*})\dotp{\wt{\mf x}_i}{\theta-\eta\theta_*} \r| \\
&= \sup_{\mf v\in B_\Psi(0,\rho)\cap B_2(0,r)}
\l| \frac1N\sum_{i=1}^N\varepsilon_i(\wt y_i - \dotp{\wt{\mf x}_i}{\eta\theta_*})\dotp{\wt{\mf x}_i}{\mf v} \r|\\
&= \sup_{\mf v\in B_\Psi(0,\rho)\cap B_2(0,r)}
\l| \frac1N\sum_{i=1}^N\varepsilon_i(\wt y_i - \dotp{\wt{\mf x}_i}{\eta\theta_*})\l(\frac{\|\mf x_i\|_2}{\sqrt{d}}\wedge\tau\r)
\dotp{\frac{\sqrt{d}\mf x_i}{\|\mf x_i\|_2}}{\mf v} \r|
\end{align*}
Let $\wt {q}_i:= (\wt y_i - \dotp{\wt{\mf x}_i}{\eta\theta_*})\l(\frac{\|\mf x_i\|_2}{\sqrt{d}}\wedge\tau\r)$.
By Lemma \ref{lem:normalization}, $\frac{\sqrt{d}\mf x_i}{\|\mf x_i\|_2}$ is a sub-Gaussian random vector. To apply Lemma \ref{lem:trunc-multi}, it is enough to check the aforementioned conditions hold for $\wt {q}_i$. Let $\epsilon = \frac{q-4}{8}$. Then, we look at 
$\expect{\l| \wt q_i \r|^{(q+4)/2}} = \expect{\l| \wt q_i \r|^{2(1+\epsilon)}}$,
\begin{align*}
\expect{\l| \wt q_i \r|^{2(1+\epsilon)}}^{1/2(1+\epsilon)} 
\leq& \expect{\l( \l| \wt y_i \r| \l( \frac{\|\mf x_i\|_2}{\sqrt{d}}\wedge\tau \r) \r)^{2(1+\epsilon)}}^{1/2(1+\epsilon)} 
+\expect{\l(\l| \dotp{\wt{\mf x}_i}{\eta\theta_*} \r| \l( \frac{\|\mf x_i\|_2}{\sqrt{d}}\wedge\tau \r)\r)^{2(1+\epsilon)}}^{1/2(1+\epsilon)}\\
\leq&  \expect{\l( \l| \wt y_i \r|  \frac{\|\mf x_i\|_2}{\sqrt{d}} \r)^{2(1+\epsilon)}}^{1/2(1+\epsilon)} 
+\expect{\l(\l| \dotp{\wt{\mf x}_i}{\eta\theta_*} \r| \frac{\|\mf x_i\|_2}{\sqrt{d}}\r)^{2(1+\epsilon)}}^{1/2(1+\epsilon)}
\end{align*}
For the first term on the right hand side, we have
\begin{align*}
\expect{\l( \l| \wt y_i \r|  \frac{\|\mf x_i\|_2}{\sqrt{d}} \r)^{2(1+\epsilon)}}^{1/2(1+\epsilon)}
=& \expect{\l( \l| \wt y_i \r| \cdot \|\mf B U\|_2 \frac{|\mu_i|}{\sqrt{d}} \r)^{2(1+\epsilon)}}^{1/2(1+\epsilon)}\\
\leq&\sqrt{\lambda_{\max}(\mf{\Sigma})}\cdot\expect{\l( \l| \wt y_i \r| \cdot \frac{|\mu_i|}{\sqrt{d}} \r)^{2(1+\epsilon)}}^{1/2(1+\epsilon)}\\
\leq&\sqrt{\lambda_{\max}(\mf{\Sigma})}\cdot \expect{ \l| \wt y_i \r|^{4(1+\epsilon)}}^{1/4(1+\epsilon)}
 \expect{ \l| \frac{\mu_i}{\sqrt{d}} \r|^{4(1+\epsilon)}}^{1/4(1+\epsilon)}\\
 \leq& \sqrt{\lambda_{\max}(\mf{\Sigma})}\cdot \nu_q \frac{\nu}{\kappa^{q/2}}=:c_1(\nu,\kappa,\nu_q)
\end{align*}

by Lemma \ref{lem:bound-mu}. For the second term, we have 
\begin{align*}
&\expect{\l(\l| \dotp{\wt{\mf x}_i}{\eta\theta_*} \r| \frac{\|\mf x_i\|_2}{\sqrt{d}}\r)^{2(1+\epsilon)}}^{1/2(1+\epsilon)}\\
\leq& \expect{\l(\l(\frac{|\mu|}{\sqrt d}\r)^2\|\mf B U\|_2\|\theta_*\|_2|\eta|\l|\dotp{\frac{\sqrt{d}\mf x_i}{\|\mf x_i\|_2}}{\eta\theta_*}\r|\r)^{2(1+\epsilon)}}^{1/2(1+\epsilon)}\\
\leq& \sqrt{\lambda_{\max}(\mf{\mathbf{\Sigma}})}\|\theta_*\|_2|\eta| \expect{\l(\l(\frac{|\mu|}{\sqrt d}\r)^2\l|\dotp{\frac{\sqrt{d}\mf x_i}{\|\mf x_i\|_2}}{\eta\theta_*}\r|\r)^{2(1+\epsilon)}}^{1/2(1+\epsilon)}\\
=&\sqrt{\lambda_{\max}(\mf{\mathbf{\Sigma}})}\|\theta_*\|_2|\eta| \expect{\l(\frac{|\mu|}{\sqrt d}\r)^{4(1+\epsilon)}\l|\dotp{\frac{\sqrt{d}\mf x_i}{\|\mf x_i\|_2}}{\eta\theta_*}\r|^{2(1+\epsilon)}}^{1/2(1+\epsilon)}\\
\leq&\sqrt{\lambda_{\max}(\mf{\mathbf{\Sigma}})}\|\theta_*\|_2|\eta| \expect{\l(\frac{|\mu|}{\sqrt d}\r)^q}^{\frac{2}{q}}
\expect{\l|\dotp{\frac{\sqrt{d}\mf x_i}{\|\mf x_i\|_2}}{\eta\theta_*}\r|^{\frac{4q(1+\epsilon)}{q-4}}}^{\frac{q-4}{4q(1+\epsilon)}}\\
=:& c_2(\nu,\kappa,\nu_q)
\end{align*}
by Lemma \ref{lem:bound-mu} and \ref{lem:normalization}. Furthermore, we have 
\[
\l| \wt{q}_i \r|\leq \l(1+\sqrt{\lambda_{\max}(\mf{\mathbf{\Sigma}})}\|\theta_*\|_2|\eta|\r)N^{2/(q+4)}.
\]
Now, applying Lemma \ref{lem:trunc-multi} gives
\begin{multline*}
\sup_{\theta\in B_\Psi(\eta\theta_*,\rho)\cap B_2(\eta\theta_*,r)}
\l| \frac1N\sum_{i=1}^N\varepsilon_i(\wt y_i - \dotp{\wt{\mf x}_i}{\eta\theta_*})\dotp{\wt{\mf x}_i}{\theta-\eta\theta_*} \r| \\
\leq C(\nu,\kappa,\nu_q)\beta\frac{ \omega( B_\Psi(0,\rho)\cap B_2(0,r))+ r}{\sqrt{N}},
\end{multline*}
with probability at least $1-ce^{-\beta}$, for some absolute constant $c>0$.
Thus, by taking $p_{\mathcal{M}} = e^{-\beta}$ and $\Lambda_M:= \delta^2Q^2/64$ we obtain the claim.
\end{proof}

\subsubsection{Bounding the radius $r_{\mc V}$}

\begin{lemma}\label{lem:bound-V-2}
The following bound holds
\[
r_{\mc V}\leq \frac{64C(\nu,\kappa,\nu_q)}{\delta^2Q^2\sqrt N},
\]
where $C(\nu,\kappa,\nu_q)$ is a constant depending only on $\nu,\kappa,\nu_q$ in Assumption \ref{assumption:moment}.
\end{lemma}

\begin{proof}[Proof of Lemma \ref{lem:bound-V-2}]
First, we have, for any $\theta \in B_\Psi(\eta\theta_*,\rho)\cap B_2(\eta\theta_*,r)$, we have
\begin{multline*}
\mc{V}_{\theta-\eta\theta_*} = \expect{\l(\wt y_i - \dotp{\wt{\mf x}_i}{\eta\theta_*}\r)\dotp{\wt{\mf{x}}_i}{\theta-\eta\theta_*}}
= \expect{\dotp{\wt y_i \wt{\mf x}_i - y_i\mf x_i}{\theta - \eta\theta_*}} + \expect{\dotp{y_i\mf x_i}{\theta - \eta\theta_*}}  \\
- \expect{\dotp{\mf x_i}{\eta\theta_*}\dotp{\mf{x}_i}{\theta-\eta\theta_*}}
- \expect{\dotp{\wt{\mf x}_i}{\eta\theta_*}\dotp{\wt{\mf{x}}_i}{\theta-\eta\theta_*} - \dotp{\mf x_i}{\eta\theta_*}\dotp{\mf{x}_i}{\theta-\eta\theta_*}}.
\end{multline*}
Thus,
\begin{multline*}
|\mc{V}_{\theta-\eta\theta_*}|
\leq   \underbrace{\l|\expect{\dotp{\wt y_i \wt{\mf x}_i - y_i\mf x_i}{\theta - \eta\theta_*}}\r|}_{\text{(I)}}
+ \underbrace{\l|\expect{\dotp{y_i\mf x_i}{\theta - \eta\theta_*}}  
- \expect{\dotp{\mf x_i}{\eta\theta_*}\dotp{\mf{x}_i}{\theta-\eta\theta_*}}\r|}_{\text{(II)}}\\
+ \underbrace{\l|\expect{\dotp{\wt{\mf x}_i}{\eta\theta_*}\dotp{\wt{\mf{x}}_i}{\theta-\eta\theta_*} - \dotp{\mf x_i}{\eta\theta_*}\dotp{\mf{x}_i}{\theta-\eta\theta_*}}\r|}_{\text{(III)}}
\end{multline*}
For the first term, we have by definition of $\wt y_i$ and $\wt{\mf x}_i$ in \eqref{eq:trunc2},
\begin{align*}
\text{(I)}
\leq& \expect{\l|\dotp{\wt y_i \wt{\mf x}_i - y_i\mf x_i}{\theta - \eta\theta_*}\r|} \\
=&  \expect{\l| \sign(y_i)(|y_i|\wedge\tau)\l( \frac{\|\mf x_i\|_2}{\sqrt d}\wedge\tau \r) - y\frac{\|\mf x\|_2}{\sqrt d} \r| \cdot
\l| \dotp{\frac{\sqrt d \mf x_i}{\|\mf x_i\|_2}}{\theta-\eta\theta_*} \r|}\\
\leq& \expect{
\l| y_i\frac{\|\mf x_i\|_2}{\sqrt d} \r|\cdot 1_{\l\{\|\mf x_i\|_2/\sqrt d\geq\tau\r\} \cup\l\{ |y_i|\geq\tau \r\} }
\cdot \l| \dotp{\frac{\sqrt d \mf x_i}{\|\mf x_i\|_2}}{\theta-\eta\theta_*} \r|}\\
=& \expect{\l| y_i\dotp{\mf x_i}{\theta-\eta\theta_*} \r| \cdot 1_{\l\{\|\mf x_i\|_2/\sqrt d\geq\tau\r\} \cup\l\{ |y_i|\geq\tau \r\} }}\\
\leq &\expect{\l| y_i\dotp{\mf x_i}{\theta-\eta\theta_*} \r|^2}^{1/2}Pr\l( \l\{\|\mf x_i\|_2/\sqrt d\geq\tau\r\} \cup\l\{ |y_i|\geq\tau \r\}  \r)^{1/2},
\end{align*}
where the second inequality follows from the fact that on the set $\l\{\|\mf x_i\|_2/\sqrt d<\tau\r\} \cap\l\{ |y_i|<\tau \r\}$ the expression is 0, and the last inequality follows from Holder's inequality. By Assumption \ref{assumption:moment}, let $\epsilon = (q-4)/4$,
\begin{align*}
\expect{\l| y_i\dotp{\mf x_i}{\theta-\eta\theta_*} \r|^2}^{1/2} 
\leq& \expect{|y_i|^4}^{1/4}\expect{|\dotp{\mf x_i}{\theta-\eta\theta_*} |^4}^{1/4}\\
=& \expect{\l|y_i\r|^4}^{1/4} \expect{ \l|\frac{\mu}{\sqrt{d}}\r|^4 \l|\dotp{\frac{\sqrt{d}\mf x_i}{\|\mf x_i\|}}{\theta-\eta\theta_*} \r|^4}^{1/4}\\
\leq& \expect{\l|y_i\r|^4}^{1/4} 
\expect{\l|\frac{\mu}{\sqrt{d}}\r|^{4(1+\epsilon)}}^{1/4(1+\epsilon)}
\expect{\l|\dotp{\frac{\sqrt{d}\mf x_i}{\|\mf x_i\|}}{\theta-\eta\theta_*} \r|^{\frac{4(1+\epsilon)}{\epsilon}}}^{\frac{\epsilon}{4(1+\epsilon)}}\\
\leq& \expect{\l|y_i\r|^4}^{1/4} 
\expect{\l|\frac{\mu}{\sqrt{d}}\r|^{4(1+\epsilon)}}^{1/4(1+\epsilon)}\l\| \frac{\sqrt{d}\mf x_i}{\|\mf x_i\|} \r\|_{\psi_2}\sqrt{\frac{4(1+\epsilon)}{\epsilon}}
\|\theta-\eta\theta_*\|_2,
\end{align*}
where the second from the last inequality follows from Holder's inequality and the last inequality follows from Lemma \ref{lem:normalization}. Furthermore, we have
\begin{align*}
Pr\l( \l\{\|\mf x_i\|_2/\sqrt d\geq\tau\r\} \cup\l\{ |y_i|\geq\tau \r\}  \r)^{1/2}
\leq& \l( Pr\l( \|\mf x_i\|_2/\sqrt d\geq\tau \r)  +  Pr(|y_i|\geq\tau)   \r)^{1/2}\\
\leq& \l( \frac{\expect{\|\mf x_i\|_2^q}}{d^{q/2}\tau^q}  +  \frac{\expect{|y_i|^q}}{\tau^q}   \r)^{1/2}\\
\leq& \lambda_{\max}(\mathbf{\Sigma})^{q/4}\expect{\l|\frac{\mu_i}{\sqrt{d}}\r|^{q}}^{1/2}\frac{1}{\sqrt N}+ \frac{\expect{|y_i|^q}^{1/2}}{\sqrt N}  
\end{align*}
Thus, it follows
\begin{equation}\label{eq:inter-11}
\text{(I)} \leq \frac{C_1(\nu,\kappa,\nu_q)}{\sqrt N}\|\theta - \eta\theta_*\|_2.
\end{equation}

Now, we consider the term (II). Let $\mf x_0 = \mf{\Sigma}^{-1/2}\mf x_i$, then, we have
\begin{align*}
\expect{y_i\dotp{\mf x_i}{\theta - \eta\theta_*}}
=& \expect{f(\dotp{\mf x_i}{\theta_*}, \xi_i)\dotp{\mf x_i}{\theta - \eta\theta_*}}\\
=& \expect{f\l(\dotp{\mf{\Sigma}^{1/2}\mf x_0}{\theta_*}, \xi_i\r)\dotp{\mf{\Sigma}^{1/2}\mf x_0}{\theta - \eta\theta_*}}\\
=& \expect{f\l(\dotp{\mf x_0}{\mf{\Sigma}^{1/2}\theta_*}, \xi_i\r)\dotp{\mf x_0}{\mf{\Sigma}^{1/2}(\theta - \eta\theta_*)}}\\
=& \expect{\expect{  \l. f\l(\dotp{\mf x_0}{\mf{\Sigma}^{1/2}\theta_*}, \xi_i\r)\dotp{\mf x_0}{\mf{\Sigma}^{1/2}(\theta - \eta\theta_*)} \r|~\dotp{\mf x_0}{\mf{\Sigma}^{1/2}\theta_*}, \xi_i}}
\end{align*}
Note that $\mf x_0 \sim\mathcal{E}_d(0,\mf I_{d\times d},F_\mu)$ and $\l\| \mf{\Sigma}^{1/2}\theta_* \r\|_2=1$, by Lemma \ref{lem:elliptical}, we have 
\begin{align*}
\expect{y_i\dotp{\mf x_i}{\theta - \eta\theta_*}} =& 
\expect{f\l(\dotp{\mf x_0}{\mf{\Sigma}^{1/2}\theta_*}, \xi_i\r)\dotp{\mf x_0}{\mf{\Sigma}^{1/2}\theta_*}\dotp{\mf{\Sigma}^{1/2}\theta_*}
{\mf{\Sigma}^{1/2}(\theta-\eta\theta_*)}}\\
=&\expect{f\l(\dotp{\mf x_i}{\theta_*}, \xi_i\r)\dotp{\mf x_i}{\theta_*}\dotp{\mf{\Sigma}\theta_*}
{\theta-\eta\theta_*}}
= \eta \dotp{\mf{\Sigma}\theta_*}{\theta-\eta\theta_*},
\end{align*}
where the scaling constant $\eta$ is defined in \eqref{eq:scaling-const-2}. On the other hand, it is obvious that
\[
\expect{\dotp{\mf x_i}{\eta\theta_*}\dotp{\mf{x}_i}{\theta-\eta\theta_*}} = \eta \dotp{\mf{\Sigma}\theta_*}{\theta-\eta\theta_*},
\]
which implies 
\begin{equation}\label{eq:inter-12}
\text{(II)} = 0.
\end{equation}

Finally, we have
\begin{align*}
\text{(III)} =& 
\l|\expect{\dotp{\wt{\mf x}_i}{\eta\theta_*}\dotp{\wt{\mf{x}}_i}{\theta-\eta\theta_*} - \dotp{\mf x_i}{\eta\theta_*}\dotp{\mf{x}_i}{\theta-\eta\theta_*}}\r|\\
\leq& \|\eta\theta_*\|_2\|\theta- \eta\theta_*\|_2\l\| \expect{\mf x_i\mf x_i^T - \wt{\mf x}_i\wt{\mf x}_i^T} \r\|_*\\
\leq& \|\eta\theta_*\|_2\|\theta- \eta\theta_*\|_2\l\| \expect{\mf x_i\mf x_i^T \cdot 1_{\l\{ \|\mf x_i\|_2/\sqrt{d}\geq \tau \r\}}} \r\|_*\\
=& \|\eta\theta_*\|_2\|\theta- \eta\theta_*\|_2\cdot \sup_{\|\mf v\|_2 = 1}\expect{\dotp{\mf x_i}{\mf v}^2 \cdot 1_{\l\{ \|\mf x_i\|_2/\sqrt{d}\geq \tau \r\}}} \\
\leq&  \|\eta\theta_*\|_2\|\theta- \eta\theta_*\|_2\cdot \sup_{\|\mf v\|_2 = 1}\expect{\dotp{\mf x_i}{\mf v}^4}^{1/2} Pr\l( \|\mf x_i\|_2/\sqrt{d}\geq \tau \r)^{1/2},
\end{align*}
where the first inequality follows from Cauchy-Schwarz, the second inequality follows from the fact that the expression is equal to 0 when 
$\|\mf x_i\|_2/\sqrt{d}< \tau$ and the final inequality follows from Holder's inequality. By Assumption \ref{assumption:moment}, we have
$\sup_{\|\mf v\|_2 = 1}\expect{\dotp{\mf x_i}{\mf v}^4}^{1/2}\leq \nu^{1/2}$. Also, $\|\eta\theta_*\|_2 = |\eta|\|\mf{\Sigma}^{-1/2}\mf{\Sigma}^{1/2}\theta_*\|_2
\leq \frac{|\eta|}{\sqrt{\kappa}}$.
Furthermore, by Markov inequality,
\[
Pr\l( \|\mf x_i\|_2/\sqrt{d}\geq \tau \r)^{1/2}\leq \l( \frac{\expect{\|\mf x_i\|_2^4}}{d^2\tau^4} \r)^{1/2}
= \l( \frac{\expect{\mu_i^4\|\mf B U_i\|_2^4}}{d^2\tau^4} \r)^{1/2}
\leq \frac{\lambda_{\max}(\mf{\Sigma})}{\sqrt{N}}\cdot \expect{\l| \frac{\mu^4}{d^2} \r|}^{1/2}.
\]
By Lemma \ref{lem:bound-mu} we have $\expect{\l| \frac{\mu^4}{d^2} \r|}^{1/2}$ is bounded by some constant, thus, 
\[
\text{(III)} \leq \frac{C_2(\nu,\kappa,\nu_q)}{\sqrt{N}}\|\theta- \eta\theta_*\|_2
\]
by some constant $C>0$. Overall, combining the above bound with \eqref{eq:inter-11} and \eqref{eq:inter-12} gives
\[
|\mc{V}_{\theta-\eta\theta_*}| \leq \frac{C(\nu,\kappa,\nu_q)}{\sqrt{N}}\|\theta- \eta\theta_*\|_2.
\]
By definition of $r_{\mc V}$ and the setting that $\Lambda_{\mc V}=\frac{\delta^2Q^2}{64}$, we let 
$$\frac{C(\nu,\kappa,\nu_q)r}{\sqrt{N}} = \frac{\delta^2Q^2}{64}r^2
\Rightarrow r =  \frac{64C(\nu,\kappa,\nu_q)}{\delta^2Q^2\sqrt N},$$
thus, the radius  $r_{\mc V}$ must be bounded above by this value.
\end{proof}

\subsection{Applying bounds to low-rank matrix recovery}
In this Section, we show that by combining Lemma \ref{lem:bound-radiuses-2} with Theorem \ref{thm:master}, we can obtain tight sample and error rates in the low-rank recovery problems.

We analyze the scenario where $\theta_*\in\mathbb{R}^{m\times n}$, $\Psi(\cdot) = \|\cdot\|_*$, the nuclear norm of the matrix, and it is close to a rank $s$ matrix $\theta_0$. We use $\|\cdot\|_2$ to denote the Frobenius norm of a matrix.
\begin{lemma}\label{lem:sub-diff-matrix}
Suppose $\|\eta\theta_*-\theta_0\|_*\leq\rho/16$ where $\theta_0$ is a rank $s$ matrix. Then, under the condition $\rho\geq 16r(\rho)\sqrt{s}$, 
$\Delta(\eta\theta_*,\rho)\geq\frac{3}{4}\rho$.
\end{lemma}
Similar types of bounds characterizing the set of sub-differentials  also appears in Lemma 4.4 of \citep{lecue2016regularization} and the proof is rather standard. For completeness, we provide a proof which uses the following classical lemma stating that the nuclear norm $\|\cdot\|_*$, similar to the $\|\cdot\|_1$, is also decomposable. 
\begin{lemma}[\citep{watson1992characterization}]\label{lem:decompose-matrix}
Let $\mf V\in\mb R^{m\times n}$ such that $\mf V = P_I \mf V P_J$ for orthogonal projections $P_I$ and $P_J$ on to subspaces $I\subseteq\mb R^m$ and $J\subseteq R^n$, respectively. Then, for every $\mf W\in \mb R^{m\times n}$, there exists a matrix $\mf Z \in R^{m\times n}$ such that 
$\|\mf Z\| = 1$ and 
\[
\dotp{\mf Z}{\mf V} = \|\mf V\|_*,~\dotp{\mf Z}{P_{I^{\perp}} \mf W P_{J^{\perp}}} = \|P_{I^{\perp}} \mf W P_{J^{\perp}}\|_*,~
\dotp{\mf Z}{P_{I} \mf W P_{J^{\perp}}} = 0,~\dotp{\mf Z}{P_{I^{\perp}} \mf W P_{J}} = 0.
\]
\end{lemma}
\begin{proof}[Proof of Lemma \ref{lem:sub-diff-matrix}]
Recall that 
\[
\Delta(\eta\theta_*,\rho) := \inf_{\theta\in B_{2}(\eta\theta_*,r)\cap S_\Psi(\eta\theta_*,\rho)}
~\sup_{\mf{z}\in\Gamma_\Psi(\eta\theta_*,\rho)}\dotp{\mf z}{\theta-\eta\theta_*}
\]
Suppose $I, J$ are subspaces of $\mb R^m$ and $\mb R^n$ such that $\theta_0 = P_I \theta_0 P_J$.
Since $\|\eta\theta_*-\theta_0\|_*\leq\rho/16$, the set of  subdifferentials of $\theta_0$ are contained in $\Gamma_\Psi(\eta\theta_*,\rho)$. By Lemma \ref{lem:decompose-matrix}, there exists $\mf z\in \Gamma_\Psi(\eta\theta_*,\rho)$ such that for any $\mf W \in B_{2}(0,r)\cap S_\Psi(0,\rho)$, 
\[
\dotp{\mf z}{\theta_0} = \|\theta_0\|_*,~\dotp{\mf z}{P_{I^{\perp}} \mf W P_{J^{\perp}}} = \|P_{I^{\perp}} \mf W P_{J^{\perp}}\|_*,~
\dotp{\mf z}{P_{I} \mf W P_{J^{\perp}}} = 0,~\dotp{\mf z}{P_{I^{\perp}} \mf W P_{J}} = 0.
\]
In particular, this implies,
\begin{align*}
\dotp{\mf z}{\mf W} =&  \dotp{\mf z}{P_{I^{\perp}} \mf W P_{J^{\perp}}} + \dotp{\mf z}{P_{I^{\perp}} \mf W P_{J}} + \dotp{\mf z}{P_{I} \mf W P_{J^{\perp}}}
+ \dotp{\mf z}{P_{I} \mf W P_{J}}\\
\geq& \|P_{I^{\perp}} \mf W P_{J^{\perp}}\|_* - \|P_{I^{\perp}} \mf W P_{J}\|_* -  \|P_{I} \mf W P_{J^{\perp}}\|_* -  \|P_{I} \mf W P_{J}\|_*\\
\geq& \|\mf W\|_* - \|P_{I^{\perp}} \mf W P_{J}\|_* -  \|P_{I} \mf W P_{J^{\perp}}\|_* -  2\|P_{I} \mf W P_{J}\|_*.
\end{align*}
Let $\l\{\mathbf{\Sigma}_i(\mf W)\r\}_{i=1}^{\min\{m,n\}}$ we the sequence inf singular values of $\mf W$ in decreasing order.
$$\|P_{I^{\perp}} \mf W P_{J}\|_*\leq \sum_{i=1}^s\mathbf{\Sigma}_i(\mf W)\leq \sqrt{s}\|\mf W\|_2\leq \sqrt{s}r(\rho).$$
Same bounds hold for $\|P_{I} \mf W P_{J^{\perp}}\|_*$ and  $\|P_{I} \mf W P_{J}\|_*$. Thus, we get for any $\mf W \in B_{2}(0,r)\cap S_\Psi(0,\rho)$, there exists $\mf z\in \Gamma_\Psi(\eta\theta_*,\rho)$ such that
\[
\dotp{\mf z}{\mf W} \geq \rho - 4r(\rho)\sqrt{s},
\]
which is greater than $\frac{3}{4}\rho$ when $\rho\geq 16r(\rho)\sqrt{s}$.
\end{proof}

\begin{proof}[Proof of Theorem \ref{thm:low-rank}]
First, note that the Gaussian mean width can be bounded as follows
\begin{multline*}
\omega( B_\Psi(0,\rho)\cap B_2(0,r))\leq \min\l\{ \expect{\sup_{\mf v\in B_\Psi(0,\rho)} \dotp{\mf g}{\mf v}  },
~ \expect{\sup_{\mf v\in B_2(0,r)} \dotp{\mf g}{\mf v}  } \r\}\\
\leq C_0 \min\l\{ \rho \sqrt{m+n},~r\sqrt{mn} \r\},
\end{multline*}
for some absolute constant $C_0>0$.
By Corollary \ref{lem:bound-Q-2}, we have 
$$\Omega_{\mathcal{Q}}=\l\{r>0, N\geq\frac{4}{Q^2}\frac{\expect{\mu^2}^2\lambda_{\max}(\mf{\Sigma})}{d^2} + C_0^2 \min\l\{ \rho \sqrt{m+n},~r\sqrt{mn} \r\}^2\r\},$$
and then
\[
r_{\mathcal{Q}}\leq \inf\l\{ r\in\Omega_{\mathcal{Q}}: \l(\frac{\delta Q}{2} - \frac{\delta t + C(\nu,\kappa,\nu_q)}{\sqrt{N}}\r)r\geq \frac{C(\nu,\kappa,\nu_q)\min\l\{ \rho \sqrt{m+n},~r\sqrt{mn} \r\}}{\sqrt{N}} \r\}.
\]
Also, by Lemma \ref{lem:bound-M-2}, $\Omega_{\mathcal{M}} = \l\{r>0, N\geq C_0^2 \min\l\{ \rho \sqrt{m+n},~r\sqrt{mn} \r\}^2\r\}$,
\[
 r_{\mathcal{M}}\leq \inf\l\{ r\in\Omega_{\mathcal{M}}:  C(\nu,\kappa,\nu_q)\beta\frac{\min\l\{ \rho \sqrt{m+n},~r\sqrt{mn} \r\} + r}{\sqrt{N}}\leq \frac{\delta^2Q^2}{64}r^2 \r\}
\]
Furthermore, by Lemma \ref{lem:bound-V-2},we have
\[
r_{\mc V}\leq \frac{64C(\nu,\kappa,\nu_q)}{\delta^2Q^2\sqrt N}.
\]
Since the final radius bound $r(\rho)= \max\{ r_{\mathcal{Q}},r_{\mathcal{M}}, r_{\mc V} \}$, and
by Lemma \ref{lem:sparse-eq}, $\rho\geq 16r(\rho)\sqrt{s}$ implies the sparsity condition $\Delta(\eta\theta_*,\rho)\geq 3\rho/4$. Thus, the aforementioned sparsity condition holds for any 
\[
\rho\geq C_1(\nu,\kappa,\nu_q)\beta\frac{s\sqrt{m+n}}{\sqrt{N}}.
\]
In particular, this implies $r(\rho)\leq C_1(\nu,\kappa,\nu_q)\frac{\beta\sqrt{s(m+n)}}{\sqrt{N}}$, and by Theorem \ref{thm:master}, we need to choose $\lambda = c_1(\nu,\kappa,\nu_q)\frac{\beta\sqrt{m+n}}{\sqrt{N}}$.
\end{proof}

\end{document}